\documentclass[french,11pt,a4paper]{smfart}

\usepackage[utf8]{inputenc}
\usepackage[T1]{fontenc}
\usepackage[french]{babel}
\usepackage{lmodern}
\usepackage{smfthm}
\usepackage[headings]{fullpage}
\usepackage{amssymb}

\let\cal\mathcal

\let\hat\widehat
\let\tilde\widetilde
\let\phi\varphi

\let\epsilon\varepsilon

\def\Q{{\bf Q}} 
\def\Z{{\bf Z}}
\def\C{{\bf C}}
\def\N{{\bf N}}
\def\R{{\bf R}}
\def\A{{\bf A}}
\def\E{{\bf E}}
\def\B{{\bf B}}
\def\O{{\cal O}}
\def\G{{\cal G}}

\def\Ker{{\mathrm{ker}}}

\def\k{{\bf k}}

\def\D{{\bf D}}
\def\M{{\bf M}}

\def\Mat{{\mathrm{Mat}}}
\def\Fil{{\mathrm{Fil}}}

\def\Lie{{\mathrm{Lie}}}
\def\pa{{\mathrm{pa}}}

\def\coker{{\mathrm{coker}}}
\def\Mod{{\mathrm{Mod}}}
\newcommand{\pscal}[1]{\langle #1 \rangle}

\def\Zp{{\Z_p}}
\def\Qp{{\Q_p}}

\def\At{{\tilde{\bf{A}}}}
\def\Atplus{{\tilde{\bf{A}}^+}}
\def\Bt{{\tilde{\bf{B}}}}
\def\Btplus{{\tilde{\bf{B}}^+}}
\def\Et{{\tilde{\bf{E}}}}
\def\Etplus{{\tilde{\bf{E}}^+}}

\def\Bdr{{\bf{B}_{\mathrm{dR}}}}
\def\Bdrplus{{\bf{B}_{\mathrm{dR}}^+}}

\def\Gal{{\mathrm{Gal}}}

\def\GL{{\mathrm{GL}}}

\def\k{{\mathbf{k}}}
\def\la{{\mathrm{la}}}

\NumberTheoremsIn{subsection}

\author{Léo Poyeton}
\address{UMPA de l'ENS de Lyon \\
UMR 5669 du CNRS}
\email{leo.poyeton@ens-lyon.fr}
\urladdr{perso.ens-lyon.fr/leo.poyeton/}

\date{\today}

\title{$(\phi,\tau)$-modules différentiels et représentations potentiellement semi-stables}

\begin{document}

\begin{abstract}
Soit $K$ un corps $p$-adique et soit $V$ une représentation $p$-adique de $\G_K = \Gal(\overline{K}/K)$. La surconvergence des $(\phi,\tau)$-modules nous permet d'attacher à $V$ un $\phi$-module différentiel à connexion $D_{\tau,\mathrm{rig}}^\dagger(V)$ sur l'anneau de Robba $\B_{\tau,\mathrm{rig},K}^\dagger$. On montre dans cet article comment retrouver les invariants $D_{\mathrm{cris}}(V)$ et $D_{\mathrm{st}}(V)$ à partir de $D_{\tau,\mathrm{rig}}^\dagger(V)$, et comment caractériser les représentations potentiellement semi-stables, ainsi que celles de $E$-hauteur finie, à partir de la connexion.
\end{abstract}

\begin{altabstract}
Let $K$ be a $p$-adic field and let $V$ be a $p$-adic representation of $\G_K=\Gal(\overline{K}/K)$. The overconvergence of $(\phi,\tau)$-modules allows us to attach to $V$ a differential $\phi$-module $D_{\tau,\mathrm{rig}}^\dagger(V)$ on the Robba ring $\B_{\tau,\mathrm{rig},K}^\dagger$ that comes equipped with a connection. We show in this paper how to recover the invariants $D_{\mathrm{cris}}(V)$ and $D_{\mathrm{st}}(V)$ from $D_{\tau,\mathrm{rig}}^\dagger(V)$, and give a characterization of both potentially semi-stable representations of $\G_K$ and finite $E$-height representations in terms of the connection operator.
\end{altabstract}

\maketitle

\setcounter{tocdepth}{2}
\tableofcontents

\setlength{\baselineskip}{18pt}

\section*{Introduction}\label{intro} 
Soit $K$ un corps $p$-adique et soit $\overline{K}$ une clôture algébrique de $K$. Pour étudier les représentations $p$-adiques de $\G_K = \Gal(\overline{K}/K)$, Fontaine a construit dans \cite{Fon90} une équivalence de catégories $V \mapsto D(V)$ entre la catégorie des représentations $p$-adiques de $\G_K$ et celle des $(\phi,\Gamma_K)$-modules étales. Un $(\phi,\Gamma_K)$-module est un espace vectoriel de dimension finie sur un corps local $\B_K$ de dimension $2$, muni d'actions semi-linéaires compatibles d'un Frobenius $\phi$ et de $\Gamma_K$, et on dit qu'il est étale si le Frobenius est de pente $0$. On peut en fait identifier $\B_K$ à l'anneau des séries formelles $\sum_{k \in \Z}a_kT^k$, où la suite $(a_k)$ est une suite bornée d'éléments d'une certaine extension non ramifiée $F'$ de $F$, et telle que $\lim\limits_{k \to -\infty}a_k = 0$. L'extension cyclotomique $K(\zeta_{p^\infty}) = \bigcup_{n \geq 1}K(\zeta_{p^n})$, engendrée par les racines $p^n$-ièmes de l'unité, joue dans cette construction un rôle fondamental. La construction de ces $(\phi,\Gamma_K)$-modules étales repose en effet sur le fait que les représentations $p$-adiques de $H_K=\Gal(\overline{K}/K(\zeta_{p^\infty}))$ sont classifiées par les $\phi$-modules étales sur $\B_K$, et c'est en rajoutant une action de $\Gamma_K = \Gal(K(\zeta_{p^\infty})/K)$ qu'on obtient l'équivalence de catégories $V \mapsto D(V)$ énoncée précédemment. 

Remarquant le rôle particulier joué par les extensions de Kummer vis-à-vis des représentations semi-stables dans les travaux de Breuil \cite{breuil1998schemas} et Kisin \cite{KisinFiso}, Caruso a introduit dans \cite{Car12} une variante des $(\phi,\Gamma_K)$-modules de Fontaine, les $(\phi,\tau)$-modules, en remplaçant l'extension cyclotomique dans la théorie de Fontaine par une extension de Kummer $K_\infty=\bigcup_{n \geq 1}K_n$ où $K_n = K(\pi_n)$ et $(\pi_n)$ est une suite compatible de racines $p^n$-ièmes d'une uniformisante $\pi$ de $K$ fixée. Comme dans le cas cyclotomique, les représentations $p$-adiques de $H_{\tau,K}=\Gal(\overline{K}/K_\infty)$ sont classifiées par les $\phi$-modules étales sur $\B_{\tau,K}$, un corps local de dimension $2$ qu'on peut identifier à l'anneau de séries formelles $\sum_{k \in \Z}a_kT^k$, où la suite $(a_k)$ est une suite bornée d'éléments de $F$, et telle que $\lim\limits_{k \to -\infty}a_k = 0$. Cependant, la comparaison avec la stratégie de Fontaine s'arrête ici, puisque l'extension $K_\infty/K$ n'est pas Galoisienne et on n'a donc pas de groupe à faire agir. L'idée de Caruso est alors d'ajouter une action d'un élément bien choisi $\tau$ de $\G_K$, pas directement sur le $\phi$-module mais après avoir tensorisé au-dessus de $\B_{\tau,K}$ par une certaine $\B_{\tau,K}$-algèbre $\Bt_L$ munie d'une action de $\Gal(K_\infty^{\Gal}/K)$. Caruso montre alors que la catégorie de ces $(\phi,\tau)$-modules étales est équivalente à celle des représentations $p$-adiques de $\G_K$. 

Les anneaux $\B_K$ et $\B_{\tau,K}$ n'ont malheureusement pas d'interprétation analytique satisfaisante, ce qui les rend difficiles à manipuler, mais ils contiennent les anneaux respectifs $\B_K^\dagger$ et $\B_{\tau,K}^\dagger$ des séries surconvergentes, c'est-à-dire qui convergent et sont bornées sur une couronne bordée par le cercle unité. Un des résultats fondamentaux concernant les $(\phi,\Gamma_K)$-modules étales est le théorème principal de \cite{colmez1999representations} qui montre que tout $(\phi,\Gamma_K)$-module étale provient par extension des scalaires d'un $(\phi,\Gamma_K)$-module surconvergent défini sur $\B_K^\dagger$. Le pendant de ce théorème pour les $(\phi,\tau)$-modules a été démontré récemment (voir \cite{gao2016loose} lorsque $k$ est fini et \cite{GP18} pour le cas général), c'est-à-dire que les $(\phi,\tau)$-modules étales sur $\B_{\tau,K}$ proviennent par extension des scalaires d'un $(\phi,\tau)$-module surconvergent sur $\B_{\tau,K}^\dagger$.

La surconvergence des $(\phi,\Gamma_K)$-modules cyclotomiques a notamment permis à Berger \cite{Ber02} d'associer à toute représentation $p$-adique $V$ un $(\phi,\Gamma_K)$-module sur l'anneau de Robba $\B_{\mathrm{rig},K}^\dagger$ constitué des séries $A(T)=\sum_{k \in \Z}a_kT^k$ où la suite $(a_k)$ est une suite d'éléments de $F'$ telle que la série $A(T)$ converge sur une couronne bordée par le cercle unité (on ne suppose plus que la suite $(a_k)$ est bornée), et d'obtenir et de retrouver les invariants $D_{\mathrm{cris}}$ et $D_{\mathrm{st}}$ associés à $V$ via l'étude de ce module. De plus, l'action infinitésimale de $\Gamma_K$ permet de munir ce $(\phi,\Gamma_K)$-module d'une connexion $\nabla$, dont l'étude poussée a permis à Berger de faire le lien avec les $(\phi,N)$-modules filtrés et de montrer comment caractériser les représentations potentiellement semi-stables à partir de la connexion \cite{Ber08}. 

On va dans cet article s'intéresser à des conséquences de la surconvergence des $(\phi,\tau)$-modules, en montrant notamment qu'on peut adapter les constructions de Berger dans \cite{Ber02} et \cite{Ber08} aux $(\phi,\tau)$-modules. Partant d'une représentation $p$-adique $V$ de $\G_K$, la surconvergence des $(\phi,\tau)$-modules permet, en tensorisant le $(\phi,\tau)$-module surconvergent avec l'anneau de Robba $\B_{\tau,\mathrm{rig},K}^\dagger$, d'associer à $V$ un $(\phi,\tau)$-module $D_{\tau,\mathrm{rig}}^\dagger(V)$, dont l'étude permet de retrouver les invariants $D_{\mathrm{cris}}$ et $D_{\mathrm{st}}$ associés à la représentation $V$. On pose $\B_{\tau,\log,K}^\dagger = \B_{\tau,\mathrm{rig},K}^\dagger[\log(T)]$ et $D_{\tau,\log}^\dagger(V)= \B_{\tau,\log,K}^\dagger \otimes_{B_{\tau,\mathrm{rig},K}^\dagger}D_{\tau,\mathrm{rig}}^\dagger(V)$. Si $\lambda$ désigne l'élément défini dans \cite[1.1.1]{KisinFiso} et si prendre les invariants sous $\tau=1$ dans ce qui suit signifie que les éléments sont invariants sous l'action de $\tau$ une fois qu'on a tensorisé par des $\B_{\tau,K}^\dagger$-algèbres sur lesquelles l'action de $\tau$ est bien définie, on a~:

\begin{enonce*}{Théorème A}
Si $V$ est une représentation $p$-adique de $\G_K$, alors 
$$D_{\mathrm{st}}(V) = (D_{\tau,\log}^{\dagger}(V)[1/\lambda])^{\tau=1} \textrm{ et } D_{\mathrm{cris}}(V) = (D_{\tau,\mathrm{rig}}^\dagger(V)[1/\lambda])^{\tau=1}.$$
\end{enonce*}

On dispose également d'isomorphismes de comparaison faisant le lien entre $D_{\mathrm{cris}}$ ou $D_{\mathrm{st}}$ et $D_{\tau,\mathrm{rig}}^\dagger(V)$~: 

\begin{enonce*}{Théorème B}
Soit $V$ une représentation $p$-adique de $\G_K$.
\begin{enumerate}
\item si $V$ est semi-stable, alors 
$$\B_{\tau,\log,K}^{\dagger}[1/\lambda] \otimes_{\B_{\tau,K}^{\dagger}}D_{\tau}^{\dagger}(V) =  \B_{\tau,\log,K}^{\dagger}[1/\lambda] \otimes_F D_{\mathrm{st}}(V)$$
\item si $V$ est cristalline, alors 
$$\B_{\tau,\mathrm{rig},K}^{\dagger}[1/\lambda] \otimes_{\B_{\tau,K}^{\dagger}}D_{\tau}^{\dagger}(V) =  \B_{\tau,\mathrm{rig},K}^{\dagger}[1/\lambda] \otimes_F D_{\mathrm{cris}}(V).$$
\end{enumerate}
\end{enonce*}

Si on souhaite attacher au $(\phi,\tau)$-module $D_{\tau,\mathrm{rig}}^\dagger(V)$ une connexion, il semble naturel de considérer l'action infinitésimale de $\tau^{\Zp}$, mais l'opérateur ainsi obtenu n'est pas défini sur $D_{\tau,\mathrm{rig}}^\dagger(V)$ mais sur un objet plus gros, puisque $\tau$ n'agit pas directement sur $D_{\tau,\mathrm{rig}}^\dagger(V)$. On va en fait montrer qu'on peut renormaliser l'opérateur obtenu en une connexion $N_\nabla$ qui a le bon goût de stabiliser $D_{\tau,\mathrm{rig}}^\dagger(V)$ et donc de lui conférer une structure de $\phi$-module différentiel à connexion, ce qui nous permet de définir la notion de $(\phi,N_\nabla)$-module sur $\B_{\tau,\mathrm{rig},K}^\dagger$ et de construire un foncteur $V \mapsto D_{\tau,\mathrm{rig}}^\dagger(V)$ de la catégorie des représentations $p$-adiques de $\G_K$ dans celle des $(\phi,N_\nabla)$-modules sur $\B_{\tau,\mathrm{rig},K}^\dagger$. L'avantage de considérer des $(\phi,N_\nabla)$-modules sur $\B_{\tau,\mathrm{rig},K}^\dagger$ plutôt que des $(\phi,\tau)$-modules est que la connexion $N_\nabla$ est déjà définie sur le $\phi$-module, ce qui fait qu'on n'a pas besoin de tensoriser par des $\B_{\tau,K}$-algèbres sur lesquelles l'action de $\tau$ peut être compliquée. Il y a cependant plusieurs problèmes à travailler avec des $(\phi,N_\nabla)$-modules sur $\B_{\tau,\mathrm{rig},K}^\dagger$. Le premier problème est qu'on travaille avec des objets définis sur $\B_{\tau,\mathrm{rig},K}^\dagger$, ce qui fait qu'on n'a pas de bonne notion d'intégralité, là où les $(\phi,\tau)$-modules étaient au départ très liés aux modules de Breuil-Kisin : si $V$ est une représentation $p$-adique semi-stable à poids de Hodge-Tate négatifs, alors le $\phi$-module sous-jacent au $(\phi,\tau)$-module associé à $V$ par la théorie de Caruso est en fait engendré par un $\phi$-module sur $\O_F[\![T]\!][1/p]$ qui est exactement le module de Breuil-Kisin associé à $V$ par la théorie de Kisin \cite{KisinFiso}. Le deuxième problème est qu'on peut avoir un isomorphisme de $(\phi,N_\nabla)$-modules $D_{\tau,\mathrm{rig}}^\dagger(V) = D_{\tau,\mathrm{rig}}^\dagger(V')$ sans que $V$ et $V'$ soient isomorphes en tant que représentations. En fait, la proposition \ref{V V' meme Nabla} nous dit exactement quand cela peut se produire~: deux représentations $p$-adiques $V, V'$ de $\G_K$ sont telles que $D_{\tau,\mathrm{rig}}^\dagger(V) = D_{\tau,\mathrm{rig}}^\dagger(V')$ en tant que $(\phi,N_\nabla)$-modules si, et seulement si, il existe un rang $n \geq 0$ tel que $V_{|\G_{K_n}}$ et $V'_{|\G_{K_n}}$ sont isomorphes en tant que représentations de $\G_{K_n}$ (ce qui semble raisonnable si on se rappelle de l'opérateur $N_\nabla$ provient de l'action infinitésimale de $\tau^{\Zp}$). 

Ensuite, en utilisant les $(\phi,\tau)$-modules différentiels à connexion, on peut généraliser les constructions de Kisin dans \cite{KisinFiso}, en utilisant des démonstrations analogues à celles de Berger dans \cite{Ber08} pour les $(\phi,\Gamma)$-modules, et associer à tout $(\phi,N,\G_{M/K})$-module filtré $D$ un $(\phi,\tau)$-module différentiel $\cal{M}(D)$, ce qui nous permet d'obtenir le résultat suivant~:

\begin{enonce*}{Théorème C}
Le foncteur $D \mapsto \cal{M}(D)$, de la catégorie des $(\phi,N,\G_{M/K})$-modules filtrés dans la catégorie des $(\phi,\tau)$-modules différentiels sur $\B_{\tau,\mathrm{rig},K}^\dagger$ dont la connexion associée est localement triviale, est une équivalence de catégories. De plus, le $(\phi,N,\G_{M/K})$-module filtré $D$ est admissible si et seulement si $\cal{M}(D)$ est étale.
\end{enonce*}

Enfin, on s'intéresse aux représentations de $E$-hauteur finie (ce sont celles dont le $(\phi,\tau)$-module associé admet un $(\phi,\tau)$-réseau sur $\O_F[\![T]\!]$ et vérifient une condition technique sur le conoyau du Frobenius), en montrant comment retrouver le théorème principal de \cite{GaoEhauteur} à partir des $(\phi,N_\nabla)$-modules~:

\begin{enonce*}{Théorème D}
Soit $V$ une représentation $p$-adique de $\G_K$. Alors $V$ est de $E$-hauteur finie si, et seulement si, il existe $n \geq 0$ tel que $V_{|\G_{K_n}}$ est semi-stable à poids de Hodge-Tate négatifs.
\end{enonce*}

On peut même montrer qu'un tel $n$ est nécessairement borné par une constante ne dépendant que de $K$. On montre également comment réinterpréter le fait d'être de $E$-hauteur finie en termes de la connexion $N_\nabla$~:

\begin{enonce*}{Théorème E}
Soit $V$ une représentation $p$-adique. Alors $V$ est de $E$-hauteur finie si et seulement si $D_{\tau,\mathrm{rig}}^+(V)$ est unipotent.
\end{enonce*}

On montre également comment passer du $(\phi,\tau)$-module associé à une représentation $V$ et relatif à une extension de Kummer donnée, au $(\phi,\tau)$-module associé à $V$ mais relatif à une autre extension de Kummer, en construisant un anneau $\B_{\tau,\tau',K}$~:

\begin{enonce*}{Théorème F}
Soit $V$ une représentation $p$-adique de $\G_K$. Alors on a 
$$\B_{\tau,\tau',K} \otimes_{\B_{\tau,K}}D(V) = \B_{\tau,\tau',K} \otimes_{\B_{\tau',K}}D'(V).$$
De plus, on peut récupérer $D(V)$ en fonction de $D'(V)$ \textit{via} la formule
$$D(V) = (\B_{\tau,\tau',K} \otimes_{\B_{\tau',K}}D'(V))^{H_{\tau,K}}.$$
\end{enonce*}

\subsection*{Plan de l'article}
Cet article comporte quatre chapitres, chacun divisé en plusieurs sections. Le premier chapitre est consacré à des rappels sur la théorie des $(\phi,\tau)$-modules de Caruso et des anneaux de périodes associés, et à l'exposition de nombreux résultats concernant la théorie des vecteurs localement analytiques dans le cadre des $(\phi,\tau)$-modules. Les vecteurs localement analytiques ont été un ingrédient essentiel dans la démonstration de la surconvergence des $(\phi,\tau)$-modules dans \cite{GP18} et apportent un point de vue très utile pour étudier les $(\phi,\tau)$-modules différentiels. On donne dans le deuxième chapitre la construction des $(\phi,\tau)$-modules différentiels et $(\phi,N_\nabla)$-modules et on y montre les théorèmes A et B. Dans le troisième chapitre, on rappelle les propriétés à connaître en ce qui concerne les $(\phi,N,\G_{M/K})$-modules filtrés et la théorie de Kisin et on démontre le théorème C. Enfin, dans le quatrième et dernier chapitre, on s'intéresse aux représentations de $E$-hauteur finie et on montre comment déduire des constructions du chapitre 2 concernant les $(\phi,N_\nabla)$-modules les théorèmes D et E. On en profite également pour donner une recette qui permet de passer du $(\phi,\tau)$-module associé à une représentation et relatif à une extension de Kummer à son $(\phi,\tau')$-module relatif à une autre extension de Kummer, ce qui est l'objet du théorème F, et on discute de la compatibilité de cette construction avec la notion de $E$-hauteur.
  
\subsection*{Remerciements}  
Une grande partie des résultats de cet article ont été démontrés dans le cadre de ma thèse. Je tiens en particulier à remercier mon directeur, Laurent Berger, pour ses conseils et discussions sur le sujet, ainsi que l'UMPA (UMR 5669 CNRS) pour m'avoir offert d'excellentes conditions de travail durant ma thèse. La fin de la rédaction de cet article a eu lieu au BICMR, que je tiens à remercier ainsi que Ruochuan Liu pour leur hospitalité. Je tiens également à remercier Xavier Caruso pour sa lecture et ses remarques.

\section{$(\phi,\tau)$-modules et vecteurs localement analytiques}
\subsection{Extensions de Kummer et $(\phi,\tau)$-modules}
Soit $K$ un corps $p$-adique, c'est-à-dire un corps de caractéristique $0$, muni d'une valuation discrète pour lequel il est complet et dont le corps résiduel est un corps parfait $k$ de caractéristique $p$, i.e. $K$ est une extension finie totalement ramifiée de $F = W(k)[1/p]$, où $W(k)$ est l'anneau des vecteurs de Witt à coefficients dans $k$. 

Soit $\pi_0 = \pi$ une uniformisante de $\O_K$. On note aussi $E(X) \in \O_F[X]$ le polynôme minimal de $\pi$ sur $F$ qui est un polynôme d'Eisenstein, et on note $e=[K:F]$. On fixe également une suite compatible de racines $p^n$-ièmes de $\pi$, c'est-à-dire une suite $(\pi_n)_{n \geq 0}$ telle que pour tout $n \in \N$, $\pi_{n+1}^p = \pi_n$. On pose $K_n = K(\pi_n)$ et $K_\infty = \bigcup_{n \in \N}K_n$. On appelle extension de Kummer une telle extension $K_\infty/K$. Les extensions $K_n/K$ ne sont pas galoisiennes, en tout cas à partir du moment où $\zeta_{p^n}$ n'est pas dans $K$, et donc $K_{\infty}/K$ n'est pas galoisienne. La clôture galoisienne de $K_{\infty}$ est $L = K_{\infty}\cdot K_{\mathrm{cycl}}$. On note $G_{\infty} = \Gal(L/K)$ et $H_{\infty} = \G_L = \Gal(\overline{K}/L)$, et $\Gamma = \Gal(L/K_{\infty})$ qui s'identifie à $\Gal(K_{\mathrm{cycl}}/(K_\infty \cap K_{\mathrm{cycl}}))$ et donc aussi à un sous-groupe ouvert de $\Z_p^{\times}$. Pour $g \in \G_K$ et pour $n \in \N$, il existe un unique élément $c_n(g) \in \Z/p^n\Z$ tel que $g(\pi_n) = \zeta_{p^n}^{c_n(g)}\pi_n$. Comme $c_{n+1}(g) = c_n(g) \mod p^n$, la suite $(c_n(g))$ définit donc un élément $c(g)$ de $\Zp$. De façon équivalente, la suite $(\pi_n)_{n \in \N}$ définit un élément $\tilde{\pi}$ de l'anneau $\Etplus$, et si $g \in \G_K$, il existe un unique élément $c(g) \in \Zp$ tel que $g(\tilde{\pi}) = \epsilon^{c(g)}\tilde{\pi}$. 

L'application $g \mapsto c(g)$ est en fait un $1$-cocycle (continu) de $\G_K$ dans $\Zp(1)$, tel que $c^{-1}(0) = \Gal(\overline{K}/K_{\infty})$, et vérifie pour $g,h \in \Gal(\overline{K}/K_{\infty})$~:
$$c(gh) = c(g)+\chi_{\mathrm{cycl}}(g)c(h).$$

Par conséquent, si $\Zp \rtimes \Z_p^{\times}$ désigne le produit semi-direct de $\Zp$ par $\Z_p^{\times}$ et où $\Z_p^{\times}$ agit sur $\Zp$ par multiplication, l'application $g \in \G_K \mapsto (c(g),\chi_{\mathrm{cycl}}(g)) \in \Zp \rtimes \Z_p^{\times}$ est un morphisme de groupes de noyau $H_{\infty}$. Le cocycle $c$ se factorise à travers $H_{\infty}$, ce qui nous donne un cocycle qu'on note toujours $c~: G_\infty \to \Zp$ et qu'on appelle le cocycle de Kummer de l'extension $K_\infty/K$. 

Soit maintenant $\tau$ un générateur topologique de $\Gal(L/K_{\mathrm{cycl}})$ tel que $c(\tau)=1$ (c'est donc l'élément correspondant à $(1,1)$ via l'isomorphisme $g \in \G_L \mapsto (c(g),\chi_{\mathrm{cycl}}(g)) \in \Zp \rtimes \Z_p^\times$). La relation entre $\tau$ et $\Gamma$ est donnée par $g\tau g^{-1} = \tau^{\chi_{\mathrm{cycl}}(g)}$. On note également $\tau_n:=\tau^{p^n}$.

Le cas $p=2$ est légèrement différent des autres et le résultat suivant montre pourquoi on devra faire attention et parfois effectuer des modifications pour inclure ce cas de figure~:
\begin{prop}
\label{p=2=pb}
On a~:
\begin{enumerate}
\item Si $K_\infty \cap K_{\mathrm{cycl}} = K$, $\Gal(L/K_\infty)$ et $\Gal(L/K_{\mathrm{cycl}})$ engendrent topologiquement $G_\infty$ ;
\item si $K_\infty \cap K_{\mathrm{cycl}} \neq K$, alors nécessairement $p=2$ et dans ce cas $\Gal(L/K_\infty)$ et $\Gal(L/K_{\mathrm{cycl}})$ engendrent un sous-groupe ouvert de $G_\infty$ d'indice $2$.
\end{enumerate}
En particulier, si $p \neq 2$, alors $K_\infty \cap K_{\mathrm{cycl}} = K$.
\end{prop}
\begin{proof}
Pour le premier point, voir \cite[Lemm. 5.1.2]{Liu08}. Pour le deuxième, voir \cite[Prop. 4.1.5]{Liu10}.
\end{proof}

Pour bien faire la différence avec les anneaux définis dans le cadre cyclotomique, on mettra en indice des anneaux relatifs à l'extension de Kummer un symbole $\tau$. Il s'agit simplement d'une notation, les anneaux ne dépendant pas du choix de l'élément $\tau$ mais uniquement de l'extension de Kummer considérée. Ainsi, on notera $H_K = \Gal(\overline{K}/K_{\mathrm{cycl}})$ comme usuellement, et $H_{\tau,K}=\Gal(\overline{K}/K_\infty)$ et, si $A$ est une algèbre munie d'une action de $\G_K$, on pose $A_K = A^{H_K}$ et $A_{\tau,K}=A^{H_{\tau,K}}$.  

On définit $\Etplus = \varprojlim\limits_{x \to x^p} \O_{\C_p} = \{(x^{(0)},x^{(1)},\dots) \in \O_{\C_p}^{\N}~: (x^{(n+1)})^p=x^{(n)}\}$ et on rappelle que $\Etplus$ est naturellement muni d'une structure d'anneau qui en fait un anneau parfait de caractéristique $p$, complet pour la valuation $v_{\E}$ définie par $v_{\E}(x) = v_p(x^{(0)})$. On note $\Et$ son corps des fractions. La théorie du corps des normes (cf. \cite{Win83}) permet d'associer à l'extension $K_\infty/K$ son corps des normes et de le plonger dans $\Et$. La famille $(\zeta_{p^n})$ et la famille $(\pi_n)$ définissent chacune un élément de $\Etplus$, qu'on notera respectivement $\epsilon$ et $\tilde{\pi}$. On pose $\overline{u} = \epsilon -1$, et on rappelle que $v_{\E}(\overline{u}) = \frac{p}{p-1}$. L'image du plongement du corps des normes de $K_{\infty}/K$ dans $\Et$ est alors $\E_{\tau,K} := k((\tilde{\pi}))$. Soit $\E_\tau$ la clôture séparable de $\E_{\tau,K}$ dans $\Et$. Comme $\Gal(\overline{K}/K_{\infty})$ agit trivialement sur $\E_{\tau,K}$, tout élément de $\Gal(\overline{K}/K_{\infty})$ stabilise $\E_\tau$, ce qui nous donne un morphisme $\Gal(\overline{K}/K_{\infty}) \to \Gal(\E_\tau/\E_{\tau,K})$. Par le théorème 3.2.2 de \cite{Win83}, c'est même un isomorphisme.

On pose $\At = W(\Et)$, $\Atplus = W(\Etplus)$, $\Bt = \At[1/p]$ et $\Btplus = \Atplus[1/p]$.

On définit maintenant un anneau $\A_{\tau,K}$ à l'intérieur de $\At$ de la façon suivante~:
$$\A_{\tau,K} = \{\sum_{i \in \Z}a_i[\tilde{\pi}]^i, a_i \in \O_F \lim\limits_{i \to - \infty}a_i = 0 \}.$$
Muni de la valuation $p$-adique $v_p(\sum_{i \in \Z}a_i[\tilde{\pi}]^i) = \min_{i \in \Z}v_p(a_i)$, $\A_{\tau,K}$ est un anneau de valuation discrète qui admet $\E_{\tau,K}$ comme corps résiduel. Si $M/K$ est une extension finie, on note $\E_{\tau,M}$ l'extension de $\E_{\tau,K}$ correspondant à $M\cdot K_\infty/K_\infty$ par la théorie du corps des normes.

\begin{prop}
\label{relevement E_tau en A_tau Hensel}
Si $M$ est une extension finie de $K$, alors $\E_{\tau,M}$ est une extension séparable de $\E_{\tau,K}$ et il existe une unique extension non ramifiée $\B_{\tau,M}/\B_{\tau,K}$, contenue dans $\At$, de corps résiduel $\E_{\tau,M}$ et telle que $\Gal(\B_{\tau,M}/\B_{\tau,K}) \simeq \Gal(\E_{\tau,M}/\E_{\tau,K})$. On note $\A_{\tau,M}$ son anneau des entiers pour la valuation $p$-adique et on pose $\A_{\tau,M}^+~:= \Atplus \cap \A_{\tau,M}$, $\B_{\tau,M}^+~:= \Btplus \cap \B_{\tau,M}$.
\end{prop}
\begin{proof}
Comme $\Bt$ est absolument non ramifié, et que $\E_{\tau,M}/\E_{\tau,K}$ est séparable et finie avec $\E_{\tau,M}$ incluse dans le corps résiduel de $\Bt$, il existe une unique extension $\B_{\tau,M}$ de $\B_{\tau,K}$ (qui est donc automatiquement non ramifiée), contenue dans $\Bt$, dont le corps résiduel est $\E_{\tau,M}$, et on a bien $\Gal(\B_{\tau,M}\B_{\tau,K}) \simeq \Gal(\E_{\tau,M}/\E_{\tau,K})$ puisque l'extension est non ramifiée.
\end{proof}

Comme $\E_\tau = \bigcup_{M/F}\E_{\tau,M}$ est la clôture séparable de $\E_{\tau,K}$, l'extension maximale non ramifiée $\B_{\tau,K}^{\mathrm{unr}}$ de $\B_{\tau,K}$ dans $\Bt$ est aussi la réunion des $\B_{\tau,M}$ quand $M$ parcourt les extensions finies de $K$. On note alors $\B_\tau$ l'adhérence de $\B_{\tau,K}^{\mathrm{unr}}$ dans $\Bt$ pour la topologie $p$-adique, et on note $\A_\tau$ son anneau des entiers (qui est aussi le complété de l'anneau des entiers de $\B_{\tau,K}^{\mathrm{unr}}$ pour la topologie $p$-adique). On a donc $\A_\tau/p\A_\tau = \E_\tau$. Comme $\B_{\tau,K}^{\mathrm{unr}}$ est stable sous l'action de $\G_{K_\infty}$, c'est aussi le cas de $\B_\tau$ et on a $\Gal(\B_{\tau,K}^{\mathrm{unr}}/\B_{\tau,M}) \simeq \Gal(\E_\tau/\E_{\tau,M}) \simeq H_{\tau,M}:=\Gal(\overline{K}/M\cdot K_\infty)$. Si $M$ est une extension finie de $K$, le théorème d'Ax-Sen-Tate (voir \cite[Thm. 1]{Tat67}) montre alors que $\B_{\tau}^{H_{\tau,M}} = (\B_\tau^{\mathrm{unr}})^{H_M} = \B_{\tau,M}$ et donc que $\A_{\tau,M} = \A^{H_{\tau,M}}$.

Le Frobenius déduit de la fonctorialité des vecteurs de Witt sur $\Bt$ définit par restriction des endomorphismes de $\A_{\tau,K},\B_{\tau,K},\A_{\tau}$ et $\B_{\tau}$, et envoie notamment $[\tilde{\pi}]$ sur $[\tilde{\pi}]^p$. 

Soit $\At_L = \At^{H_{\infty}}$ (comme $\Et_L$ est parfait, on a aussi $\At_L = W(\Et_L)$) et $\Bt_L = \At_L[1/p]$. Ces anneaux sont munis d'un endomorphisme de Frobenius qui provient de celui sur $\Bt$. On définit également, pour $r > 0$, $\Bt^{\dagger,r}$ l'ensemble des éléments surconvergents de $\Bt$ de rayon $r$, par 

$$\left\{x = \sum_{n \gg -\infty}p^n[x_n] \textrm{ tels que } \lim\limits_{k \to +\infty}v_{\E}(x_k)+\frac{pr}{p-1}k =+\infty \right\}$$

et on note $\Bt^\dagger = \bigcup_{r > 0}\Bt^{\dagger,r}$ l'ensemble des éléments surconvergents.

\begin{defi}
\label{defi generaliser phitau p=2}
On appelle $(\phi,\tau)$-module étale sur $(\A_{\tau,K},\At_L)$ (resp. $(\B_{\tau,K},\Bt_L)$) tout triplet $(D,\phi_D,G)$ où~:
\begin{enumerate}
\item $(D,\phi_D)$ est un $\phi$-module étale sur $\A_{\tau,K}$ (resp. $\B_{\tau,K}$) ;
\item $G$ est une action $G_\infty$-semi-linéaire de $G_\infty$ sur $M:=\At_L \otimes_{\A_{\tau,K}}D$ (resp. sur $M:=\Bt_L \otimes_{\B_{\tau,K}}D$) telle que $G$ commute à $\phi_M:=\phi_{\At_L}\otimes \phi_D$ (resp. $\phi_M:=\phi_{\Bt_L}\otimes \phi_D$), c'est-à-dire que pour tout $g \in G_\infty$, $g\phi_M = \phi_Mg$ ;
\item en tant que sous $\A_{\tau,K}$-module de $M$, $D \subset M^{\Gal(L/K_{\infty})}$.
\end{enumerate}
\end{defi}

En particulier, si $(D,\phi_D,G)$ est un $(\phi,\tau)$-module étale sur $(\A_{\tau,K},\At_L)$, on définit $V(D):=(\At \otimes_{\At_L}D)^{\phi=1}:=(\At \otimes_{\At_L}(\At_L \otimes_{\A_{\tau,K}}D))^{\phi=1}$ comme la $\Zp$-représentation de $\G_K$ où $\G_K$ agit sur $\At_L \otimes_{\A_{\tau,K}}D$ \textit{via} $G$ et $\G_K$ agit diagonalement sur le produit tensoriel $\At \otimes_{\At_L}(\At_L \otimes_{\A_{\tau,K}}D)$.

\begin{prop}
Le foncteur qui à un $(\phi,\tau)$-module étale $(D,\phi_D,G_\infty)$ associe la $\Zp$-représentation de $\G_K$ 
$$V(D):=(\At \otimes_{\At_L}M)^{\phi=1},$$
induit une équivalence de catégories entre la catégorie des $(\phi,\tau)$-modules étales sur $(\A_{\tau,K},\At_L)$ et celle des $\Zp$-représentations de $\G_K$.
\end{prop}
\begin{proof}
Voir \cite[Prop. 2.1.7]{gao2016loose}.
\end{proof}

La définition \ref{defi generaliser phitau p=2} n'est pas tout à fait la définition originale de Caruso mais celle de Gao et Liu dans \cite{gao2016loose}, qui a le bon goût d'être équivalente à celle de Caruso lorsque $p \neq 2$ via la proposition précédente et de se généraliser au cas $p=2$.

On rappelle au passage le résultat principal de \cite{GP18} :

\begin{theo}
\label{surconvphitau}
Soit $V$ une représentation $p$-adique de $\G_K$, et soit $D(V)$ son $(\phi,\tau)$-module associé sur $(\B_{\tau,K},\Bt_L)$, c'est-à-dire la donnée du $\phi$-module $D(V) = (\B_{\tau}\otimes_{\Qp}V)^{\G_{K_{\infty}}}$ et d'une action semi-linéaire de $G_\infty$ sur $M(V) = (\Bt \otimes_{\Qp} V)^{\G_L}$. 

Pour $r \geq 0$, on définit $D^{\dagger,r}(V) = (\B_{\tau}^{\dagger,r}\otimes_{\Qp}V)^{\G_{K_{\infty}}}$ et $M^{\dagger,r}(V)=(\Bt^{\dagger,r}\otimes_{\Qp}V)^{\G_L}$. 

Alors il existe $r \geq 0$ tel qu'on ait
$$D(V) = \B_{\tau,K} \otimes_{\B_{\tau,K}^{\dagger,r}} D^{\dagger,r}(V)$$
et 
$$M(V) = \Bt_L \otimes_{\Bt_L^{\dagger,r}}M^{\dagger,r}(V).$$
où $\B_{\tau,K}^{\dagger,r} = \B_{\tau,K} \cap \Bt^{\dagger,r}$.
\end{theo}

Pour ce qu'on compte faire ensuite, on aura également besoin de savoir ce que devient le $(\phi,\tau)$-module associé à la restriction d'une représentation $V$ de $\G_K$ à $\G_M$, où $M$ est une extension finie de $K$.

On pose $M_\infty:=K_\infty\cdot M$, $L_M:=L\cdot M$ et $H_{\tau,M}:=\Gal(\overline{K}/M_\infty)$. Si $V$ est une représentation $p$-adique de $\G_M$, on peut définir $D(V):= (\B_\tau \otimes_\Qp V)^{H_{\tau,M}}$ qui est naturellement muni d'une structure de $\phi$-module étale sur $\B_{\tau,M}$, et $\tilde{D}(V):=(\Bt\otimes_\Qp V)^{\G_{L_M}}$, qui est naturellement muni d'une structure de $\phi$-module étale sur $\Bt_{L_M}$, muni d'une action de $\Gal(L_M/M)$ qui commute à l'action de $\phi$.

\begin{defi}
On appelle $(\phi,\tau_M)$-module sur $(\B_{\tau,M},\Bt_{L_M})$ la donnée d'un triplet 
$$(D,\phi_D,\Gal(L_M/M)),$$ 
où~:
\begin{enumerate}
\item $(D,\phi_D)$ est un $\phi$-module sur $\B_{\tau,M}$ ;
\item $\Gal(L_M/M)$ est une action $\Gal(L_M/M)$-semi-linéaire sur $\Bt_{L_M}$ de $\Gal(L_M/M)$ sur $\hat{D}:=\Bt_{L_M} \otimes_{\B_{\tau,M}}D$ telle que cette action commute à $\phi:= \phi_{\Bt_{L_M}}\otimes \phi_D$ ;
\item en tant que sous $\B_{\tau,M}$-module de $\hat{D}$, on a $D \subset \hat{D}^{H_{\tau,M}}$.
\end{enumerate}
\end{defi}

Si $(D,\phi_D,\Gal(L_M/M))$ est un tel $(\phi,\tau_M)$-module, on définit $V(D):=(\B_\tau \otimes_{\B_{\tau,M}}D)^{\phi=1}$ et 
$$\tilde{V}(D):=(\Bt \otimes_{\B_{\tau,M}}D)^{\phi=1}=(\Bt \otimes_{\Bt_{L_M}}\hat{D})^{\phi=1}.$$
Comme habituellement, un $(\phi,\tau_M)$-module est dit étale si le $\phi$-module sous-jacent est étale.
\begin{prop}~
\begin{enumerate}
\item Les foncteurs $D \mapsto V(D)$ et $V \mapsto D(V)$ sont quasi-inverses l'un de l'autre et induisent une équivalence de catégories tannakiennes entre la catégorie des $\phi$-modules étales sur $\B_{\tau,M}$ et celle des représentations $p$-adiques de $\G_{M_\infty}$ ;
\item $\tilde{V}(D)_{|\G_{M_\infty}} \simeq V(D)$ ;
\item Les foncteurs $D \mapsto \tilde{V}(D)$ et $V \mapsto \tilde{D}(V)$ sont quasi-inverses l'un de l'autre et induisent une équivalence de catégories tannakiennes entre la catégorie des $(\phi,\tau_M)$-modules étales sur $(\B_{\tau,M},\Bt_{L_M})$ et celle des représentations $p$-adiques de $\G_M$.
\end{enumerate}
\end{prop}
\begin{proof}
Le premier point est la proposition $A.1.2.6$ de \cite{Fon90}. Le deuxième point est une conséquence directe de \cite[Lemm. 2.1.4]{gao2016loose} et de notre définition de $(\phi,\tau_M)$-module. Pour le dernier point, on renvoie à la preuve de \cite[Prop. 2.1.7]{gao2016loose}.
\end{proof} 

Si on était parti au départ d'une représentation $p$-adique $V$ de $\G_K$ dont on a considéré la restriction à $\G_M$, et si de plus on suppose que $M_\infty/K_\infty$ est galoisienne, alors $(\B_\tau \otimes_\Qp V)^{H_{\tau,M}}$ est naturellement muni d'une action de $H_{\tau,K}/H_{\tau,M} \simeq \Gal(M_\infty/K_\infty)$. Comme de plus, $(\Bt\otimes_\Qp V)^{\G_L} \subset (\Bt\otimes_\Qp V)^{\G_{L_M}}$, l'action de $\G_M/\G_{L_M}$ en tant que sous-groupe de $\G_K/\G_L$ sur $(\Bt\otimes_\Qp V)^{\G_{L_M}}$ coïncide avec celle de $\G_M/\G_{L_M}$ sur $(\Bt\otimes_\Qp V_{|\G_M})^{\G_{L_M}}$. Ces considérations nous amènent à faire la définition suivante~:

\begin{defi}
Si $M/K$ est une extension galoisienne finie telle que $M$ et $K_\infty$ soient linéairement disjointes au-dessus de $K$, et si $D=(D,\phi_D,\Gal(L_M/M))$ est un $(\phi,\tau_M)$-module sur $(\B_{\tau,M},\Bt_{L_M})$, on dit que $D$ est muni d'une action de $\Gal(M/K)$ si $\G_K$ agit sur $\hat{D}:=\Bt_{L_M} \otimes_{\B_{\tau,M}}D$ et si~:
\begin{enumerate}
\item $D$ en tant que sous ensemble de $\hat{D}$ est stable sous l'action de $H_K \subset \G_K$ ;
\item $\G_{L_M}$ agit trivialement sur $\hat{D}$, et $H_{\tau,M}$ agit trivialement sur $D$ ;
\item l'action de $\G_M/\G_{L_M} \subset \G_K/\G_L$ coïncide avec l'action de $\Gal(L_M/M)$ sur $\hat{D}$.
\end{enumerate}
\end{defi}

\subsection{Vecteurs localement analytiques et pro-analytiques}
Soit $G$ un groupe de Lie $p$-adique et soit $W$ une $\Qp$-représentation de Banach de $G$. L'espace $W^{\la}$ des vecteurs localement analytiques de $W$ est défini dans \cite{schneider2002bis} et on va brièvement rappeler les définitions générales ainsi que les propriétés des vecteurs localement analytiques qui seront nécessaires pour le reste de l'article, en suivant la présentation de \cite{Ber14SenLa} et \cite{Ber14MultiLa}. 

On utilisera la notation multi-indice suivante~: si $\mathbf{c} = (c_1,\cdots,c_d)$ et si $\k=(k_1,\cdots,k_d) \in \N^d$, alors $\mathbf{c}^\k:=(c_1^{k_1},\cdots,c_d^{k_d})$. On pose aussi $\k!=k_1!\times \cdots \times k_d!$ et $|\k|=k_1+\cdots k_d$. On note également $\mathbf{1}_j$ le $d$-uplet $(k_1,\cdots,k_d)$ où $k_i = 0$ si $i \neq j$ et $k_j=1$.

On rappelle qu'un espace LB est un espace topologique localement convexe qui est limite inductive d'une famille dénombrable d'espaces de Banach, et qu'un espace LF est un espace topologique localement convexe qui est limite inductive d'une famille dénombrable d'espaces de Fréchet. 

Soit $H$ un sous-groupe ouvert de $G$ tel qu'il existe des coordonnées $c_1, \cdots, c_d~: H \to \Zp$ donnant lieu à une bijection analytique $\mathbf{c}~: H \to \Z_p^d$. 

\begin{defi}
On dit que $x \in W$ est un vecteur $H$-analytique si l'application orbite $H \to W$ donnée par $g \mapsto g(x)$ est analytique. 

On dit que $x \in W$ est un vecteur localement analytique pour $G$ si l'application orbite $G \to W$ donnée par $g \mapsto g(x)$ est localement analytique, ou de façon équivalente, s'il existe un ouvert $H$ de $G$ comme ci-dessus tel que $x$ est $H$-analytique.

On note $W^{H-an}$ l'ensemble des vecteurs $H$-analytiques de $W$ et $W^{\la}$ l'ensemble des vecteurs localement analytiques de $W$ pour $G$.
\end{defi}

Dans le cas particulier où $U=\Z_p^d \subset \Q_p^d$, l'ensemble des fonctions analytiques $f~: \Z_p^d \to W$ s'identifie à l'ensemble des fonctions $f~: \Z_p^d \to W$ telles qu'il existe une suite $(f_{\k})_{\k \in \N^d}$ avec $f_{\k} \rightarrow 0$ dans $W$, telle que $f(\mathbf{x}) = \sum_{\k \in \N^d}\mathbf{x}^{\k}f_{\k}$ pour tout $\mathbf{x} \in \Z_p^d$. L'espace des fonctions analytiques $\Z_p^d \to W$ s'injecte dans l'espace $\cal{C}^{an}(\Z_p^d,W)$, et est muni d'une norme $||\cdot||$ donnée par $||f||=\sup_{\k \in \N^d}||f_\k||$ qui en fait un espace de Banach. 

De façon plus générale, si $H$ est comme précédemment un sous-groupe ouvert de $G$ tel qu'il existe des coordonnées $c_1, \cdots, c_d~: H \to \Zp$ donnant lieu à une bijection analytique $\mathbf{c}~: H \to \Z_p^d$, et si  $w \in W^{H-an}$, alors il existe une suite $(w_{\k})_{\k \in \N^d}$ avec $w_{\k} \rightarrow 0$ dans $W$, telle que $g(w) = \sum_{\k \in \N^d}\mathbf{c}(g)^{\k}w_{\k}$ pour tout $g$ dans $H$. De plus, $W^{H-an}$ s'injecte dans l'espace $\cal{C}^{an}(H,W)$. On peut montrer (voir §10 et §12 de \cite{schneider2011p}) que $\mathcal{C}^{an}(H,W)$ est naturellement muni d'une norme telle que, si $||\cdot||_H$ désigne la norme induite sur $W^{H-an}$ \textit{via} l'injection $\cal{C}^{an}(H,W)$ et si $w \in W^{H-an}$, alors $||w||_H = \sup_{\k \in \N^d}||w_\k||$. Cela fait de $W^{H-an}$ un espace de Banach. La définition de $W^{\la}$ montre que $W^{\la} = \cup_H W^{H-an}$ où $H$ parcourt une suite de certains sous-groupes ouverts de $G$, et on munit $W^{\la}$ de la topologie de la limite inductive, ce qui fait de $W^{\la}$ un espace LB. 

Soient $G,H$ deux groupes de Lie $p$-adiques et $f~: G \to H$ un morphisme analytique de groupes. Si $W$ est une représentation de $H$, on peut aussi voir $W$ comme une représentation de $G$ et on a directement le lemme suivant~:
\begin{lemm}
Si $w$ est un vecteur localement analytique pour $H$ alors c'est un vecteur localement $\Qp$-analytique pour $G$.
\end{lemm}

\begin{lemm}
\label{app linéaire la}
Soient $W,X$ deux $\Qp$-espaces de Banach et soit $\pi~: W \to X$ une application linéaire continue. Si $f~: G \to W$ est une fonction localement analytique alors $\pi\circ f~: G \to X$ est localement $\Qp$-analytique.
\end{lemm}
\begin{proof}
Voir \cite[Lemm. 2.2]{Ber14SenLa}.
\end{proof}

\begin{prop}
\label{ladansla}
Soit $B$ un $G$-anneau de Banach et soit $W$ un $B$-module libre de type fini muni d'une action compatible de $G$. Si $W$ admet une base $w_1,\cdots,w_d$ dans laquelle $g \mapsto \Mat(g)$ est une fonction analytique $G \to \GL_d(B) \subset M_d(B)$, alors 
\begin{enumerate}
\item $W^{H-an}=\oplus_{j=1}^d B^{H-an}\cdot w_j$ si $H$ est un sous-groupe ouvert de $G$. 
\item $W^{\la} = \oplus_{j=1}^dB^{\la}\cdot w_j$.
\end{enumerate}
\end{prop}
\begin{proof}
Voir \cite[Prop. 2.3]{Ber14SenLa}.
\end{proof}

Soit $W$ un espace de Fréchet, dont la topologie est définie par une suite $(p_i)_{i \geq 1}$ de semi-normes. On note $W_i$ la complétion de $W$ pour $p_i$, de telle sorte que $W = \varprojlim_{i \geq 1}W_i$. L'espace $W^{\la}$ peut en fait être défini si $W$ est un espace de Fréchet (voir \cite{emerton2004locally} par exemple), mais en général l'objet obtenu est trop petit, et on fait donc comme dans \cite[Déf. 2.3]{Ber14MultiLa} la définition suivante~:
\begin{defi}
Si $W= \varprojlim_{i \geq 1}W_i$ est une représentation de Fréchet de $G$, on dit que $w \in W$ est pro-analytique si son image $\pi_i(w)$ dans $W_i$ est localement analytique pour tout $i$. On note $W^{\pa}$ l'ensemble de ces vecteurs.
\end{defi} 

On étend la définition de $W^{\la}$ et $W^{\pa}$ aux cas où $W$ est respectivement un espace LB et un espace LF. Remarquons que si $W$ est LB, alors $W^{\la} = W^{\pa}$. Si $W$ est un espace LF, alors $W^{\la} \subset W^{\pa}$ mais $W^{\pa}$ est en général plus gros.

\begin{prop}
\label{padanspa}
Soit $B$ un $G$-anneau de Fréchet et soit $W$ un $B$-module libre de type fini muni d'une action compatible de $G$. Si $W$ admet une base $w_1,\cdots,w_d$ dans laquelle $g \mapsto Mat(g)$ est une fonction pro-analytique $G \to \GL_d(B) \subset M_d(B)$, alors $W^{\pa}=\oplus_{j=1}^dB^{\pa}\cdot w_j$.
\end{prop}
\begin{proof}
Si $w \in W$, on peut écrire $w = \sum_{j=1}^db_jw_j$ avec $b_j \in B$. Si $w \in W^{\pa}$ et $i \geq 1$, alors $\pi_i(b_j) \in B_i^{\la}$ pour tout $i$ par la proposition \ref{ladansla} et donc $b_j \in B^{\pa}$.
\end{proof}

\begin{rema}
\label{groupes coordonnées}
Il est souvent pratique de pouvoir travailler avec une suite de sous-groupes $(G_n)_{n \geq 1}$ de $G$ telle que $G_1$ est un sous-groupe compact ouvert de $G$, $G_n = G_1^{p^{n-1}}$ et $G_n$ est un sous-groupe de $G_1$ tel qu'il existe des coordonnées locales $c_i~: G_1 \to \Zp$ telles que $\mathbf{c}(G_n) = (p^n\Zp)^d$. On peut toujours trouver un tel sous-groupe puisqu'il suffit de se donner un sous-groupe compact ouvert $G_0$ de $G$ qui est $p$-valué et saturé (voir \cite[§23,§26 et §27]{schneider2011p} pour la définition et l'existence d'un tel sous-groupe), et de poser $G_n=G_0^{p^n}$.
\end{rema}

Si on dispose d'une suite de sous-groupes $(G_n)_{n \geq 1}$ vérifiant les conditions exposées dans la remarque \ref{groupes coordonnées}, et si $w \in W^{\la}$, alors il existe $n \geq 1$ tel que $w \in W^{G_n-an}$ et on peut écrire $g(w) = \sum_{\k \in \N^d}\mathbf{c}(g)^{\k}w_{\k}$ si $g \in G_n$, où $(w_{\k})_{\k \in \N^d}$ est une suite de $W$ telle que $p^{n|\k|}w_{\k} \rightarrow 0$. On pose alors $||w||_{G_n} = \sup_{\k \in \N^d}||p^{n|\k|}\cdot w_{\k}||$. Cette norme coïncide avec celle qu'on déduit de l'inclusion $W^{G_n-an} \to  \mathcal{C}^{an}(G_n,W)$. Par le corollaire 3.3.6 de \cite{emerton2004locally}, l'inclusion $(W^{G_n-an})^{G_n-an} \to W^{G_n-an}$ est un isomorphisme topologique. En particulier, si $w \in W^{G_n-an}$, alors $w_{\k} \in W^{G_n-an}$ pour tout $\k \in \N^d$. Comme conséquence directe des définitions, on dispose du lemme et de la proposition suivants~:

\begin{lemm}
\label{normela}
Si $w \in W^{G_n-an}$, alors
\begin{enumerate}
\item $w \in W^{G_m-an}$ pour tout $m \geq n$.
\item $||w||_{G_{m+1}} \leq ||w||_{G_m}$ si $m \geq n$.
\item $||w||_{G_m} = ||w||$ si $m \gg 0$.
\end{enumerate}
\end{lemm}

\begin{prop}
L'espace $W^{G_n-an}$, muni de $||\cdot||_{G_n}$, est un espace de Banach.
\end{prop}

\begin{lemm}
\label{anneaula}
Si $W$ est un anneau tel que $||xy|| \leq ||x||\cdot ||y||$ alors $W^{H-an}$ est aussi un anneau et $||xy||_H \leq ||x||_H\cdot ||y||_H$ pour $x, y \in W^{H-an}$. Si de plus, $w \in W^{\times} \cap W^{\la}$, alors $\frac{1}{w} \in W^{\la}$. En particulier, si $W$ est un corps, alors $W^{\la}$ est aussi un corps.
\end{lemm}
\begin{proof}
C'est exactement ce que montre la preuve du lemme 2.5 de \cite{Ber14SenLa}. Cependant, l'énoncé du lemme 2.5 d'ibid. ne prétend pas montrer que si $w \in W^{\times} \cap W^{\la}$, alors $\frac{1}{w} \in W^{\la}$, et on en redonne donc la preuve ici.

Soit $w \in W^{G_n-an} \setminus \{0\}$ avec $g(w) = \sum_{\k \in \N^d}\mathbf{c}(g)^\k w_\k$. On a alors
\[
\frac{1}{g(w)}=\frac{1}{w+\sum_{\k \neq 0}\mathbf{c}(g)^\k w_\k} = \frac{1}{w}\cdot \frac{1}{1+\sum_{\k \neq 0}\mathbf{c}(g)^\k\cdot w_\k/w}.
\]
En particulier, $1/w \in W^{G_m-an}$ dès que $m \geq n$ est assez grand pour que $\sup_{\k \neq 0}||p^{m|\k|}\cdot w_\k/w|| < 1$ de sorte que $g(1/w) = \sum_{j \geq 0}(-1)^j\left(\sum_{\k \neq 0}\mathbf{c}(g)^\k w_\k/w\right)^j/w$.
\end{proof}

\begin{lemm}
\label{LieG}
Si $D \in \Lie(G)$ et $n \geq 1$, alors $D(W^{G_n-an}) \subset W^{G_n-an}$ et il existe une constante $C_D$ telle que $||D(x)||_{G_n} \leq C_D||x||_{G_n}$ pour $x \in W^{G_n-an}$.
\end{lemm}
\begin{proof}
Voir \cite[Prop. 3.2]{schneider2002}.
\end{proof}

Pour finir, on s'intéressera dans cet article au cas particulier où le groupe de Lie $p$-adique $G$ sera égal à $G_\infty$, et on va donc introduire des notations spécifiques à ce cas de figure.

\begin{defi}
Si $W$ est une représentation de $G_\infty$, on note $W^{\tau=1}$ et $W^{\gamma=1}$ pour respectivement $W^{\Gal(L/K_{\mathrm{cycl}})=1}$ et $W^{\Gal(L/K_\infty)=1}$, et on note
$$W^{\tau-\la}, \quad W^{\tau_n-an}, \quad W^{\gamma-\la},$$
pour respectivement
$$W^{\Gal(L/K_{\mathrm{cycl}})-\la}, \quad W^{\Gal(L/K_{\mathrm{cycl}}(\pi_n))-\la}, \quad W^{\Gal(L/K_\infty)-\la}.$$
\end{defi}

Si $W$ est une représentation $G_\infty$-localement analytique, on définit les opérateurs de dérivation, respectivement dans la direction $\tau$ et la direction cyclotomique, par 
$$\nabla_\tau:=\frac{\log \tau^{p^n}}{p^n} \textrm{ pour } n \gg 0 $$
et
$$\nabla_\gamma:=\frac{\log g}{\log_p(\chi_{\mathrm{cycl}}(g))} \textrm{ pour } g \in \Gal(L/K_\infty) \textrm{ assez proche de } 1.$$

\begin{rema}
Contrairement à ce que la notation pourrait indiquer, $\gamma$ n'est pas un élément de $\Gal(L/K_\infty)$. Dans le cas où $\Gal(L/K_\infty)$ est pro-cyclique (et donc en particulier lorsque $p \neq 2$), on peut en fait montrer qu'on peut choisir un générateur topologique $\gamma$ de $\Gal(L/K_\infty)$ et que les notations sont alors cohérentes avec ce choix. La notation n'est donc ambigüe que pour $p=2$, néanmoins on fait ce choix par souci de simplifier les notations.
\end{rema}

Si $W$ est une représentation de $G_\infty$ et si $W'=W^{\gamma=1}$, on appellera vecteurs localement analytiques de $W'$ (relativement à $K_\infty/K$) et on notera $(W')^{\la}$ les éléments de $(W^{\la})^{\gamma=1}=W^{\tau-\la,\gamma=1}$.

\subsection{Anneaux de périodes et vecteurs localement analytiques}
Dans cette section, on va définir certains anneaux de périodes qui seront utiles par la suite et on va redonner certains résultats concernant la structure des vecteurs localement analytiques et pro-analytiques dans certains modules et anneaux de périodes.

Si $r \geq 0$, on définit une valuation $V(\cdot,r)$ sur $\Btplus[\frac{1}{[\tilde{\pi}]}]$ en posant
$$V(x,r) = \inf_{k \in \Z}(k+\frac{p-1}{pr}v_{\E}(x_k))$$
pour $x = \sum_{k \gg - \infty}p^k[x_k]$. Si maintenant $I$ est un sous-intervalle fermé de $[0;+\infty[$, on pose $V(x,I) = \inf_{r \in I}V(x,r)$. On définit alors l'anneau $\Bt^I$ comme le complété de $\Btplus[1/[\tilde{\pi}]]$ pour la valuation $V(\cdot,I)$ si $0 \not \in I$, et comme le complété de $\Btplus$ pour $V(\cdot,I)$ si $I=[0;r]$. On notera $\Bt_{\mathrm{rig}}^{\dagger,r}$ pour $\Bt^{[r,+\infty[}$ et $\Bt_{\mathrm{rig}}^+$ pour $\Bt^{[0,+\infty[}$. On définit également $\Bt_{\mathrm{rig}}^\dagger = \bigcup_{r \geq 0}\Bt_{\mathrm{rig}}^{\dagger,r}$.

Soit $I$ un sous-intervalle de $]1,+\infty[$ ou tel que $0 \in I$. Soit $f(Y) = \sum_{k \in \Z}a_kY^k$ une série formelle telle que $a_k \in F$ et $v_p(a_k)+\frac{p-1}{pe}k/\rho \to +\infty$ quand $|k| \to + \infty$ pour tout $\rho \in I$. La série $f([\tilde{\pi}])$ converge dans $\Bt^I$ et on note $\B_{\tau,K}^I$ l'ensemble des $f([\tilde{\pi}])$ avec $f$ comme précédemment. C'est un sous-anneau de $\Bt_{\tau,K}^I$. Le Frobenius définit une application $\phi~: \B_{\tau,K}^I \to \B_{\tau,K}^{pI}$. Si $m \geq 0$, alors $\phi^{-m}(\B_{\tau,K}^{p^mI}) \subset \Bt_K^I$ et on note $\B_{\tau,K,m}^I = \phi^{-m}(\B_{\tau,K}^{p^mI})$, de telle sorte qu'on ait $\B_{\tau,K,m}^I \subset \B_{\tau,K,m+1}^I$ pour tout $m \geq 0$.

On notera aussi $\B_{\tau,\mathrm{rig},K}^{\dagger,r}$ pour $\B_{\tau,K}^{[r;+\infty[}$. C'est un sous-anneau de $\B_{\tau,K}^{[r;s]}$ pour tout $s \geq r$ et on note $\B_{\tau,K}^{\dagger,r}$ l'ensemble des $f(u) \in \B_{\tau,\mathrm{rig},K}^{\dagger,r}$ tels que la suite $(a_k)_{k \in \Z}$ soit de plus bornée. Soit $\B_{\tau,K}^{\dagger}=\cup_{r \gg 0}\B_{\tau,K}^{\dagger,r}$. C'est un corps Hensélien (voir \cite[§2]{matsuda1995local}) dont le corps résiduel est $\E_{\tau,K}$. On note $\B_{\tau,K,m}^I = \phi^{-m}(\B_{\tau,K}^{p^mI})$ et $\B_{\tau,K,\infty}^I=\cup_{m \geq 0}\B_{\tau,K,m}^I$ et donc en particulier $\B_{\tau,K,m}^I \subset \Bt_K^I$. 

Pour $M$ extension finie de $K$, il existe par la théorie du corps des normes une extension séparable $\E_{\tau,M}/\E_{\tau,K}$ de degré $[M_{\infty}:K_{\infty}]$. Comme $\B_{\tau,K}^{\dagger}$ est Hensélien, il existe une unique extension finie non ramifiée $\B_{\tau,M}^{\dagger}/\B_{\tau,K}^{\dagger}$ de degré $f = [M_{\infty}:K_{\infty}]$ de corps résiduel $\E_{\tau,M}$ (voir \cite{matsuda1995local}). Il existe par conséquent $r(M) > 0$ et des éléments $x_1,\cdots x_f \in \B_{\tau,M}^{\dagger,r(M)}$ tels que $\B_{\tau,M}^{\dagger,s}= \oplus_{i=1}^f\B_{\tau,K}^{\dagger,s}\cdot x_i$ pour tout $s \geq r(M)$. Si $r(M) \leq \min(I)$, on pose alors $\B_{\tau,M}^I$ la complétion de $\B_{\tau,M}^{\dagger,r(M)}$ pour $V(\cdot,I)$, de telle sorte que $\B_{\tau,M}^I=\oplus_{i=1}^f\B_{\tau,K}^I\cdot x_i$. 

On rappelle que Berger a construit dans \cite[§ 2.4]{Ber02} un anneau $\Bt_{\log}^{\dagger}$, qui est à $\Bt_{\mathrm{rig}}^\dagger$ ce que $\B_{\mathrm{st}}$ est à $\B_{\mathrm{cris}}$, en posant $\Bt_{\log}^{\dagger,r}:=\Bt_{\mathrm{rig}}^{\dagger,r}[\log[\tilde{\pi}]]$ pour $r \geq 0$, et en définissant alors $\Bt_{\log}^{\dagger}:=\bigcup_{r \geq 0}\Bt_{\log}^{\dagger,r}$ (on renvoie à \cite[Prop. 2.24]{Ber02} pour la construction de l'application $\log$). On définit également un opérateur de monodromie $N$ sur $\Bt_{\log}^\dagger$ par

$$N(\sum_{i=0}^na_i\log[\tilde{\pi}]^i) = -\sum_{i=1}^nia_i\log[\tilde{\pi}]^{i-1}.$$

\begin{prop}
L'opérateur de monodromie $N$ vérifie 
\begin{enumerate}
\item $gN=Ng$ pour tout $g \in \G_K$ ;
\item $N\phi = p\phi N$.
\end{enumerate}
\end{prop}
\begin{proof}
Voir \cite[3.2.3]{fontaine1994corps}.
\end{proof}

On pose $\Bt_{\tau,\log,K}^\dagger:=\Bt_{\tau,\mathrm{rig},K}^\dagger[\log[\tilde{\pi}]]= (\Bt_{\log}^\dagger)^{H_{\tau,K}}$. On définit également un anneau $\B_{\tau,\log,K}^{\dagger}$, en posant $\B_{\tau,\log,K}^{\dagger}=\B_{\tau,\mathrm{rig},K}^{\dagger}[\log[\tilde{\pi}]]$. Cet anneau est stable sous l'action de $\phi$ puisque $\phi(\log[\tilde{\pi}])=p\cdot \log[\tilde{\pi}]$ par définition, et est un sous-anneau de $\Bt_{\log}^{\dagger}$. On définit également un anneau $\B_{\tau,\log,K}^{\dagger,r}$ par $\B_{\tau,\log,K}^{\dagger,r}=\B_{\tau,\mathrm{rig},K}^{\dagger,r}[\log[\tilde{\pi}]]$. Comme le Frobenius induit une bijection de $\Bt_{\mathrm{rig},K}^{\dagger,r}$ sur $\Bt_{\mathrm{rig},K}^{\dagger,pr}$, il induit également une bijection de $\Bt_{\tau,\log,K}^{\dagger,r}$ sur $\Bt_{\tau,\log,K}^{\dagger,pr}$. 

En particulier, $\phi(\Bt_{\tau,\log,K}^{\dagger,r}) \subset \Bt_{\tau,\log,K}^{\dagger,pr}$. On définit alors 
$$\B_{\tau,\log,K,n}^{\dagger,r} = \phi^{-n}(\B_{\tau,\log,K,n}^{\dagger,p^nr})$$
et 
$$\B_{\tau,\log,K,\infty}^{\dagger,r}=\bigcup_{n \geq 0}\B_{\tau,\log,K,n}^{\dagger,r},$$
qui sont des sous-anneaux de $\Bt_{\tau,\log,K}^{\dagger,r}$. On notera respectivement $\B_{\tau,\mathrm{rig},K}^+$, $\B_{\tau,\log,K}^+$ et $\Bt_{\log}^+$ pour $\B_{\tau,\mathrm{rig},K}^{\dagger,0}$, $\B_{\tau,\log,K}^{\dagger,0}$ et $\Bt_{\log}^{\dagger,0}$.

On va maintenant redonner certains résultats concernant les vecteurs localement analytiques et pro-analytiques dans les anneaux de périodes $\Bt_L^I$, ainsi que certaines conséquences de la surconvergence des $(\phi,\tau)$-modules. Pour $I$ un sous-intervalle de $[0;+\infty)$, on définit $\tilde{D}_L^{I}(V)$ par $\tilde{D}_L^{I}(V) = (\Bt^{I} \otimes_\Qp V)^{\G_L}$. Pour $k \geq 1$, on pose $r_k = p^{k-1}(p-1)$.

\begin{theo}
\label{struc loc ana}
Soit $I = [r_\ell;r_k]$ ou $I = [0,r_k]$, et soit $m \geq m_0$ avec $m_0$ comme dans \cite[Lemm. 3.4.2]{GP18}. Alors on a~:
\begin{enumerate}
\item $(\Bt_L^I)^{\tau_{m+k}-an,\gamma=1} \subset \B_{\tau,K,m}^I$ pour $m \geq m_0$ ;
\item $(\Bt_L^I)^{\tau-la,\gamma=1}=\B_{\tau,K,\infty}^I$ ;
\item $(\Bt_{\tau,\mathrm{rig},L}^{\dagger,r_\ell})^{\tau-\pa,\gamma=1} = \B_{\tau,\mathrm{rig},K,\infty}^{\dagger,r_\ell}$.
\end{enumerate} 
\end{theo}
\begin{proof}
Voir \cite[Thm. 3.4.4]{GP18}.
\end{proof}

\begin{prop}
\label{log tau la}
On a $(\Bt_{\tau,\log,K}^{\dagger,r})^{\pa} = \B_{\tau,\log,K,\infty}^{\dagger,r}$.
\end{prop}
\begin{proof}
Le calcul des vecteurs pro-analytiques dans $\Bt_{\tau,\mathrm{rig},K}^{\dagger,r}$ est donné par le théorème \ref{struc loc ana} et nous dit que $(\Bt_{\mathrm{rig},K}^{\dagger,r})^{\pa} =\B_{\tau,\mathrm{rig},K,\infty}^{\dagger,r}$. Remarquons également que $\log[\tilde{\pi}] \in (\Bt_{\tau,\log,K}^{\dagger,r})^{\pa}$, puisque $\log[\tilde{\pi}]$ est invariant sous l'action de $\gamma$ et qu'on a $\tau^k(\log[\tilde{\pi}])=kt+\log[\tilde{\pi}]$. 

Si maintenant $x \in (\Bt_{\tau,\log,K}^{\dagger,r})^{\pa}$, on peut écrire $x = \sum_{k=0}^nx_k\log[\tilde{\pi}]^k$. Comme 
$$N~: \Bt_{\tau,\log,K}^{\dagger,r} \rightarrow \Bt_{\tau,\log,K}^{\dagger,r}$$
est une application $\Bt_{\mathrm{rig},K}^{\dagger,r}$-linéaire continue, le lemme \ref{app linéaire la} montre que $N(x) \in (\Bt_{\tau,\log,K}^{\dagger,r})^{\pa}$. En itérant $n$ fois l'opérateur $N$, on trouve que $x_n \in (\Bt_{\tau,\log,K}^{\dagger,r})^{\pa}$. Comme $\log[\tilde{\pi}] \in (\Bt_{\tau,\log,K}^{\dagger,r})^{\pa}$, on en déduit que $x-x_n\log[\tilde{\pi}]^n \in (\Bt_{\tau,\log,K}^{\dagger,r})^{\pa}$. En appliquant le même résultat à $x-x_n\log[\tilde{\pi}]^n$, on en déduit que $x_{n-1}$ est aussi dans $(\Bt_{\tau,\log,K}^{\dagger,r})^{\pa}$, et par récurrence descendante sur $k \in \left\{0,\cdots,n\right\}$, on en déduit que chacun des $x_k$ appartient à $(\Bt_{\mathrm{rig},K}^{\dagger,r})^{\pa}$. Comme $(\Bt_{\mathrm{rig},K}^{\dagger,r})^{\pa} =\B_{\tau,\mathrm{rig},K,\infty}^{\dagger,r}$, on en déduit le résultat.
\end{proof}

La surconvergence des $(\phi,\tau)$-modules a également plusieurs conséquences intéressantes qui vont nous être utiles par la suite :

\begin{lemm}
\label{D_tau est loc ana}
Soit $V$ une représentation $p$-adique de $\G_K$ et soit $D_{\tau}^{\dagger}(V):=(\B_{\tau}^{\dagger}\otimes V)^{H_K}$ le $\phi$-module surconvergent associé, et soit $r \geq 0$ tel que $D_{\tau}^{\dagger}(V) = \B_{\tau,K}^{\dagger}\otimes_{\B_{\tau,K}^{\dagger,r}}((\B_\tau^{\dagger,r}\otimes_{\Qp}V)^{H_K})$. Alors pour tout intervalle $I$ compact tel que $r \leq \min(I)$, les éléments de $D_{\tau}^\dagger(V)$ et $D_{\tau,\mathrm{rig}}^{\dagger}(V)$, vus comme des éléments de $\tilde{D}_{L}^I(V)$ sont des vecteurs localement analytiques pour $G_\infty$.
\end{lemm}
\begin{proof}
On définit $\tilde{D}_{\mathrm{rig},L}^{\dagger,r}(V)$ par 
$$\tilde{D}_{\mathrm{rig},L}^{\dagger,r}(V) = (\Bt_{\mathrm{rig}}^{\dagger,r} \otimes_\Qp V)^{\G_L}.$$
Dans la démonstration de \cite[Thm. 6.2.6]{GP18}, les modules $D_{\tau}^{\dagger}(V)$ et $D_{\tau,\mathrm{rig}}^{\dagger}(V)$ sont construits comme des sous-objets de $\tilde{D}_{\mathrm{rig},L}^{\dagger}(V)^\pa$, de sorte que 
$$D_{\tau}^{\dagger,r}(V) \subset D_{\tau,\mathrm{rig}}^{\dagger,r}(V) \subset \tilde{D}_{\mathrm{rig},L}^{\dagger,r}(V)^\pa \subset \tilde{D}_L^{[r;s]}(V)^\la$$
pour $s \geq r$, ce qui montre le résultat.
\end{proof}

On va maintenant s'intéresser à des conséquences de la démonstration de la surconvergence des $(\phi,\tau)$-modules et principalement concernant un certain élément $b_\gamma \in \At_L^+$ (c'est le $\mathfrak{t}$ de \cite[Exemple 3.2.3]{Liu08}) qu'on définit par $b_\gamma~:= \frac{t}{p\lambda}$, où $t$ est le $t$ usuel en théorie de Hodge $p$-adique et $\lambda = \prod_{n \geq 0}\phi^n(E([\tilde{\pi}])/E(0)) \in \B_{\tau,\mathrm{rig},K}^+$ est l'élément défini dans \cite[1.1.1]{KisinFiso}. On rappelle que, par \cite[Lemm. 5.1.1]{GP18}, si $I$ est un sous-intervalle de $\R$ avec $\min(I)$ assez grand, alors $b_\gamma \in (\Bt_L^I)^{\la}$. La surconvergence des $(\phi,\tau)$-modules nous permet de montrer des résultats plus précis concernant cet élément, et d'obtenir la conséquence suivante :

\begin{prop}
\label{t dans BtauL}
On a $t \in \lambda \cdot \B_{\tau,L}^\dagger$.
\end{prop}
\begin{proof}
Soit $V=\Qp(-1)$ comme représentation de $\G_K$. La surconvergence des $(\phi,\tau)$-modules montre en particulier que le $(\phi,\tau)$-module associé à $V$ est surconvergent, et donc $(\B_\tau^\dagger\otimes_\Qp V)^{H_{\tau,K}}$ est de dimension $1$. En particulier, $(\B_\tau^\dagger\otimes_\Qp V)^{H_{\tau,K}}$ est engendré par un élément de la forme $z \otimes a \neq 0$, et, quitte à diviser par un élément de $\Q_p^\times$, on peut supposer que $a=1$. Il existe donc $z \in \B_\tau^\dagger$, $z \neq 0$, tel que, pour tout $g \in H_{\tau,K}$, $g(z) = \chi(g)z$. On en déduit donc que $z$ est invariant sous l'action de $K_\infty \cdot K_{\mathrm{cycl}}=L$, et donc $z \in \B_{\tau,L}^\dagger$. On fixe désormais un $r > 0$ tel que $z \in \B_{\tau,L}^{\dagger,r}$ et tel que $1/b_\gamma \in \Bt_L^{\dagger,r}$. Le lemme \ref{D_tau est loc ana} montre que $z \otimes 1$ est pro-analytique pour $\Gal(L/K)$, et donc $z \in (\Bt_{\mathrm{rig},L}^{\dagger,r})^{\pa}$.

Or, si $\gamma$ est un générateur topologique de $\Gal(L/K_\infty)$, on a $\gamma(b_\gamma)=\chi(\gamma)b_\gamma$, de sorte que $z/b_\gamma \in \Bt_L^I$ est invariant sous $\gamma$. De plus, $z$ et $1/b_\gamma$ étant des vecteurs pro-analytiques de $\Bt_{\mathrm{rig},L}^{\dagger,r}$, c'est encore le cas de $z/b_\gamma$. En particulier, on en déduit que $z/b_\gamma \in (\Bt_{\mathrm{rig},L}^{\dagger,r})^{\pa,\gamma=1}$. Or $(\Bt_{\mathrm{rig},L}^{\dagger,r})^{\pa,\gamma=1} = \B_{\tau,\mathrm{rig},K,\infty}^{\dagger,r}$ par la proposition \ref{struc loc ana}, de sorte que $z/b_\gamma \in \B_{\tau,\mathrm{rig},K,\infty}^{\dagger,r}$. 

Il existe donc un entier $n$ tel que $z/b_\gamma \in \phi^{-n}(\B_{\tau,\mathrm{rig},K}^{\dagger,p^nr})$, et donc $\phi^n(z/b_\gamma) \in \B_{\tau,\mathrm{rig},K}^{\dagger,p^nr}$. Mais $z$ et $b_\gamma$ sont des éléments bornés, appartenant à $\Bt_L^\dagger$ et $\Bt_L^\dagger \cap \B_{\tau,\mathrm{rig},K}^{\dagger,p^nr} =\B_{\tau,K}^{\dagger,p^nr}$, de sorte que $\phi^n(z/b_\gamma) \in \B_{\tau,K}^\dagger$. Comme $b_\gamma= \frac{t}{p\lambda}$, on en déduit que $\phi^n(t)=p^nt \in \phi^n(\lambda)\cdot \B_{\tau,L}^\dagger$, et comme $\phi^n(\lambda)= \frac{1}{\prod_{k=0}^{n-1}\phi^k(E([\tilde{\pi}])/E(0)}\cdot \lambda$, on a bien $t \in \lambda \cdot \B_{\tau,L}^\dagger$.
\end{proof}

En particulier, $b_\gamma \in \B_{\tau,L}^\dagger \subset \B_{\tau}^\dagger$, et donc on dispose d'un moyen très simple pour passer du $(\phi,\tau)$-module d'une certaine représentation $p$-adique $V$ de $\G_K$ à sa tordue $V(d)$ pour un certain entier $d$. On rappelle que $V(d)$ est la représentation $p$-adique de $\G_K$ telle que $V(d) = V$ en tant que $\Qp$-espaces vectoriels, et si $x \in V(d)$ et $g \in \G_K$, alors $g\cdot x = \chi_{\mathrm{cycl}}(g)^d(g \cdot x)_V$, où $(g \cdot x)_V$ désigne l'élément $g\cdot x$ dans $V$.

\begin{prop}
\label{tordre rep phitau}
Si $V$ est une représentation $p$-adique de $\G_K$, si $d$ est un entier relatif et si $D(V)$ désigne le $(\phi,\tau)$-module associé à $V$, alors
$$D(V(-d)) = b_\gamma^d \cdot D(V).$$
\end{prop}
\begin{proof}
C'est une conséquence directe de la proposition \ref{t dans BtauL}, puisque si $(e_1,\cdots,e_n)$ est une base du $\phi$-module associé à $V$, alors $(b_\gamma^de_1,\cdots,b_\gamma^de_n)$ est bien une base du $\phi$-module associé à $V(-d)$ puisque $b_\gamma \in \B_\tau$.
\end{proof}

Kisin montre dans \cite{KisinFiso} que les représentations semi-stables à poids de Hodge-Tate négatifs sont de $E$-hauteur finie, et donc \textit{a fortiori} de hauteur finie, c'est-à-dire que si $V$ est semi-stable à poids de Hodge-Tate négatifs, alors le $\phi$-module sous-jacent au $(\phi,\tau)$-module $D_\tau(V)$ est engendré par un $\phi$-module étale sur $\B_{\tau,K}^+$. Les propositions \ref{tordre rep phitau} et \ref{t dans BtauL} montrent donc que, étant donné une représentation semi-stable $V$ quelconque de $\G_K$, il existe $h(V) \geq 0$ tel que $b_\gamma^{h(V)}$ est engendré par un un $(\phi,\tau)$-module étale $D^+$ sur $\B_{\tau,K}^+$. En particulier, si $r_b$ est tel que $b_\gamma^{-1} \in \Bt^{\dagger,r_b}$, alors toute représentation semi-stable est surconvergente de rayon $\leq r_b$.

\section{$(\phi,\tau)$-modules différentiels}
\subsection{Les invariants $D_{\mathrm{cris}}$ et $D_{\mathrm{st}}$}
Soient 
$$D_{\tau,\mathrm{rig}}^{\dagger}(V) = \B_{\tau,\mathrm{rig},K}^{\dagger}\otimes_{\B_{\tau,K}^{\dagger}}D_{\tau}^{\dagger}(V) \textrm{ et } D_{\tau,\log}^{\dagger}(V) = \B_{\tau,\log,K}^{\dagger}\otimes_{\B_{\tau,K}^{\dagger}}D_{\tau}^{\dagger}(V).$$
Le théorème \ref{surconvphitau} montre que $D_{\tau,\mathrm{rig}}^{\dagger}(V)$ et $D_{\tau,\log}^{\dagger}(V)$ sont respectivement des $\B_{\tau,\mathrm{rig},K}^{\dagger}$- et $\B_{\tau,\log,K}^{\dagger}$-modules libres de rang $d = \dim_{\Qp}(V)$. Le but de cette partie est de montrer comment récupérer $D_{\mathrm{cris}}(V)$ et $D_{\mathrm{st}}(V)$ à partir de ces modules. 

Dans le cas particulier où $V$ est semi-stable à poids négatifs, Kisin a montré que $D_\tau^+(V) = (\B_{\tau}^+\otimes_{\Qp}V)^{H_{\tau,K}}$ est un $\B_{\tau,K}^+$-module libre de rang $d$, et on verra qu'on peut récupérer $D_{\mathrm{st}}(V)$ à partir de $D_{\tau,\log}^+(V) = \B_{\tau,\log,K}^+\otimes_{\B_{\tau,K}^+}D_{\tau}^+(V)$.

On va voir besoin pour la suite de certains résultats de la partie 3.2 de \cite{Ber02}, qu'on va maintenant rappeler.

On rappelle que (voir par exemple la discussion précédant \cite[Prop. 3.4]{Ber02}), si on pose $D_{\mathrm{st}}^+(V) = (\B_{\mathrm{st}}^+\otimes_{\Qp}V)^{\G_K}$, alors $D_{\mathrm{st}}^+(V) = (\Bt_{\log}^+\otimes_{\Qp}V)^{\G_K}$. De plus, si $V$ est à poids de Hodge-Tate négatifs, alors $D_{\mathrm{st}}^+(V) = D_{\mathrm{st}}(V)$, et dans le cas général, on a $D_{\mathrm{st}}(V) = t^{-d}D_{\mathrm{st}}^+(V(-d))$ pour $d$ assez grand.

\begin{prop}
\label{log+ = logdagger}
Si $V$ est une représentation $p$-adique, alors $(\Bt_{\log}^{\dagger}\otimes_{\Qp}V)^{\G_K}$ est un $F$-espace vectoriel de dimension finie, et le morphisme induit par l'inclusion de $\Bt_{\log}^+$ dans $\Bt_{\log}^{\dagger}$
$$D_{\mathrm{st}}^+(V) \rightarrow (\Bt_{\log}^{\dagger}\otimes_{\Qp}V)^{\G_K}$$
est un isomorphisme de $(\phi,N)$-modules.
\end{prop}
\begin{proof}
Voir \cite[Prop. 3.4]{Ber02}.
\end{proof}

En particulier, une représentation $V$ à poids négatifs est semi-stable si et seulement si elle est $\Bt_{\log}^{\dagger}$-admissible, et elle est cristalline si et seulement si elle est $\Bt_{\mathrm{rig}}^{\dagger}$-admissible puisque ses périodes sont tuées par $N$. Si les poids de Hodge-Tate de $V$ ne sont pas négatifs, on peut tordre $V$ et on en déduit alors que $V$ est semi-stable si et seulement si elle est $\Bt_{\log}^{\dagger}[1/t]$-admissible et elle est cristalline si et seulement si elle est $\Bt_{\mathrm{rig}}^{\dagger}[1/t]$-admissible.

On se propose maintenant de faire le lien entre $D_{\tau,\mathrm{rig}}^\dagger(V)$ et les invariants $D_{\mathrm{cris}}(V)$ et $D_{\mathrm{st}}(V)$. Pour simplifier les notations, on notera $(D_{\tau,\log}^{\dagger}(V)[1/\lambda])^{\tau=1}$ (resp. $(D_{\tau,\mathrm{rig}}^\dagger(V)[1/\lambda])^{\tau=1}$) les éléments de $D_{\tau,\log}^{\dagger}(V)[1/\lambda]$ (resp. $D_{\tau,\mathrm{rig}}^\dagger(V)[1/\lambda]$) qui sont invariants sous l'action de $\tau$ en tant qu'éléments de $(\Bt_{\log}^\dagger \otimes_{\B_{\tau,\log}^\dagger}D_{\tau,\log}^\dagger(V))[1/\lambda]$ (resp. $(\Bt_{\mathrm{rig}}^\dagger \otimes_{\B_{\tau,\mathrm{rig}}^\dagger}D_{\tau,\mathrm{rig}}^\dagger(V))[1/\lambda]$) via l'identification $x \mapsto 1 \otimes x$. Remarquons en fait que, comme les éléments de $\B_{\tau,\log,K}^\dagger$, de $\B_{\tau,\mathrm{rig},K}^\dagger$ et de $D_\tau^\dagger(V)$ sont invariants sous l'action de $H_{\tau,K}$, les éléments de $(D_{\tau,\log}^{\dagger}(V)[1/\lambda])^{\tau=1}$ (resp. $(D_{\tau,\mathrm{rig}}^\dagger(V)[1/\lambda])^{\tau=1}$) sont, considérés en tant qu'éléments de $(\Bt_{\log}^\dagger \otimes_{\B_{\tau,\log}^\dagger}D_{\tau,\log}^\dagger(V))[1/\lambda]$ (resp. $(\Bt_{\mathrm{rig}}^\dagger \otimes_{\B_{\tau,\mathrm{rig}}^\dagger}D_{\tau,\mathrm{rig}}^\dagger(V))[1/\lambda]$), à la fois invariants par $\tau$ et par $H_{\tau,K}$ donc par $\G_K$.

\begin{theo}
\label{retrouver Dst Dcris phitau}
Si $V$ est une représentation $p$-adique, alors 
$$D_{\mathrm{st}}(V) = (D_{\tau,\log}^{\dagger}(V)[1/\lambda])^{\tau=1} \textrm{ et } D_{\mathrm{cris}}(V) = (D_{\tau,\mathrm{rig}}^\dagger(V)[1/\lambda])^{\tau=1}.$$

Si de plus, $V$ est semi-stable (resp. cristalline) à poids négatifs, alors 
$$D_{\mathrm{st}}(V) = (D_{\tau,\log}^+(V))^{\tau=1} \textrm{ (resp. } D_{\mathrm{cris}}(V) = (D_{\tau,\mathrm{rig}}^+(V))^{\tau=1}).$$
\end{theo}
\begin{proof}
On a $D_{\tau,\log}^{\dagger}(V)[1/\lambda] \subset \Bt_{\log}^{\dagger}[1/\lambda]\otimes_{\Qp}V$ et $D_{\tau,\mathrm{rig}}^{\dagger}(V)[1/\lambda] \subset \Bt_{\mathrm{rig}}^{\dagger}[1/\lambda]\otimes_{\Qp}V$. Comme inverser $\lambda$ dans $\Bt_{\mathrm{rig}}^{\dagger}$ revient à inverser $t$ puisque $\lambda/t \in \Bt^{\dagger}$, on en déduit que $(D_{\tau,\log}^\dagger(V)[1/\lambda])^{\tau=1}=(D_{\tau,\log}^\dagger(V)[1/\lambda])^{\G_K}$ est inclus dans $D_{\mathrm{st}}(V)$, et de même que $(D_{\tau,\mathrm{rig}}^\dagger(V)[1/\lambda])^{\tau=1}=(D_{\tau,\mathrm{rig}}^\dagger(V)[1/\lambda])^{\G_K}$ est inclus dans $D_{\mathrm{cris}}(V)$.

On va maintenant montrer que $D_{\mathrm{st}} \subset (D_{\tau,\log}^{\dagger}(V)[1/\lambda])^{\tau=1}$ et on va également montrer que $D_{\mathrm{cris}} \subset (D_{\tau,\mathrm{rig}}^{\dagger}(V)[1/\lambda])^{\tau=1}$. On va en fait se contenter de démontrer le premier point, le deuxième s'en déduisant en faisant $N=0$. Dans un premier temps, on va supposer que $V$ est à poids négatifs, de sorte que $D_{\mathrm{st}}(V) = (\Bt_{\log}^{\dagger}\otimes_{\Qp}V)^{\G_K}$ par la proposition \ref{log+ = logdagger}. De plus, comme $D_{\tau}^{\dagger}(V)$ a la bonne dimension, on a $(\Bt_{\log}^{\dagger}\otimes_{\Qp}V)^{H_{\tau,K}} = \Bt_{\tau,\log,K}^{\dagger}\otimes_{\B_{\tau,K}^{\dagger}}D_{\tau}^{\dagger}(V)$. On en déduit donc que $D_{\mathrm{st}}(V) = (\Bt_{\tau,\log,K}^{\dagger} \otimes_{\B_{\tau,K}^{\dagger}}D_{\tau}^{\dagger}(V))^{\tau=1}$. Or $\G_K$ agit trivialement sur $D_{\mathrm{st}}(V)$, ce qui fait que $D_{\mathrm{st}}(V)$ est constitué de vecteurs localement analytiques. En particulier, $(\Bt_{\tau,\log,K}^{\dagger} \otimes_{\B_{\tau,K}^{\dagger}}D_{\tau}^{\dagger}(V))^{\tau=1}$ est constitué de vecteurs localement analytiques (ou pro-analytiques) de $\Bt_{\tau,\log,K}^{\dagger} \otimes_{\B_{\tau,K}^{\dagger}}D_{\tau}^{\dagger}(V)$ invariants par $\tau$. Mais le lemme \ref{ladansla} montre que $(\Bt_{\tau,\log,K}^{\dagger} \otimes_{\B_{\tau,K}^{\dagger}}D_{\tau}^{\dagger}(V))^{\pa} = (\Bt_{\tau,\log,K}^{\dagger})^{\pa}\otimes_{\B_{\tau,K}^{\dagger}}D_{\tau}^{\dagger}(V)$. Le calcul des vecteurs pro-analytiques de $\Bt_{\tau,\log,K}^{\dagger}$ effectué en \ref{log tau la} montre que $(\Bt_{\tau,\log,K}^{\dagger,r})^{\pa} = \B_{\tau,\log,K,\infty}^{\dagger,r}$. En particulier, il existe un entier $n$ et une base $(e_1,\cdots,e_d)$ de $D_{\tau}^{\dagger}(V)$ tels que $D_{\mathrm{st}}(V) = (\oplus_{i=1}^d \B_{\tau,\log,K,n}^{\dagger,r}\cdot e_i)^{\tau=1}$. Mais comme $\phi~: D_{\mathrm{st}}(V) \to D_{\mathrm{st}}(V)$ est bijective et commute à l'action de Galois, on en déduit que $\phi^n(D_{\mathrm{st}}(V)) = D_{\mathrm{st}}(V)$ et donc que $D_{\mathrm{st}}(V) = (\oplus_{i=1}^d \B_{\tau,\log,K}^{\dagger,p^nr}\cdot \phi^n(e_i))^{\tau=1}$, ce qui montre que $D_{\mathrm{st}}(V) \subset (D_{\tau,\log}^{\dagger}(V)[1/\lambda])^{\tau=1}$.

Ce qu'on vient de faire montre que, si $d$ est assez grand, on a l'inclusion $D_{\mathrm{st}}(V(-d)) \subset (D_{\tau,\log}^{\dagger}(V(-d))[1/\lambda])^{\tau=1}$. On a $D_{\tau,\log}^{\dagger}(V(-d)) = \frac{t^d}{\lambda^d}D_{\tau,\log}^{\dagger}(V)$ par la proposition \ref{tordre rep phitau} et on a aussi $D_{\mathrm{st}}(V(-d))=t^dD_{\mathrm{st}}(V)$. On a donc $D_{\tau,\log}^{\dagger}(V(-d))[1/\lambda] \subset t^dD_{\tau,\log}^{\dagger}(V)[1/\lambda]$, et donc
$$D_{\mathrm{st}}(V(-d)) \subset (D_{\tau,\log}^{\dagger}(V(-d))[1/\lambda])^{\tau=1} \subset (t^dD_{\tau,\log}^{\dagger}(V)[1/\lambda])^{\tau=1} = t^d(D_{\tau,\log}^{\dagger}(V)[1/\lambda])^{\tau=1}$$
puisque $t$ est invariant sous $\tau$. Comme $D_{\mathrm{st}}(V(-d)) = t^dD_{\mathrm{st}}(V)$, on obtient alors 
$$D_{\mathrm{st}}(V) = t^{-d}D_{\mathrm{st}}(V(-d)) \subset t^{-d}t^d(D_{\tau,\log}^{\dagger}(V)[1/\lambda])^{\tau=1} = (D_{\tau,\log}^{\dagger}(V)[1/\lambda])^{\tau=1}.$$
On en déduit donc le résultat pour $V$ quelconque.

Dans le cas où $V$ est semi-stable à poids négatifs, on sait que $D_\tau^+(V)$ est un $\B_{\tau,K}^+$-module libre de rang $d$ et que $D_{\mathrm{st}}(V) = D_{\mathrm{st}}^+(V)=(\Bt_{\log}^+\otimes_{\Qp}V)^{\G_K}$. Comme $D_{\tau,\log}^+(V) \subset (\Bt_{\log}^+\otimes_{\Qp}V)$, on en déduit en prenant les invariants sous $\G_K$ que $(D_{\tau,\log}^+(V))^{\tau=1}$ est inclus dans $D_{\mathrm{st}}(V)$. Comme $D_{\tau}^+(V)$ a la bonne dimension, on a $(\Bt_{\log}^+\otimes_{\Qp}V)^{H_{\tau,K}} = \Bt_{\tau,\log,K}^+\otimes_{\B_{\tau,K}^+}D_{\tau}^+(V)$. On en déduit donc que $D_{\mathrm{st}}(V) = (\Bt_{\tau,\log,K}^+ \otimes_{\B_{\tau,K}^{+}}D_{\tau}^{+}(V))^{\tau=1}$, et la même preuve que précédemment montre alors que $D_{\mathrm{st}}(V) \subset (D_{\tau,\log}^+(V))^{\tau=1}$. Là encore, le cas cristallin s'en déduit en faisant $N=0$.
\end{proof}

\begin{prop}
\label{iso comparaison phitau}
On a les isomorphismes de comparaison suivants~:
\begin{enumerate}
\item si $V$ est semi-stable, alors 
$$\B_{\tau,\log,K}^{\dagger}[1/\lambda] \otimes_{\B_{\tau,K}^{\dagger}}D_{\tau}^{\dagger}(V) =  \B_{\tau,\log,K}^{\dagger}[1/\lambda] \otimes_F D_{\mathrm{st}}(V) ;$$
\item si $V$ est cristalline, alors 
$$\B_{\tau,\mathrm{rig},K}^{\dagger}[1/\lambda] \otimes_{\B_{\tau,K}^{\dagger}}D_{\tau}^{\dagger}(V) =  \B_{\tau,\mathrm{rig},K}^{\dagger}[1/\lambda] \otimes_F D_{\mathrm{cris}}(V).$$
\end{enumerate}
\end{prop}
\begin{proof}
Là aussi, montrer le cas semi-stable suffit puisqu'on obtient le cas cristallin en faisant $N=0$. Par la proposition \ref{retrouver Dst Dcris phitau}, on a $D_{\mathrm{st}}(V)=(D_{\tau,\log}^\dagger(V)[1/\lambda])^{\tau=1}$, donc \textit{a fortiori} $D_{\mathrm{st}}(V) \subset D_{\tau,\log}^\dagger(V)[1/\lambda] \subset \Bt_{\tau,\log,K}^\dagger[1/\lambda]\otimes_{\B_{\tau,K}^\dagger}D_{\tau}^\dagger(V)$. Réciproquement, \cite[Prop. 3.5]{Ber02} nous dit que comme $V$ est semi-stable, $\Bt_{\log}^\dagger[1/t] \otimes_F D_{\mathrm{st}}(V) = \Bt_{\log}^\dagger[1/t]\otimes_\Qp V$, et donc puisqu'inverser $t$ dans $\Bt_{\log}^\dagger$ revient à inverser $\lambda$ et que $D_{\tau}^\dagger(V) \subset D_{\tau,\log}^\dagger(V) \subset \Bt_{\log}^\dagger[1/\lambda]\otimes_\Qp V$, on en déduit en prenant les invariants sous $H_{\tau,K}$ que $D_{\tau,\log}^\dagger(V) \subset (\Bt_{\log}^\dagger[1/\lambda] \otimes_F D_\mathrm{st}(V))^{H_{\tau,K}} = \Bt_{\tau,\log,K}[1/\lambda]\otimes_F D_{\mathrm{st}}(V)$. En particulier, on a donc
$$\Bt_{\tau,\log,K}^{\dagger}[1/\lambda]\otimes_{\B_{\tau,K}^{\dagger}}D_{\tau}^{\dagger}(V)=\Bt_{\tau,\log,K}^{\dagger}[1/\lambda]\otimes_F D_{\mathrm{st}}(V)$$
et donc, en choisissant des bases $\{d_i\}$ de $D_{\mathrm{st}}(V)$ et $\{e_i\}$ de $D_{\tau}^{\dagger}(V)$, il existe une matrice $A \in \GL_d(\Bt_{\tau,\log,K}^{\dagger}[1/\lambda])$ telle que $(e_i)=A(d_i)$.

En particulier, il existe $n \geq 0$ tel que $\lambda^nA \in \M_d(\Bt_{\tau,\log,K}^\dagger)$, et il existe $m \geq 0$ tel que $\lambda^mA^{-1} \in \M_d(\Bt_{\tau,\log,K}^\dagger)$. On note $(a_{ij})$ les coefficients de $A$ et $(b_{ij})$ ceux de $A^{-1}$. Alors pour tout $i$, 
$$\lambda^ne_i=\sum\limits_j(\lambda^na_{ij})d_j$$
et
$$\lambda^md_i=\sum\limits_j(\lambda^mb_{ij})e_j.$$

Comme les $(e_i)$ et $(d_i)$ sont des vecteurs pro-analytiques pour l'action de $\Gal(L/K)$ et invariants sous $\gamma$ (et $\lambda$ aussi), c'est aussi le cas des $(\lambda^na_{ij})$ et des $(\lambda^mb_{ij})$ par le lemme \ref{anneaula}, de sorte qu'il existe $r \geq 0$ et $\ell \geq 0$ tels que les $(\lambda^na_{ij})$ et les $(\lambda^mb_{ij})$ soient tous dans $\B_{\tau,\log,K,\ell}^{\dagger,r}$ par la proposition \ref{log tau la}. Comme $\phi$ définit un isomorphisme sur $D_{\mathrm{st}}(V)$ et comme $\phi(D_\tau^\dagger(V))$ engendre $D_\tau^\dagger(V)$, les $(\phi^\ell(e_i))$ et les $(\phi^\ell(d_i))$ forment également une base respectivement de $D_{\mathrm{st}}(V)$ et de $D_\tau^\dagger(V)$, et $\phi^\ell(A) \in \GL_d(\B_{\tau,\log,K}^{\dagger})$, ce qui permet de conclure.
\end{proof}

\begin{rema}
\label{rema isocomparaisonst}
Si $V$ est semi-stable à poids négatifs, on peut se demander si on a 
$$\B_{\tau,\log,K}^+ \otimes_{\B_{\tau,K}^+}D_{\tau}^+(V) =  \B_{\tau,\log,K}^+ \otimes_F D_{\mathrm{st}}(V).$$
Dans le cas où $V$ est semi-stable à poids négatifs, on a 
$$\Bt_{\tau,\log,K}^+ \otimes_F D_{\mathrm{st}}(V) = \Bt_{\tau,\log,K}^+ \otimes_{\B_{\tau,K}^+}D_{\tau}^+(V),$$
et donc, en prenant les vecteurs pro-analytiques et en utilisant le lemme \ref{D_tau est loc ana}, on en déduit comme dans la preuve de la proposition précédente que la matrice de passage d'une base de $D_{\mathrm{st}}(V)$ à une base de $D_{\tau}^+(V)$ est dans $\mathrm{GL}_d((\Bt_{\tau,\log,K}^+)^{\pa})$. Malheureusement, contrairement au cas précédent, si $(e_1,\cdots,e_d)$ est une base de $D_{\tau}^+(V)$, alors $(\phi^n(e_1),\cdots,\phi^n(e_d))$ n'a plus de raison d'être une base de $D_{\tau}^+(V)$, puisque la matrice de $\phi$ dans cette base est dans $\mathrm{GL}_d(\B_{\tau,K}) \cap \mathrm{M}_d(\B_{\tau,K}^+)$ mais n'a pas de raison d'être dans $\mathrm{GL}_d(\B_{\tau,K}^+)$. On en déduit donc seulement qu'on a un isomorphisme de comparaison
$$\B_{\tau,\log,K,\infty}^+ \otimes_{\B_{\tau,K}^+}D_{\tau}^+(V) =  \B_{\tau,\log,K,\infty}^+ \otimes_F D_{\mathrm{st}}(V).$$
\end{rema}

\subsection{$(\phi,N_\nabla)$-modules}
Si $V$ est une représentation $p$-adique, on peut associer à son $(\phi,\Gamma)$-module cyclotomique sur l'anneau de Robba $D_{\mathrm{rig}}^{\dagger}(V)$ une connexion provenant de l'action infinitésimale de $\Gamma$, ce que Berger a étudié dans \cite{Ber02}. La théorie des vecteurs localement analytiques rend cette construction directe puisque les éléments de $D_{\mathrm{rig}}^{\dagger}(V)$ sont localement analytiques (voir par exemple \cite[§2.1]{KR09}) et puisque la connexion ainsi associée n'est autre que l'opérateur de dérivation dans la direction cyclotomique $\nabla_{\gamma}~: D_{\mathrm{rig}}^{\dagger}(V) \rightarrow D_{\mathrm{rig}}^{\dagger}(V)$, défini par $\nabla_{\gamma}(x):=\lim\limits_{\gamma \to 1}\frac{(\gamma-1)x}{\log\chi(\gamma)}$ (voir \cite[§4.1]{Ber02} pour plus de détails). C'est une connexion au-dessus de l'opérateur $\nabla_{\gamma}~: \B_{\mathrm{rig},K}^{\dagger} \rightarrow \B_{\mathrm{rig},K}^{\dagger}$, et cet opérateur $\nabla_{\gamma}$ sur $\B_{\mathrm{rig},K}^{\dagger}$ n'est autre que l'opérateur $\frac{d}{du}$ si $K=F$ (on rappelle que $u = [\epsilon]-1$).

En gardant cette vision localement analytique des choses, la définition analogue dans le cas des $(\phi,\tau)$-modules serait de simplement remplacer la connexion $\nabla_{\gamma}$ sur $D_{\mathrm{rig}}^{\dagger}(V)$ dans le cas cyclotomique par la connexion $\nabla_{\tau}$ sur $D_{\tau,\mathrm{rig}}^{\dagger}(V)$ au-dessus de l'opérateur $\nabla_{\tau}~: \B_{\tau,\mathrm{rig},K}^{\dagger} \rightarrow \B_{\tau,\mathrm{rig},K}^\dagger$. Le problème est que l'action de $\tau$ ne se fait qu'après avoir tensorisé au-dessus de $\Bt_{\mathrm{rig},L}^{\dagger}$, et qu'on n'a même pas $\nabla_{\tau}(\B_{\tau,\mathrm{rig},K}^{\dagger}) \subset \B_{\tau,\mathrm{rig},K}^{\dagger}$. On dispose en revanche d'un espace naturel sur lequel l'opérateur $\nabla_{\tau}$ est défini, à savoir $(\Bt_{\mathrm{rig},L}^{\dagger})^{\pa}$, et plus généralement on peut définir une connexion $\nabla_{\tau}$ sur $(\Bt_{\mathrm{rig},L}^{\dagger}\otimes D_{\tau,\mathrm{rig}}^{\dagger}(V))^{\pa} = (\Bt_{\mathrm{rig},L}^{\dagger})^{\pa} \otimes D_{\tau,\mathrm{rig}}^{\dagger}(V)$.

Pour des raisons évidentes, on souhaiterait néanmoins renormaliser l'opérateur $\nabla_{\tau}$ de façon à avoir $\nabla_{\tau}(\B_{\tau,\mathrm{rig},K}^{\dagger}) \subset \B_{\tau,\mathrm{rig},K}^{\dagger}$. On définit donc un nouvel opérateur $N_{\nabla}$ sur $(\Bt_{\mathrm{rig},L}^{\dagger})^{\pa}$, en posant $N_{\nabla}~:= \frac{-\lambda}{t}\nabla_{\tau}$. Remarquons que, comme $\frac{\lambda}{t} \in \Bt_L^{\dagger}$ et est localement analytique d'après \cite[Lemm. 5.1.1]{GP18}, l'opérateur $N_{\nabla}~: (\Bt_{\mathrm{rig},L}^{\dagger})^{\pa} \rightarrow (\Bt_{\mathrm{rig},L}^{\dagger})^{\pa}$ est bien défini, et plus généralement, la connexion $N_{\nabla}~: (\Bt_{\mathrm{rig},L}^{\dagger}\otimes D_{\tau,\mathrm{rig}}^{\dagger}(V))^{\pa} \rightarrow (\Bt_{\mathrm{rig},L}^{\dagger}\otimes D_{\tau,\mathrm{rig}}^{\dagger}(V))^{\pa}$ est bien définie. De plus, comme $\nabla_{\tau}([\tilde{\pi}]) = t[\tilde{\pi}]$ et comme $\lambda \in \B_{\tau,\mathrm{rig},K}^{\dagger}$, on a bien $N_{\nabla}(\B_{\tau,\mathrm{rig},K}^{\dagger}) \subset \B_{\tau,\mathrm{rig},K}^{\dagger}$, et le choix du signe est fait pour que l'opérateur qu'on vient de définir sur $\B_{\tau,\mathrm{rig},K}^{\dagger}$ coïncide avec l'opérateur $N_{\nabla}$ défini par Kisin dans \cite{KisinFiso}, puisqu'on a avec cette définition $N_\nabla([\tilde{\pi}])=-\lambda[\tilde{\pi}]$.

\begin{defi}
On appelle $(\phi,N_\nabla)$-module sur $\B_{\tau,\mathrm{rig},K}^{\dagger}$ un $\B_{\tau,\mathrm{rig},K}^{\dagger}$-module libre $D$ muni d'un Frobenius et d'une connexion $N_{\nabla}~: D \rightarrow D$ au-dessus de $N_{\nabla}~: \B_{\tau,\mathrm{rig},K}^{\dagger} \rightarrow \B_{\tau,\mathrm{rig},K}^{\dagger}$, c'est-à-dire que pour tout $m \in D$ et pour tout $x \in \B_{\tau,\mathrm{rig},K}^{\dagger}$, $N_{\nabla}(x\cdot m) = N_{\nabla}(x)\cdot m +x \cdot N_{\nabla}(m)$.
\end{defi}

\begin{prop}
\label{stabilité connexion}
Si $V$ est une représentation $p$-adique de $\G_K$, alors 
$$N_{\nabla}(D_{\tau,\mathrm{rig}}^{\dagger}(V)) \subset D_{\tau,\mathrm{rig}}^{\dagger}(V).$$
\end{prop}
\begin{proof}
Le lemme \ref{D_tau est loc ana} montre que $D_{\tau,\mathrm{rig}}^\dagger(V)$ est inclus dans $(\tilde{D}_{L,\mathrm{rig}}^\dagger(V))^{\pa}$ qui est stable sous l'action de $\tau$ et donc par la connexion $N_\nabla$. On va commencer par montrer que, si $x \in (\tilde{D}_{L,\mathrm{rig}}^\dagger(V))^{\pa}$ est invariant sous $\Gal(L/K_\infty)$, alors c'est aussi le cas de $N_\nabla(x)$, ce qui nous donnera $N_\nabla(D_{\tau,\mathrm{rig}}^\dagger(V)) \subset (\tilde{D}_{L,\mathrm{rig}}^\dagger(V))^{\pa,\gamma=1}$. Soit donc $x \in (\tilde{D}_{L,\mathrm{rig}}^\dagger(V))^{\pa,\gamma=1}$, et soit $g \in \Gal(L/K_\infty)$. Comme $N_\nabla(x)=-\frac{\lambda}{t}\nabla_\tau(x)$ et que la relation entre $\tau$ et $\Gal(L/K_\infty)$ est donnée par $g\tau=\tau^{\chi_{\mathrm{cycl}}(g)}g$, on en déduit que $g(N_\nabla(x))=-\frac{1}{\chi_{\mathrm{cycl}}(g)}\frac{\lambda}{t}\nabla_{\tau^{\chi_{\mathrm{cycl}}(g)}}(g(x)) = N_\nabla(x)$.

On sait que $(\tilde{D}_{L,\mathrm{rig}}^\dagger(V))^{\pa,\gamma=1} = D_{\tau,\mathrm{rig}}^\dagger(V) \otimes_{\B_{\tau,\mathrm{rig},K}^\dagger}\B_{\tau,\mathrm{rig},K,\infty}^\dagger$ par la proposition \ref{padanspa} et le théorème \ref{struc loc ana}. En particulier, si $(e_1,\cdots,e_d)$ est une base de $D_\tau^\dagger(V)$, il existe $m$ tel que les $(N_\nabla(e_i)$ appartiennent tous à $D_{\tau,\mathrm{rig}}^\dagger(V) \otimes_{\B_{\tau,\mathrm{rig},K}}\phi^{-m}(\B_{\tau,\mathrm{rig},K}^\dagger)$, de sorte que les $\phi^m(N_\nabla(e_i))$ appartiennent tous à $D_{\tau,\mathrm{rig}}^\dagger(V)$. Comme la relation entre $N_\nabla$ et $\phi$ est donnée par $N_\nabla\phi = \frac{E([\tilde{\pi}]}{E(0)}p\phi N_\nabla$, on en déduit que les $N_\nabla(\phi^m(e_i))$ sont tous dans $D_{\tau,\mathrm{rig}}^\dagger(V)$. Comme $(\phi^m(e_1),\cdots,\phi^m(e_d))$ est aussi une base de $D_{\tau}^\dagger(V)$, et donc \textit{a fortiori} aussi une base de $D_{\tau,\mathrm{rig}}^\dagger(V)$ et que $N_\nabla$ stabilise $\B_{\tau,\mathrm{rig},K}^\dagger$, on en déduit le résultat.
\end{proof}

En particulier, si $V$ est une représentation $p$-adique de $\G_K$, la connexion $N_\nabla$ associée à son $(\phi,\tau)$-module différentiel $D_{\tau,\mathrm{rig}}^\dagger(V)$ définit une structure de $(\phi,N_\nabla)$-module.

\begin{prop}
\label{V V' meme Nabla}
Si $V, V'$ sont deux représentations $p$-adiques de $\G_K$, alors $D_{\tau,\mathrm{rig}}^\dagger(V)$ et $D_{\tau,\mathrm{rig}}^\dagger(V')$ définissent le même $(\phi,N_\nabla)$-module si, et seulement si, il existe $n \geq 0$ tel que $V$ et $V'$ sont isomorphes en tant que représentations de $\G_{K_n}$.
\end{prop}
\begin{proof}
Dans cette preuve, on utilisera systématiquement le théorème de Kedlaya \cite[Thm. 6.3.3]{slopes} qui nous dit que tout $\phi$-module étale sur $\B_{\tau,\mathrm{rig},K}^\dagger$ provient par extension des scalaires d'un $\phi$-module étale sur $\B_{\tau,K}^\dagger$ en jonglant entre les modules surconvergents et ceux sur l'anneau de Robba, en gardant à l'esprit que ces données sont équivalentes.

S'il existe $n \geq 0$ tel que $V$ et $V'$ sont isomorphes en tant que représentations de $\G_{K_n}$, alors $V$ et $V'$ sont \textit{a fortiori} isomorphes en tant que représentations de $H_{\tau,K}$, de sorte que $D_{\tau,\mathrm{rig}}^\dagger(V)=D_{\tau,\mathrm{rig}}^\dagger(V')$ en tant que $\phi$-modules. Comme de plus, $V$ et $V'$ sont isomorphes en tant que représentations de $\G_{K_n}$, l'action de $\tau^{p^n}$ est la même sur $(\Bt^\dagger \otimes_{\Qp}V)^{\G_L}$ et $(\Bt^\dagger \otimes_{\Qp}V')^{\G_L}$, de sorte que $D_\tau^\dagger(V)$, $D_\tau^\dagger(V')$ sont isomorphes en tant que $(\phi,\tau^{p^n})$-modules sur $(\B_{\tau,K}^\dagger,\Bt_L^\dagger)$. Par définition de $N_\nabla$, on en déduit que $D_{\tau,\mathrm{rig}}^\dagger(V) = D_{\tau,\mathrm{rig}}^\dagger(V')$ en tant que $(\phi,N_\nabla)$-modules sur $\B_{\tau,\mathrm{rig},K}^\dagger$.

Réciproquement, si $V$ et $V'$ sont deux représentations telles que $D_{\tau,\mathrm{rig}}^\dagger(V) = D_{\tau,\mathrm{rig}}^\dagger(V')$ en tant que $(\phi,N_\nabla)$-modules sur $\B_{\tau,\mathrm{rig},K}^\dagger$, il suffit clairement de montrer qu'il existe $n \geq 0$ tel que $D_{\tau,\mathrm{rig}}^\dagger(V) = D_{\tau,\mathrm{rig}}^\dagger(V')$ en tant que $(\phi,\tau^{p^n})$-modules sur $(\B_{\tau,\mathrm{rig},K}^\dagger,\Bt_{\mathrm{rig},L}^\dagger)$. Soient donc $r \geq 0$ et $(e_1,\cdots,e_d)$ des éléments de $D_\tau^{\dagger,r}(V) \cap D_\tau^{\dagger,r}(V')$ tels que $(e_1\cdots,e_d)$ soit une base du $\phi$-module $D_\tau^\dagger(V)=D_\tau^\dagger(V')$. Soit $s \geq r$ et soit $I=[r;s]$. Par le lemme \ref{D_tau est loc ana}, les $e_i$ sont des vecteurs localement analytiques de $\tilde{D}_L^I(V)=\Bt_L^I \otimes_{\B_\tau^{\dagger,r}}D_\tau^{\dagger,r}(V)$ et de $\tilde{D}_L^I(V')=\Bt_L^I \otimes_{\B_\tau^{\dagger,r}}D_\tau^{\dagger,r}(V')$ (remarquons qu'on a $\tilde{D}_L^I(V)=\tilde{D}_L^I(V')$ en tant que $\phi$-modules par hypothèse). En particulier, il existe $n \geq 0$ tel que, pour tout $i$, $\exp(p^n\frac{t}{\lambda}N_\nabla)(e_i)$ converge dans $(\tilde{D}_L^I(V))^{\la}$ et $(\tilde{D}_L^I(V'))^{\la}$, de sorte que, pour tout $i$, l'action de $\tau^{p^n}$ sur $e_i$ en tant qu'élément de $\tilde{D}_L^I(V)$ est la même qu'en tant qu'élément de $\tilde{D}_L^I(V')$. Comme on a une injection $\tilde{D}_L^{\dagger,r}(V) \to \tilde{D}_L^I(V)$ (c'est une conséquence directe de \cite[Lemm. 2.7]{Ber02}), on en déduit que l'action de $\tau^{p^n}$ sur $\tilde{D}_{\mathrm{rig},L}^{\dagger}(V)$ et $\tilde{D}_{\mathrm{rig},L}^{\dagger}(V')$ coïncide, d'où le résultat.
\end{proof}

\begin{rema}
La preuve de la proposition \ref{V V' meme Nabla} montre en fait que deux $(\phi,\tau)$-modules $D,D'$ sur $\B_{\tau,K,\mathrm{rig}}^\dagger$ définissent le même $(\phi,N_\nabla)$-module si, et seulement si, il existe $n \geq 0$ tel que $D$ et $D'$ sont isomorphes en tant que $(\phi,\tau^{p^n})$-modules, l'énoncé de la proposition correspondant au cas particulier où $D$ et $D'$ sont étales.
\end{rema}

On va maintenant montrer comment associer à tout $(\phi,N_\nabla)$-module un $(\phi,\tau^{p^n})$-module pour $n$ assez grand. 

\begin{prop}
Soit $D$ un $(\phi,N_\nabla)$-module sur $\B_{\tau,\mathrm{rig},K}^\dagger$. Alors il existe $n \geq 0$ et $D'$ un $(\phi,\tau^{p^n})$-module sur $(\B_{\tau,\mathrm{rig},K}^\dagger,\Bt_{\mathrm{rig},L}^\dagger)$ tels que $D = D'$ en tant que $(\phi,N_\nabla)$-modules.
\end{prop}
\begin{proof}
Je tiens à remercier Laurent Berger pour m'avoir fait remarquer qu'une telle construction était possible en s'inspirant de \cite[Coro. 4.3]{berger2012multivariable}. Soit $D$ un $(\phi,N_\nabla)$-module sur $\B_{\tau,\mathrm{rig},K}^\dagger$ et soit $r \geq 0$ tel que $D$ soit défini sur $\B_{\tau,\mathrm{rig},K}^{\dagger,r}$. Soit $I = [r;s]$ avec $s \geq pr$, de sorte que $I \cap pI \neq \emptyset$. On pose pour $k \geq 0$, $I_k = p^kI$ et $J_k = I_k \cap I_{k+1}$. Par \cite[Thm. 2.8.4]{slopes}, la donnée de $D^r$ est équivalente à la donnée d'un fibré vectoriel $\left\{E_k\right\}_{k \geq 0}$, où les $E_k$ sont des $\B_{\tau,K}^{I_k}$-modules libres tels que 
$$E_{k+1} \otimes_{\B_{\tau,K}^{I_{k+1}}}\B_{\tau,K}^{I_k} = E_k,$$
et il existe $(e_1,\cdots,e_d) \in \bigcap_{k \geq 0} E_k$ tels que $E_k = \bigoplus_{i=1}^d \B_{\tau,K}^{I_k}\cdot e_i$. 

L'opérateur $N_\nabla$ sur $D$ induit un opérateur sur $E_0 = \B_{\tau,K}^{I} \otimes_{\B_{\tau,\mathrm{rig},K}^{\dagger,r}}D^r$, donné par $-[\tilde{\pi}]\lambda \frac{d}{d[\tilde{\pi}]} \otimes N_\nabla$, qu'on notera toujours $N_\nabla$. Cet opérateur est continu pour la topologie donnée par $V(\cdot,I)$, et donc il existe $n \geq 0$ tel que, pour tout $i \in \{1,\cdots,d\}$, la série donnée par
$$\exp(\frac{-t}{\lambda}p^nN_\nabla(e_i))$$
converge dans $\Bt_L^I$ vers un élément qu'on notera $T_{n,0}(e_i)$. Les $e_i$ formant une base de $E_0$ sur $\B_{\tau,K}^I$, ils forment par extension des scalaires une base de $\tilde{D}_L^I:=\Bt_L^I \otimes_{\B_{\tau,K}^I}E_0 = \Bt_L^I\otimes_{\B_{\tau,\mathrm{rig},K}^{\dagger,r}}D^r$ sur $\Bt_L^I$. On rappelle que $\Bt_L^I$ est muni d'une action continue de $\G_K$, et on étend $T_{n,0}$ en un opérateur $\tau^{p^n}$-semi-linéaire de $\Bt_L^I$ en posant $T_n(x\cdot e_i) = \tau^{p^n}(x)T_{n,0}(e_i)$. Pour $k \geq 0$, on note $E_k' = \Bt_L^{I_k} \otimes_{\B_{\tau,K}^{I_k}}E_k$.

Pour tout $k \geq 0$, le Frobenius induit des bijections $\phi^k : E_k \to E_{k+1}$ et $\phi_k : E_k' \to E_{k+1}'$. Pour $x \in E_k'$, on pose $T_{n,k}(x) = \phi^k \circ T_n(\phi^{-k}(x))$, et on étend $T_{n,k}$ à $\Bt_L^{J_k} \otimes_{\Bt_L^{I_k}}E_{k+1}'$ par semi-$\tau^{p^n}$-linéarité. Comme $\phi$ commute à l'action de $\G_K$ et avec $\frac{-t}{\lambda}N_\nabla$, on en déduit que $T_{n,k+1}(x)=T_{n,k}(x)$ sur 
$$E_{k+1}'\otimes_{\Bt_L^{I_{k+1}}}\Bt_L^{J_k}=E_{k}'\otimes_{\Bt_L^{I_{k}}}\Bt_L^{J_k},$$
ce qui montre que la collection d'opérateurs $\left\{T_{n,k}\right\}_{k \geq 0}$ est compatible à la structure de fibré vectoriel sur $\Bt_{\mathrm{rig},L}^{\dagger,r}$ donnée par la collection $\left\{E_k'\right\}_{k \geq 0}$, c'est-à-dire que les opérateurs $\left\{T_{n,k}\right\}_{k \geq 0}$ se recollent en un opérateur $\tau^{p^n}$-semi-linéaire $T_n$ sur $\Bt_{\mathrm{rig},L}^{\dagger,r}\otimes_{\B_{\tau,\mathrm{rig},K}^{\dagger,r}}D^r$, ce qui confère à $D$ une structure de $(\phi,\tau^{p^n})$-module sur $(\B_{\tau,\mathrm{rig},K}^\dagger,\Bt_{\mathrm{rig},L}^\dagger)$. 

Cette structure de $(\phi,\tau^{p^n})$-module sur $(\B_{\tau,\mathrm{rig},K}^\dagger,\Bt_{\mathrm{rig},L}^\dagger)$ permet de lui associer un $(\phi,N_\nabla)$-module $D'$ sur $\B_{\tau,\mathrm{rig},K}^\dagger$ tel que les $\phi$-modules sous-jacents sont les mêmes (puisqu'on a simplement ajouté une structure additionnelle au $\phi$-module). L'action de $\tau^{p^n}$ sur $\tilde{D'}_L^I=\Bt_L^I \otimes_{\B_{\tau,\mathrm{rig},K}}^{\dagger,r}{D'}^r$ permet de reconstruire un opérateur $N_\nabla$ sur $\tilde{D'}_L^I$, qui coïncide par construction avec l'opérateur $N_\nabla$ de départ sur $D^I$, de sorte que les opérateurs $N_\nabla$ sont les mêmes.
\end{proof}

En particulier, si $D$ est un $(\phi,N_\nabla)$-module étale, la proposition précédente permet de munir $D$ d'une structure de $(\phi,\tau^{p^n})$-module étale, et donc de lui associer une représentation $p$-adique $V_n$ de $\G_{K_n}$. Si on considère $W = \mathrm{ind}_{\G_{K_n}}^{\G_K}V_n$, alors $W$ est une représentation de $\G_K$ qui contient une sous-représentation $V$ dont la restriction à $\G_{K_n}$ est isomorphe à $V_n$, de sorte que $D_{\tau,\mathrm{rig}}^\dagger(V) \simeq D$ en tant que $(\phi,N_\nabla)$-modules sur $\B_{\tau,\mathrm{rig},K}^\dagger$. Cela montre que travailler avec la catégorie des $(\phi,N_\nabla)$-modules étales revient à travailler avec la catégorie des représentations $p$-adiques de $\G_K$ quotientée par la relation d'équivalence ``$V \equiv V'$ si et seulement si $\exists n \geq 0$ tel que $V_{|\G_{K_n}} \simeq V'_{|\G_{K_n}}$''.

\section{Caractérisation des représentations potentiellement semi-stables}
\subsection{Rappels sur les $(\phi,N,\G_{M/K})$-modules filtrés}
Dans ce qui suit, $M$ désigne une extension galoisienne finie de $K$, et on notera $\G_{M/K}$ le groupe de Galois de $M/K$. On note $M_0$ l'extension maximale non ramifiée de $F$ dans $M$, c'est-à-dire $M_0 = W(k_M)[1/p]$.

\begin{defi}
Un $(\phi,N,\G_{M/K})$-module est un $M_0$-espace vectoriel $D$ de dimension finie et muni d'une application $\sigma$-semi-linéaire injective $\phi~: D \to D$ tel que $N\phi = p\phi N$ et d'une action semi-linéaire de $\G_{M/K}$ qui commute à $\phi$ et $N$. 
\end{defi}

\begin{defi}
\label{defi phiN modules}
Un $(\phi,N,\G_{M/K})$-module filtré est la donnée d'un $(\phi,N,\G_{M/K})$-module $D$ et d'une filtration décroissante, exhaustive et séparée $\Fil^iD_M$ sur $D_M = M \otimes_{M_0}D$ par des sous-$M$-espaces vectoriels stables par $\G_{M/K}$. 
\end{defi}
\begin{rema}
Il est équivalent de se donner une filtration sur $D_K = D_M^{\G_{M/K}}$.
\end{rema}

\begin{defi}
On définit si $D$ est un $(\phi,N,\G_{M/K})$-module filtré de dimension $1$, $t_H(D)$ comme le plus grand entier $i$ tel que $\Fil^i D_M \neq 0$, et si $y \in M_0^{\times}$ est une base de $\phi$ dans une certaine base, alors $v_p(y)$ ne dépend pas du choix de la base, et on définit $t_N(D) = v_p(y)$. En dimension supérieure, on définit $t_N(D)~:= t_N(\det D)$ et $t_H(D):= t_H(\det D)$.

On dira qu'un $(\phi,N,\G_{M/K})$-module filtré est admissible si $t_H(D) = t_N(D)$ et si pour tout sous-$(\phi,N,\G_{M/K})$-module de $D$, muni de la filtration induite, on a $t_H(D') \leq t_N(D')$.
\end{defi}

\begin{rema}
Cette définition d'admissibilité correspond en fait à ce que Fontaine appelait en fait dans  \cite[Déf. 4.4.3]{fontaine1994representations} faiblement admissible. Cependant, depuis la démonstration de l'équivalence \guillemotleft{ }faiblement admissible implique admissible \guillemotright{ } dans \cite{CF00}, il semble plus naturel de définir admissible de cette façon, sans rentrer en conflit avec la définition de $(\phi,N)$-module filtré admissible de Fontaine dans \cite[5.3.3]{fontaine1994representations}.
\end{rema}

\begin{prop}
La catégorie des $(\phi,N,\G_{M/K})$-modules filtrés admissibles est une sous-catégorie pleine de la catégorie des $(\phi,N,\G_{M/K})$-modules filtrés, et elle est de plus abélienne.
\end{prop}
\begin{proof}
Voir \cite[§4]{fontaine1994representations}.
\end{proof}

\begin{theo}
\item Si $V$ est une représentation $p$-adique de $\G_K$ dont la restriction à $\G_M$ est semi-stable, alors $D_{st,M}(V):=(\B_{\mathrm{st}}\otimes_{\Qp}V)^{\G_L}$ est un $(\phi,N,\G_{M/K})$-module filtré admissible sur $K$.
\item Le foncteur $D_{st,L}$, de la catégorie des représentations $p$-adiques de $\G_K$ dont la restriction à $\G_M$ est semi-stable dans celle des $(\phi,N,\G_{M/K})$-modules filtrés admissibles, est une équivalence de catégories.
\end{theo}
\begin{proof}
Voir \cite[Prop. 5.3.6]{fontaine1994representations}.
\end{proof}

\subsection{$(\phi,N)$-modules filtrés et théorie de Kisin}
On va rappeler dans cette section comment Kisin dans \cite{KisinFiso} construit des $(\phi,N_{\nabla})$-modules sur $\B_{\tau,\mathrm{rig},K}^+$ associés à des $(\phi,N)$-modules filtrés effectifs.

On définit $\tilde{X}_n$ comme le complété de $K_n\otimes_{\O_F}\A_{\tau,K}^+$ pour la topologie $([\tilde{\pi}]-\pi_n)$-adique. L'anneau $\tilde{X}_n$ est muni de sa $([\tilde{\pi}]-\pi_n)$-filtration, qui s'étend en une filtration sur le corps des fractions de $\tilde{X}_n$ qu'on notera $\tilde{Y}_n:=\tilde{X}_n[1/([\tilde{\pi}]-\pi_n)]$.

On dispose pour tout $n \in \N$ d'applications naturelles $\B_{\tau,K}^+ \to \B_{\tau,\mathrm{rig},K}^+ \to \tilde{X}_n$ données par $u \mapsto 1\otimes u$, et on peut étendre cette dernière application à $\B_{\tau,\log,K}^+$ en envoyant $\log[\tilde{\pi}]$ sur
$$\log\left(\left(\frac{[\tilde{\pi}]-\pi_n}{\pi_n}\right)+1\right):=\sum_{i=1}^{\infty}\frac{(-1)^{i-1}}{i}\left(\frac{[\tilde{\pi}]-\pi_n}{\pi_n}\right)^i \in \tilde{X}_n.$$

L'anneau $\A_{\tau,K}^+$ est muni d'un Frobenius $\phi$ qui envoie $[\tilde{\pi}]$ sur $[\tilde{\pi}]^p$ et agit comme le Frobenius absolu sur $\O_F$. On notera $\phi_F~: \A_{\tau,K}^+$ l'application $\Zp[\![[\tilde{\pi}]]\!]$-linéaire qui agit sur $\O_F$ via le Frobenius absolu de $\O_F$, et $\phi_{\pi}~: \A_{\tau,K}^+ \to \A_{\tau,K}^+$ l'application $\O_F$-linéaire envoyant $[\tilde{\pi}]$ sur $[\tilde{\pi}]^p$. Ces applications induisent des applications $\phi_F$ et $\phi_\pi$ sur $\B_{\tau,K}^I$ pour tout $I$ et sur $\B_{\tau,\mathrm{rig},K}^{\dagger}$. On récupère bien sûr $\phi$ via $\phi = \phi_F \circ \phi_\pi$ sur chacun de ces anneaux.

Toutes les définitions de ces objets sont faites par Kisin dans \cite[§1.1.1]{KisinFiso}, mais on en a changé les notations ici. On définit également la notion de $\phi$-modules et représentations de $E$-hauteur finie.

\begin{defi}
On dit qu'un $\phi$-module $\mathfrak{M}$ sur $\B_{\tau,\mathrm{rig},K}^+$ est de $E$-hauteur finie si le conoyau de l'application $\B_{\tau,\mathrm{rig},K}^+$-linéaire $\phi : \phi^*\mathfrak{M} \to \mathfrak{M}$ est tuée par une puissance de $E([\tilde{\pi}])$, et on dit qu'un $(\phi,N_\nabla)$-module sur $\B_{\tau,\mathrm{rig},K}^+$ est de $E$-hauteur finie s'il l'est en tant que $\phi$-module.

On dit qu'une représentation $V$ est de $E$-hauteur finie s'il existe dans $D_{\tau,\mathrm{rig}}^\dagger(V)$ un $\phi$-module libre sur $\B_{\tau,\mathrm{rig},K}^+$, de pente $0$ et qui est de $E$-hauteur finie.
\end{defi}

\begin{lemm}
\label{ForsterKisin}
Soit $\mathcal{M}$ un $\B_{\tau,K}^I$-module libre de type fini et soit $\mathcal{N} \subset \cal{M}$ un sous-module. Les assertions suivantes sont équivalentes~:
\begin{enumerate}
\item $\mathcal{N} \subset \cal{M}$ est fermé ;
\item $\mathcal{N}$ est de type fini ;
\item $\mathcal{N}$ est libre de type fini.
\end{enumerate}
\end{lemm}
\begin{proof}
Voir \cite[Lemme 1.1.5]{KisinFiso}.
\end{proof}

Soit maintenant $D$ un $(\phi,N)$-module filtré effectif. On rappelle que, suivant la définition \ref{defi phiN modules}, la filtration de $D$ est donnée sur $D_K= K \otimes_FD$. Pour $n \geq 0$, on définit $\iota_n$ l'application composée~:
$$\B_{\tau,\log,K}^+\otimes_F D \overset{\phi_F^{-n}\otimes\phi^{-n}}{\longrightarrow} \B_{\tau,\log,K}^+ \otimes_F D \longrightarrow \tilde{X}_n \otimes_F D = \tilde{X}_n \otimes_K D_K$$
qu'on étend en une application
$$\iota_n~: \B_{\tau,\log,K}^+[1/\lambda]\otimes_F D \longrightarrow \tilde{Y}_n \otimes_KD_K.$$

On définit maintenant 
$$\mathfrak{M}(D):=\left\{x \in (\B_{\tau,\log,K}^+[1/\lambda]\otimes_FD)^{N=0}~: \iota_n(x) \in \Fil^0(\tilde{Y}_n\otimes_KD_K) \textrm{ pour tout } n \geq 0\right\}.$$
On remarque au passage que $(\B_{\tau,\log,K}^+[1/\lambda]\otimes_FD)^{N=0}$ est un $\B_{\tau,\mathrm{rig},K}^+$-module muni d'un Frobenius semi-linéaire déduit de ceux sur $D$ et $\B_{\tau,\log,K}^+[1/\lambda]$, et d'un opérateur différentiel $N_{\nabla}$ induit par $N_{\nabla}\otimes 1$ sur $\B_{\tau,\log,K}^+[1/\lambda]\otimes_FD$.

\begin{lemm}
\label{lemme Kisin strucFil}
Si on munit $\tilde{X}_n$ d'une structure de $\B_{\tau,\mathrm{rig},K}^+$-module via $\phi_F^{-n}$, alors~:
\begin{enumerate}
\item l'application $\tilde{X}_n \otimes_{\B_{\tau,\mathrm{rig},K}^+}(\B_{\tau,\log,K}^+\otimes_FD)^{N=0} \longrightarrow \tilde{X}_n \otimes_KD_K$ induite par $\iota_n$ est un isomorphisme ;
\item On a
$$\tilde{X}_n\otimes_{\B_{\tau,\mathrm{rig},K}^+}\mathfrak{M}(D) \tilde{\longrightarrow} \sum_{j \geq 0}([\tilde{\pi}]-\pi_n)^{-j}\tilde{X}_n\otimes_K\Fil^jD_K$$
et $([\tilde{\pi}]-\pi_n)^{-j}\tilde{X}_n\otimes_K\Fil^jD_K = \phi_\pi^n(E([\tilde{\pi}]))^{-j}\tilde{X}_n\otimes_K\Fil^jD_K$.
\end{enumerate}
\end{lemm}
\begin{proof}
Voir \cite[Lemme 1.2.1]{KisinFiso}.
\end{proof}

\begin{lemm}
\label{phiNeffphinab}
Les opérateurs $\phi$ et $N_{\nabla}$ sur $(\B_{\tau,\log,K}^+[1/\lambda]\otimes_FD)^{N=0}$ induisent sur $\mathfrak{M}(D)$ une structure de $(\phi,N_\nabla)$-module sur $\B_{\tau,\mathrm{rig},K}^+$. De plus, on a un isomorphisme de $\B_{\tau,\mathrm{rig},K}^+$-modules 
$$\coker(1\otimes\phi~: \phi^*\mathfrak{M}(D)\to \mathfrak{M}(D)) \tilde{\rightarrow}\oplus_{i \geq 0}(\B_{\tau,\mathrm{rig},K}^+/E([\tilde{\pi}])^i)^{h_i},$$
où $h_i=\dim_K\mathrm{gr}^iD_K$.
\end{lemm}
\begin{proof}
Voir \cite[Lemme 1.2.2]{KisinFiso}.
\end{proof}

Si $\mathfrak{M}$ est un $\phi$-module de $E$-hauteur finie, Kisin définit dans \cite[§1.2.5]{KisinFiso} un $\phi$-module filtré associé, noté $D(\mathfrak{M})$, de la façon suivante~: le $F$-espace vectoriel sous-jacent de $D(\mathfrak{M})$ est $\mathfrak{M}/[\tilde{\pi}]\mathfrak{M}$, et l'opérateur $\phi$ est induit par celui de $\mathfrak{M}$. Il définit ensuite une filtration décroissante sur $\phi^*\mathfrak{M}$ par~:
$$\Fil^i\phi^*\mathfrak{M}:=\left\{x \in \phi^*\mathfrak{M}~: 1\otimes\phi(x) \in E([\tilde{\pi}])^i\mathfrak{M}\right\}.$$
Le lemme \ref{ForsterKisin} montre que c'est une filtration par des $\B_{\tau,\mathrm{rig},K}^+$-modules libres de type fini, dont les parties graduées successives sont des modules de $E([\tilde{\pi}])$-torsion. Par transport de structure, cela définit une filtration sur $(1\otimes\phi)(\phi*\mathfrak{M})$. Kisin montre ensuite dans \cite[§1.2.7]{KisinFiso} comment construire à partir de cette filtration une filtration sur $D(\mathfrak{M})_K$. Si maintenant, $\mathfrak{M}$ est un $(\phi,N_\nabla)$-module sur $\B_{\tau,\mathrm{rig},K}^+$ de $E$-hauteur finie, on munit $D(\mathfrak{M})$ d'un opérateur $K_0$-linéaire $N$, en réduisant l'opérateur $N_\nabla$ modulo $[\tilde{\pi}]$, ce qui munit finalement $D(\mathfrak{M})$ d'une structure de $(\phi,N)$-module filtré. Kisin montre alors le théorème suivant~:

\begin{theo}
\label{thm pcpal de Kisin}
Les foncteurs $D \mapsto \mathfrak{M}(D)$ et $\mathfrak{M} \mapsto D(\mathfrak{M})$ sont exacts et quasi-inverses l'un de l'autre et induisent une équivalence de catégories tannakiennes entre les $(\phi,N)$-modules filtrés effectifs et la catégorie des $(\phi,N_\nabla)$-modules de $E$-hauteur finie sur $\B_{\tau,\mathrm{rig},K}^+$.
\end{theo}
\begin{proof}
Voir \cite[Thm. 1.2.15]{KisinFiso}.
\end{proof}

\subsection{Construction de $(\phi,\tau)$-modules à connexion}
Dans cette section, on va généraliser les constructions de Kisin dans \cite{KisinFiso} en adaptant les méthodes de Berger dans \cite{Ber08}.

On commence par définir la notion de $(\phi,\tau_M)$-modules sur $(\B_{\tau,\mathrm{rig},M}^\dagger,\Bt_{\mathrm{rig},L_M}^\dagger)$~:

\begin{defi}
On appelle $(\phi,\tau_M)$-module sur $(\B_{\tau,\mathrm{rig},M}^\dagger,\Bt_{\mathrm{rig},L_M}^\dagger)$ la donnée d'un triplet $(D,\phi_D,\Gal(L_M/M))$, où~:
\begin{enumerate}
\item $(D,\phi_D)$ est un $\phi$-module sur $\B_{\tau,\mathrm{rig},M}^\dagger$ ;
\item $\Gal(L_M/M)$ est une action $\Gal(L_M/M)$-semi-linéaire sur $\Bt_{\mathrm{rig},L_M}^\dagger$ de $\Gal(L_M/M)$ sur $\hat{D}:=\Bt_{\mathrm{rig},L_M}^\dagger \otimes_{\B_{\tau,\mathrm{rig},M}^\dagger}D$ telle que cette action commute à $\phi:= \phi_{\Bt_{\mathrm{rig},L_M}^\dagger}\otimes \phi_D$ ;
\item en tant que sous $\B_{\tau,\mathrm{rig},M}^\dagger$-module de $\hat{D}$, on a $D \subset \hat{D}^{H_{\tau,M}}$.
\end{enumerate}
\end{defi}

\begin{defi}
\label{phi tauM action GMK rig}
Si $M/K$ est une extension galoisienne finie telle que $M$ et $K_\infty$ soient linéairement disjointes au-dessus de $K$, et si $D=(D,\phi_D,\Gal(L_M/M))$ est un $(\phi,\tau_M)$-module sur $(\B_{\tau,\mathrm{rig},M}^\dagger,\Bt_{\mathrm{rig},L_M}^\dagger)$, on dit que $D$ est muni d'une action de $\Gal(M/K)$ si $\G_K$ agit sur $\hat{D}$ et si~:
\begin{enumerate}
\item $D$ en tant que sous ensemble de $\hat{D}$ est stable sous l'action de $H_K \subset \G_K$ ;
\item $\G_{L_M}$ agit trivialement sur $\hat{D}$, et $H_{\tau,M}$ agit trivialement sur $D$ ;
\item l'action de $\G_M/\G_{L_M} \subset \G_K/\G_L$ coïncide avec l'action de $\Gal(L_M/M)$ sur $\hat{D}$.
\end{enumerate}
\end{defi}

On dispose d'un résultat de descente galoisienne pour les $(\phi,\tau_M)$-modules~:
\begin{prop}
\label{descente galoisienne pour phitauMrig}
Si $D=(D,\phi_D,\Gal(L_M/M))$ est un $(\phi,\tau_M)$-module sur $\B_{\tau,\mathrm{rig},L_M}^\dagger$ muni d'une action de $\Gal(M/K)$, avec $M/K$ une extension galoisienne finie, alors $D^{H_{\tau,K}}$ est un $(\phi,\tau)$-module sur $(\B_{\tau,\mathrm{rig},K}^\dagger,\Bt_{\mathrm{rig},L}^\dagger)$ et $D=\B_{\tau,\mathrm{rig},M}^\dagger \otimes_{\B_{\tau,\mathrm{rig},K}^\dagger}D^{H_{\tau,K}}$.
\end{prop}
\begin{proof}
La démonstration est similaire à celle de \cite[Prop. I.3.2]{Ber08}.

Le fait que $D=\B_{\tau,\mathrm{rig},M}^\dagger\otimes_{\B_{\tau,\mathrm{rig},K}}D^{H_{\tau,K}}$ en tant que $\phi$-modules suit de \cite[Lemm. 2.6]{kedlayamonodromie}. On va maintenant montrer que $\phi^*(D^{H_{\tau,K}}) = D^{H_{\tau,K}}$. Soit $(y_i)$ une base de $D$ sur $\B_{\tau,\mathrm{rig},M}^\dagger$ contenue dans $D^{H_{\tau,K}}$, et soit $x \in D^{H_{\tau,K}}$. On peut alors écrire $x = \sum_{i=1}^dx_i\phi(y_i)$ avec $x_i \in \B_{\tau,\mathrm{rig},M}^\dagger$, et comme les $y_i$ et $x$ sont fixés par $H_{\tau,K}$, on a également que les $x_i \in \B_{\tau,\mathrm{rig},K}^\dagger$, et donc $x \in \phi^*(D^{H_{\tau,K}})$.

Il reste à montrer que $\Bt_{\mathrm{rig},L}^\dagger \otimes_{\B_{\tau,\mathrm{rig},K}^\dagger}D^{H_{\tau,K}}$ est bien muni d'une action de $\Gal(L/K)$ vérifiant les bonnes conditions, ce qui découle de la définition \ref{phi tauM action GMK rig}.
\end{proof}

On pose, pour $n \geq 0$, $X_n$ le complété de $K_n \otimes_{\O_F}\O_F[\![\phi^{-n}([\tilde{\pi}])]\!]$ pour la topologie $(\phi^{-n}([\tilde{\pi}])-\pi_n)$-adique, et $Y_n:=X_n[1/(\phi^{-n}([\tilde{\pi}])-\pi_n)]$ son corps des fractions. On dispose évidemment d'applications $\phi_\pi^n~: X_n \to \tilde{X}_n$ et $\phi_\pi^n~: Y_n \to \tilde{Y}_n$, données par $\phi^{-n}([\tilde{\pi}]) \mapsto [\tilde{\pi}]$ et réciproquement, d'applications $\phi_\pi^{-n}~: \tilde{X}_n \to X_n$ et $\phi_\pi^{-n}~: \tilde{Y}_n \to Y_n$, données par $[\tilde{\pi}] \mapsto \phi^{-n}([\tilde{\pi}])$. La composée des applications $\B_{\tau,\log,K}^+ \overset{[\tilde{\pi}] \mapsto [\tilde{\pi}]}{\xrightarrow{\hspace*{1.5cm}}} \tilde{X}_n \overset{\phi_\pi^{-n}}{\xrightarrow{\hspace*{1.5cm}}} X_n$ définit un élément $\log(\phi^{-n}([\tilde{\pi}]))=\phi^{-n}(\log[\tilde{\pi}])$ dans $X_n$, et on appellera toujours $\log[\tilde{\pi}]$ l'élément $p^n\log(\phi^{-n}([\tilde{\pi}]))$.

On définit $n(r)$ comme le plus petit entier $n$ tel que $p^{n-1}(p-1) \geq r$, et on rappelle que (voir \cite[III. §2]{cherbcoliwa} par exemple) $\phi^{-n}(\Bt_{\mathrm{rig}}^{\dagger,r}) \subset \Bdrplus$ pour $n \geq n(r)$.

\begin{prop}
Les $X_n$ sont des sous-anneaux fermés de $\Bdrplus$ tels que $X_n \subset~X_{n+1}$, et si $n \geq n(r)$, alors $\phi^{-n}(\B_{\tau,\mathrm{rig},K}^{\dagger,r}) \subset X_n$.
\end{prop}
\begin{proof}
La proposition \cite[Prop. 7.14]{Col02} montre que $K_n \otimes_F \Btplus$ s'identifie à un sous-anneau de $\Bdrplus$, de sorte que c'est aussi le cas pour $K_n \otimes_{\O_F}\O_F[\![\phi^{-n}([\tilde{\pi}])]\!]$ qui est un sous-anneau de $K_n \otimes_F \Btplus$. De plus, $\phi^{-n}([\tilde{\pi}])-\pi_n$ engendre $\Ker(\theta)$ dans $\O_{K_n}\otimes_{\O_F}\Atplus$ (voir par exemple \cite[§7.4]{Col02}) et donc la topologie $(\phi^{-n}([\tilde{\pi}])-\pi_n)$-adique sur $K_n \otimes_{\O_F}\O_F[\![\phi^{-n}([\tilde{\pi}])]\!]$ coïncide avec la topologie $\Ker(\theta)$-adique, de sorte que $X_n$ est le complété de $K_n \otimes_{\O_F}\O_F[\![\phi^{-n}([\tilde{\pi}])]\!]$ dans $\Bdrplus$. En particulier, c'en est un sous-anneau fermé. Le fait que $X_n \subset X_{n+1}$ provient simplement du fait que $K_n \subset K_{n+1}$ et que $\phi^{-n}([\tilde{\pi}]) \in X_{n+1}$.

Il reste à montrer que $\phi^{-n}(\B_{\tau,\mathrm{rig},K}^{\dagger,r}) \subset X_n$. Comme $\phi^{-n}(\Bt_{\mathrm{rig}}^{\dagger,r}) \subset \Bdrplus$ pour $n \geq n(r)$, on sait déjà que $\phi^{-n}(\B_{\tau,\mathrm{rig},K}^{\dagger,r}) \subset \Bdrplus$, et que les éléments de $\B_{\tau,\mathrm{rig},K}^{\dagger,r}$ sont de la forme $\sum_{k \in \Z}a_k[\tilde{\pi}]^k$ tels que $a_k \in F$ et $v_p(a_k)+\frac{p-1}{pe}k/\rho \rightarrow +\infty$ quand $k \rightarrow \pm \infty$ pour tout $\rho \in [r,+\infty[$. De plus, on a
$$\phi^{-n}([\tilde{\pi}]) = [\tilde{\pi}^{p^{-n}}]=\pi_n\exp(\frac{\log[\tilde{\pi}]}{p^n}) \in X_n.$$
Si maintenant $x \in \B_{\tau,\mathrm{rig},K}^{\dagger,r}$, on écrit $x = \sum_{k \in \Z}a_k[\tilde{\pi}]^k$ et on a~:
\begin{align*}
\phi^{-n}(x) &= \sum_{k \in \Z}\phi^{-n}(a_k)\phi^{-n}([\tilde{\pi}])^k = \sum_{k \in \Z}a_k\pi_n^{k}\exp(\frac{k\log[\tilde{\pi}]}{p^n}) \\
&=\sum_{k \in \Z}a_k\pi_n^{k}\sum_{i=0}^\infty \frac{(k\log[\tilde{\pi}])^i}{p^{ni}i!}\\
&=\sum_{i=0}^\infty \frac{(\log[\tilde{\pi}])^i}{p^{ni}i!}(\sum_{k \in \Z}k^ia_k\pi_n^{k}).
\end{align*}
Comme la somme $\sum_{k \in \Z}k^ia_k\pi_n^{k}$ converge pour tout $i \in \N$ puisque $n \geq n(r)$ et comme $\log[\tilde{\pi}] \in (\phi^{-n}([\tilde{\pi}])-\pi_n)X_n$ par définition, on en déduit que $\phi^{-n}(x)$ définit bien un élément de $X_n$, ce qui termine la preuve.
\end{proof}

On notera par la suite $\iota_n~: \B_{\tau,\mathrm{rig},K}^{\dagger,r} \longrightarrow X_n$ l'application définie par $\phi^{-n}$ pour $n \geq n(r)$.

\begin{prop}
\label{divis élem}
Si $I$ est un idéal principal de $\B_{\tau,\mathrm{rig},K}^{\dagger,r}$ qui divise $\lambda^h$ pour $h \geq 0$, alors $I$ est engendré par un élément de la forme $\prod_{n=n(r)}^{+\infty}\left(\frac{\phi^n(E([\tilde{\pi}]))}{E(0)}\right)^{j_n}$ avec les $j_n \leq h$.
\end{prop}
\begin{proof}
On a $\lambda = \prod_{n=0}^{+\infty}\left(\frac{\phi^n(E([\tilde{\pi}]))}{E(0)}\right)$. De plus, on a $\B_{\tau,\mathrm{rig},K}^{\dagger,r}/\phi^n(E([\tilde{\pi}])) \simeq K_n$, de sorte que les $\phi^n(E([\tilde{\pi}]))$ sont des idéaux maximaux. Si maintenant $x$ est un générateur de $I$, $x|\lambda^h$ et donc par \cite[Prop. 10]{lazardfonctions1962}, $x$ est le produit d'une unité par un élément de la forme $\prod_{n=n(r)}^{+\infty}\left(\frac{\phi^n(E([\tilde{\pi}]))}{E(0)}\right)^{j_n}$.
\end{proof}

\begin{theo}
Si $D$ est un $\phi$-module sur $\B_{\tau,\mathrm{rig},K}^{\dagger}$, alors il existe $r(D) \geq 0$ tel que pour tout $r \geq r(D)$, il existe un unique sous $\B_{\tau,\mathrm{rig},K}^{\dagger,r}$-module $D_r$ de $D$ tel que~:
\begin{enumerate}
\item $D = \B_{\tau,\mathrm{rig},K}^\dagger \otimes_{\B_{\tau,\mathrm{rig},K}^{\dagger,r}}D_r$ ;
\item le $\B_{\tau,\mathrm{rig},K}^{\dagger,pr}$-module $\B_{\tau,\mathrm{rig},K}^{\dagger,pr}\otimes_{\B_{\tau,\mathrm{rig},K}^{\dagger,r}}D_r$ a une base contenue dans $\phi(D_r)$.
\end{enumerate}
\end{theo}
\begin{proof}
Ce résultat est une variante d'un résultat de Cherbonnier \cite{Che96}, et la version dans le cas des $\B_{\tau,\mathrm{rig},K}^\dagger$-modules du théorème I.3.3. de \cite{Ber08}. Comme $D$ est un $\B_{\tau,\mathrm{rig},K}^{\dagger}$-module libre de rang $d$, il en existe une base, qu'on note $e_1,\cdots,e_d$. Il existe donc $r=r(D)$ tel que la matrice de $\phi$ dans cette base soit dans $\GL_d(\B_{\tau,\mathrm{rig},K}^{\dagger,r})$, et donc en posant $D_r = \oplus_{i=1}^d\B_{\tau,\mathrm{rig},K}^{\dagger,r}e_i$, les deux conditions du théorème sont vérifiées. Il reste à montrer l'unicité. Soient maintenant $D_r^{(1)}$ et $D_r^{(2)}$ deux $\B_{\tau,\mathrm{rig},K}^{\dagger,r}$-modules vérifiant les deux conditions du théorème. On choisit des bases de ces deux modules, et on note $M$ la matrice de passage d'une base à l'autre et $P_1$, $P_2$ les matrices de $\phi$ dans les bases correspondantes. Ces matrices vérifient alors la relation $\phi(M)=P_1^{-1}MP_2$, avec $P_1,P_2 \in \GL_d(\B_{\tau,\mathrm{rig},K}^{\dagger,pr})$, et $M \in \mathrm{M}_d(\B_{\tau,\mathrm{rig},K}^{\dagger,s})$ pour un certain $s$. Mais si $s \geq pr$, alors $P_1^{-1}MP_2$ appartient à $M_d(\B_{\tau,\mathrm{rig},K}^{\dagger,s})$ et donc on a également $\phi(M) \in M_d(\B_{\tau,\mathrm{rig},K}^{\dagger,s})$. Mais 
$$\phi(\B_{\tau,\mathrm{rig},K}^{\dagger,s}) \cap \Bt_{\tau,\mathrm{rig},K}^{\dagger,s} = \phi(\B_{\tau,\mathrm{rig},K}^{\dagger,s} \cap \phi^{-1}(\Bt_{\tau,\mathrm{rig},K}^{\dagger,s})) = \phi(\B_{\tau,\mathrm{rig},K}^{\dagger,s/p})$$
puisque $\phi^{-1}(\Bt_{\tau,\mathrm{rig},K}^{\dagger,s}) = \Bt_{\tau,\mathrm{rig},K}^{\dagger,s/p}$ et que $\B_{\tau,\mathrm{rig},K}^{\dagger,s} \cap \Bt_{\tau,\mathrm{rig},K}^{\dagger,s'} = \B_{\tau,\mathrm{rig},K}^{\dagger,s'}$ si $s' \leq s$ par un analogue direct du lemme II.2.2 de \cite{cherbonnier1998representations}. On a donc $\phi(\B_{\tau,\mathrm{rig},K}^{\dagger,s}) \cap \B_{\tau,\mathrm{rig},K}^{\dagger,s} \subset \phi(\B_{\tau,\mathrm{rig},K}^{\dagger,s/p})$ mais réciproquement il est clair que $\phi(\B_{\tau,\mathrm{rig},K}^{\dagger,s/p}) \subset \phi(\B_{\tau,\mathrm{rig},K}^{\dagger,s}) \cap \B_{\tau,\mathrm{rig},K}^{\dagger,s}$ d'où l'égalité $\phi(\B_{\tau,\mathrm{rig},K}^{\dagger,s}) \cap \B_{\tau,\mathrm{rig},K}^{\dagger,s} = \phi(\B_{\tau,\mathrm{rig},K}^{\dagger,s/p})$.

En particulier, si $M \in M_d(\B_{\tau,\mathrm{rig},K}^{\dagger,s})$ avec $s \geq pr$, on peut remplacer $s$ par $s/p$, de sorte que $M$ est en fait dans $\mathrm{M}(\B_{\tau,\mathrm{rig},K}^{\dagger,r})$, et on peut appliquer le même raisonnement à $M^{-1}$, de sorte que $M \in \GL_d(\B_{\tau,\mathrm{rig},K}^{\dagger,r})$ et donc $D_r^{(1)}=D_r^{(2)}$.
\end{proof}

Si $D$ est un $\phi$-module sur $\B_{\tau,\mathrm{rig},K}^{\dagger}$ et si $n \geq n(r)$ avec $r \geq r(D)$, l'application $\iota_n~: \B_{\tau,\mathrm{rig},K}^{\dagger,r} \to X_n$ confère une structure de $\iota_n(\B_{\tau,\mathrm{rig},K}^{\dagger,r})$-module à $X_n$ et la formule $\iota_n(\mu)\cdot x = \mu x$ donne une structure de $\iota_n(\B_{\tau,\mathrm{rig},K}^{\dagger,r})$-module à $D_r$ et on note alors ce module $\iota_n(D_r)$. Cela nous permet donc de définir $X_n \otimes_{\iota_n(\B_{\tau,\mathrm{rig},K}^{\dagger,r})}\iota_n(D_r)$, qu'on écrira $X_n \otimes_{\B_{\tau,\mathrm{rig},K}^{\dagger,r}}^{\iota_n}D_r$ afin d'alléger les notations.

\begin{prop}
\label{unicité phi modules 1/lambda}
Si $D$ est un $\phi$-module de rang $d$ sur $\B_{\tau,\mathrm{rig},K}^{\dagger}$ et si $D^{(1)}$, $D^{(2)}$ sont deux sous $\B_{\tau,\mathrm{rig},K}^\dagger$-modules libres de rang $d$ et stables par $\phi$ de $D[1/\lambda]=\B_{\tau,\mathrm{rig},K}^\dagger[1/\lambda]\otimes_{\B_{\tau,\mathrm{rig},K}^\dagger}D$ tels que~:
\begin{enumerate}
\item $D^{(1)}[1/\lambda] = D^{(2)}[1/\lambda]=D[1/\lambda]$ ;
\item $X_n \otimes_{\B_{\tau,\mathrm{rig},K}^{\dagger,r}}^{\iota_n}D_r^{(1)}=X_n \otimes_{\B_{\tau,\mathrm{rig},K}^{\dagger,r}}^{\iota_n}D_r^{(2)}$ pour tout $n \gg 0$ ;
\end{enumerate}
alors $D^{(1)}=D^{(2)}$.
\end{prop}
\begin{proof}
C'est l'analogue de \cite[Prop. I.3.4]{Ber08}.

On a $D^{(1)}[1/\lambda]=D^{(2)}[1/\lambda]$ donc il existe $h \geq 0$ tel que $\lambda^hD^{(2)} \subset D^{(1)}$. Soit maintenant $r \geq \max(r(D^{(1)}),r(D^{(2)}))$ et tel que~:
$$X_n \otimes_{\B_{\tau,\mathrm{rig},K}^{\dagger,r}}^{\iota_n}D_r^{(1)}=X_n \otimes_{\B_{\tau,\mathrm{rig},K}^{\dagger,r}}^{\iota_n}D_r^{(2)}$$ 
pour tout $n \geq n(r)$. Les diviseurs élémentaires de $\lambda^hD^{(2)} \subset D^{(1)}$ sont des idéaux de $\B_{\tau,\mathrm{rig},K}^{\dagger,r}$ qui divisent $\lambda^h$ donc qui sont de la forme $\prod_{n=n(r)}^{+\infty}\left(\frac{\phi^n(E([\tilde{\pi}]))}{E(0)}\right)^{j_{n,i}}$ pour $i=1,\cdots,d$ par la proposition \ref{divis élem}. Le calcul des diviseurs élémentaires commutant à la localisation, ceux de l'inclusion 
$$\lambda^hX_n \otimes_{\B_{\tau,\mathrm{rig},K}^{\dagger,r}}^{\iota_n}D_r^{(2)} \subset X_n \otimes_{\B_{\tau,\mathrm{rig},K}^{\dagger,r}}^{\iota_n}D_r^{(1)}$$ 
sont donnés par les idéaux $(\lambda^{j_{n,i}})$, et donc $j_{n,i}=h$ pour tous $i$ et $n$, et donc $D^{(1)}=D^{(2)}$.
\end{proof}

\begin{lemm}
Si $D$ est un $\phi$-module sur $\B_{\tau,\mathrm{rig},K}^\dagger$, on dispose d'une application
$$\phi_n~: Y_{n+1} \otimes_{Y_n}\left(Y_n \otimes_{\B_{\tau,\mathrm{rig},K}^{\dagger,r}}^{\iota_n}D_r\right) \longrightarrow Y_{n+1} \otimes_{\B_{\tau,\mathrm{rig},K}^{\dagger,r}}^{\iota_{n+1}}D_r$$
donnée par $\phi_n(u_{n+1}\otimes u_n \otimes \iota_n(x)) = u_{n+1}u_n \otimes \iota_{n+1}(\phi(x))$.
\end{lemm}
\begin{proof}
Cela découle du fait qu'on a les inclusions $\phi(D_r) \subset \B_{\tau,\mathrm{rig},K}^{\dagger,pr} \otimes_{\B_{\tau,\mathrm{rig},K}^{\dagger,r}}D_r$, $\iota_{n+1}(\B_{\tau,\mathrm{rig},K}^{\dagger,pr}) \subset X_{n+1}$ et $\iota_{n+1}(\phi(\mu)) = \iota_n(\mu)$ si $\mu \in \B_{\tau,\mathrm{rig},K}^{\dagger,r}$.
\end{proof}

De façon analogue à \cite[Déf. II.1.1]{Ber08}, on définit une notion de suite $\phi$-compatible~:
\begin{defi}
\label{defi suite phi compatible}
Soit $D$ un $\phi$-module sur $\B_{\tau,\mathrm{rig},K}^\dagger$ et soit $r \geq r(D)$. Si $\{M_n\}_{n \geq n(r)}$ est une suite de $X_n$-réseaux de $Y_n \otimes_{\B_{\tau,\mathrm{rig},K}^{\dagger,r}}^{\iota_n}D_r$, on dit que c'est une suite $\phi$-compatible si 
$$\phi_n(X_{n+1}\otimes_{X_n}M_n) = M_{n+1}.$$
\end{defi}

\begin{lemm}
Si $M$ est un sous-$\phi$-module de rang maximal d'un $\phi$-module $D$ sur $\B_{\tau,\mathrm{rig},K}^\dagger$, alors les $M_n:=X_n\otimes_{\B_{\tau,\mathrm{rig},K}^{\dagger,r}}^{\iota_n}M_r$ déterminent une suite $\phi$-compatible de $X_n$-réseaux de $Y_n\otimes_{\B_{\tau,\mathrm{rig},K}^{\dagger,r}}^{\iota_n}D_r$.
\end{lemm}
\begin{proof}
Cela découle de la définition \ref{defi suite phi compatible} et de la construction des $\phi_n$.
\end{proof}

On dispose également d'un analogue du théorème II.1.2 de \cite{Ber08}~:
\begin{theo}
\label{thm recollement}
Si $D$ est un $\phi$-module sur $\B_{\tau,\mathrm{rig},K}^{\dagger}$, si $r \geq r(D)$ et si $\{M_n\}$ est une suite $\phi$-compatible de réseaux de $Y_n \otimes_{\B_{\tau,\mathrm{rig},K}^{\dagger,r}}^{\iota_n}D_r$, alors il existe un unique sous $\phi$-module $M$ de $D[1/\lambda]$ tel que $X_n\otimes_{\B_{\tau,\mathrm{rig},K}^{\dagger,r}}^{\iota_n}M_r = M_n$ pour tout $n \geq n(r)$. De plus, $M[1/\lambda]=D[1/\lambda]$. Si $D$ est un $(\phi,\tau)$-module sur $(\B_{\tau,\mathrm{rig},K}^\dagger,\Bt_{\mathrm{rig},L}^\dagger)$ tel que les $\B_{\mathrm{dR}}^{\G_L}\otimes_{X_n}M_n$ sont stables sous l'action induite de $G_\infty$, alors $M$ est aussi un $(\phi,\tau)$-module.
\end{theo}
L'unicité d'un tel $M$ est une conséquence directe de la proposition \ref{unicité phi modules 1/lambda}, et on va avoir besoin d'un certain nombre de résultats avant de démontrer son existence. Dans ce qui suit, on a donc fixé un $\phi$-module $D$ sur $\B_{\tau,\mathrm{rig},K}^{\dagger}$, un $r \geq r(D)$ et une suite $\phi$-compatible $\{M_n\}$ de réseaux de $Y_n \otimes_{\B_{\tau,\mathrm{rig},K}^{\dagger,r}}^{\iota_n}D_r$.

\begin{lemm}
\label{exist h tel que t-h < < th}
Il existe un entier $h \geq 0$ tel que pour tout $n \geq n(r)$ on ait~:
$$(\phi^{-n}([\tilde{\pi}])-\pi_n)^hX_n\otimes_{\B_{\tau,\mathrm{rig},K}^{\dagger,r}}^{\iota_n}D_r \subset M_n \subset (\phi^{-n}([\tilde{\pi}])-\pi_n)^{-h}X_n\otimes_{\B_{\tau,\mathrm{rig},K}^{\dagger,r}}^{\iota_n}D_r.$$
\end{lemm}
\begin{proof}
C'est l'analogue de \cite[Lemm. II.1.3]{Ber08}.
Comme $M_{n(r)}$ est un réseau de $X_{n(r)}\otimes_{\B_{\tau,\mathrm{rig},K}^{\dagger,r}}^{\iota_{n(r)}}D_r$, il existe $h$ tel que l'inclusion du lemme soit vraie pour $n=n(r)$. Comme de plus, $\phi_n(X_{n+1}\otimes_{X_n}M_n) = M_{n+1}$, $(\phi^{-n}([\tilde{\pi}])-\pi_n)X_{n+1} = (\phi^{-(n+1)}([\tilde{\pi}])-\pi_{n+1})X_{n+1}$ et 
$$\phi_n(X_{n+1}\otimes_{X_n}X_n\otimes_{\B_{\tau,\mathrm{rig},K}^{\dagger,r}}^{\iota_n}D_r) = X_{n+1}\otimes_{\B_{\tau,\mathrm{rig},K}^{\dagger,r}}^{\iota_{n+1}}D_r,$$
l'inclusion reste vraie pour $n+1$ si elle l'est pour $n$, ce qui permet de conclure.
\end{proof}

On fixe maintenant $h$ comme dans le lemme précédent, et on a un analogue de \cite[Lemm. II.1.4]{Ber08}~:

\begin{lemm}
\label{1/lambda}
On pose $M_r = \left\{x \in \lambda^{-h}D_r \textrm{ tels que } \iota_n(x) \in M_n \textrm{ pour tout } n \geq n(r)\right\}$. Alors $M_r$ est un $\B_{\tau,\mathrm{rig},K}^{\dagger,r}$-module libre de rang $d$.
\end{lemm}
\begin{proof}
Comme les applications $\iota_n~: \lambda^{-h}D_r \to M_n[1/\lambda]$ sont continues, $M_r$ est fermé dans $\lambda^{-h}D_r$ donc est libre de rang fini par le lemme \ref{ForsterKisin}. Comme de plus, $\lambda^hD_r \subset M_r$ par le lemme \ref{exist h tel que t-h < < th}, $M_r$ est un $\B_{\tau,\mathrm{rig},K}^{\dagger,r}$-module libre de rang $d$.
\end{proof}

\begin{lemm}
\label{X_nMr = Mn}
On a $X_n\otimes_{\B_{\tau,\mathrm{rig},K}^{\dagger,r}}^{\iota_n}M_r = M_n$ pour tout $n \geq n(r)$.
\end{lemm}
\begin{proof}
C'est l'analogue de \cite[Lemm. II.1.5]{Ber08}.

Comme $X_n\otimes_{\B_{\tau,\mathrm{rig},K}^{\dagger,r}}^{\iota_n}M_r$ et $M_n$ sont complets pour la topologie $(\phi^{-n}([\tilde{\pi}])-\pi_n)$-adique, il suffit de montrer que l'application naturelle
$$X_n\otimes_{\B_{\tau,\mathrm{rig},K}^{\dagger,r}}^{\iota_n}M_r \rightarrow M_n/(\phi^{-n}([\tilde{\pi}])-\pi_n)^hM_n = M_n/\lambda^hM_n$$
est surjective pour tout $n$.

Si $x \in M_n$, le lemme \ref{exist h tel que t-h < < th} montre qu'il existe $y \in \lambda^{-h}D_r$ tel que $\iota_n(y)-x \in \lambda^hM_n$.

La solution du problème des \guillemotleft{ }parties principales \guillemotright{ }(cf \cite[§8]{lazardfonctions1962}) montre que, si $\ell \geq 1$, il existe $\lambda_{n,\ell} \in \B_{\tau,\mathrm{rig},K}^{\dagger,r}$ tel que
$$\iota_n(\lambda_{n,\ell}) \in 1+(\phi^{-n}([\tilde{\pi}])-\pi_n)^\ell X_n$$
et
$$\iota_m(\lambda_{n,\ell}) \in (\phi^{-m}([\tilde{\pi}])-\pi_m)^\ell X_m$$
si $m \neq n$. On pose alors $z=\lambda_{n,3h}y$, de sorte que~:
$$\iota_n(z)-\iota_n(y) \in (\phi^{-n}([\tilde{\pi}])-\pi_n)^{2h}X_n \otimes_{\B_{\tau,\mathrm{rig},K}^{\dagger,r}}^{\iota_n}M_r \subset (\phi^{-n}([\tilde{\pi}])-\pi_n)^hM_n = \lambda^hM_n$$
et, si $m \neq n$, 
$$\iota_m(z) \in \lambda^{2h}X_m\otimes_{\B_{\tau,\mathrm{rig},K}^{\dagger,r}}^{\iota_n}D_r \subset \lambda^hM_m \subset M_m$$
et donc $z \in M_r$ et l'application naturelle
$$X_n\otimes_{\B_{\tau,\mathrm{rig},K}^{\dagger,r}}^{\iota_n}M_r \rightarrow (\phi^{-n}([\tilde{\pi}])-\pi_n)^hM_n$$
est bien surjective.
\end{proof}

\begin{proof}[Démonstration du théorème \ref{thm recollement}]
Soit 
$$M_r:= \left\{x \in \lambda^{-h}D_r \textrm{ tels que } \iota_n(x) \in M_n \textrm{ pour tout } n \geq n(r)\right\}.$$
Les lemmes \ref{1/lambda} et \ref{X_nMr = Mn} montrent que $M_r$ est un $\B_{\tau,\mathrm{rig},K}^{\dagger,r}$-module libre de rang $d$ et que $X_n \otimes_{\B_{\tau,\mathrm{rig},K}^{\dagger,r}}M_r = M_n$ pour tout $n \geq n(r)$. On définit donc $M:=\B_{\tau,\mathrm{rig},K}^\dagger \otimes_{\B_{\tau,\mathrm{rig},K}^{\dagger,r}}M_r$.

Comme la suite de réseaux $\{M_n\}$ est $\phi$-compatible, on en déduit que $\phi(M) \subset M$ et que $X_n\otimes_{\B_{\tau,\mathrm{rig},K}^{\dagger,r}}^{\iota_n}\phi^*(M)_{pr}=M_n$ pour tout $n \geq n(pr)$. La proposition \ref{unicité phi modules 1/lambda} appliquée à $M$ et $\phi^*(M)$ montre que $\phi^*(M)=M$, et on a bien $M[1/\lambda]=D[1/\lambda]$ par le lemme \ref{1/lambda}. Ce $M$ vérifie donc toutes les conditions, ce qui termine la preuve de l'existence.

Dans le cas où $D$ est un $(\phi,\tau)$-module tel que les $\B_{\mathrm{dR}}^{\G_L} \otimes_{X_n}M_n$ sont stables sous l'action induite de $G_\infty$, la définition de $M_r$ montre que $\Bt_{\mathrm{rig},L}^\dagger \otimes_{\B_{\tau,\mathrm{rig},K}^\dagger}M$ est muni d'une action de $G_\infty$ commutant à celle de $\phi$, ce qui conclut.
\end{proof}

On va maintenant utiliser ces résultats pour associer à un $(\phi,N)$-module filtré un $(\phi,\tau)$-module sur $\B_{\tau,\mathrm{rig},K}^\dagger$.

Si $D$ est un $(\phi,N)$-module filtré, on pose $\D=(\B_{\tau,\log,K}^\dagger\otimes_F D)^{N=0}$.

\begin{lemm}
On a
$$\B_{\tau,\log,K}^\dagger\otimes_{\B_{\tau,\mathrm{rig},K}^\dagger}\D=\B_{\tau,\log,K}^\dagger\otimes_F D.$$
\end{lemm}
\begin{proof}
Par définition de $\D$, on a $\D \subset \B_{\tau,\log,K}^\dagger\otimes_F D$ et il suffit donc de montrer que les éléments de $\D$ engendrent $\B_{\tau,\log,K}^\dagger\otimes_F D$. Comme $N$ est nilpotent sur $D$, il existe une base $(e_1,\cdots,e_d)$ dans laquelle $\Mat(N)$ est sous forme de Jordan, c'est-à-dire qu'il existe une partition de $d$ : $d_1+\cdots+d_r = d$ telle que pour tout $i \in \{1,\cdots,r\}$, le $F$-espace vectoriel $E_{d_i}=\mathrm{Vect}_F(e_{d_1+\cdots+d_i},e_{d_1+\cdots+d_i+1},\cdots,e_{d_1+\cdots+d_{i+1}-1})$ est stable par $N$ et qu'on ait $N(e_{d_1+\cdots+d_i}) = 0$ pour tout $i$ et $N(e_{d_1+\cdots+d_i+k}) = e_{d_1+\cdots+d_i+(k-1)}$ pour tout $k \in \left\{1,\cdots,d_{i+1}\right\}$. 

Comme les $E_i$ sont stables par $N$, il suffit de montrer que pour chacun des $E_i$, il existe une base de $\B_{\tau,\log,K}^\dagger\otimes_{\B_{\tau,\mathrm{rig},K}^\dagger}E_i$ tuée par $N$. Montrons le pour $E_1$, la démonstration étant la même pour les autres. 

On a $N(e_1) = 0$ donc $e_1'=1 \otimes e_1$ est un élément de $\B_{\tau,\log,K}^\dagger\otimes_{\B_{\tau,\mathrm{rig},K}^\dagger}E_1$ tué par $N$. On a $N(e_2) = e_1$ donc $N(1 \otimes e_2 - \log[\tilde{\pi}]\otimes e_1) = 0$ et donc $e_2' = 1 \otimes e_2 - \log[\tilde{\pi}]\otimes e_1$ est tué par $N$. De même, si on pose $e_j'=\sum_{k=1}^j (-1)^{j-k}\log[\tilde{\pi}]^{j-k}\otimes e_{k}$ pour $j \in \{1,\cdots,d_1\}$, alors comme $N(e_{i+1})=e_i$ pour tout $i$, $e_j'$ est bien tué par $N$. Pour conclure, il suffit de remarquer que comme $(e_1,\cdots,e_{d_1})$ est une base de $E_1$, la famille $(e_1',\cdots,e_{d_1}')$ est libre dans $\B_{\tau,\log,K}^\dagger\otimes_{\B_{\tau,\mathrm{rig},K}^\dagger}E_1$.
\end{proof}

En particulier, ce lemme montre que $\D$ est un $\phi$-module sur $\B_{\tau,\mathrm{rig},K}^\dagger$. 

Pour $n \in \Z$, on a $\phi^{-n}(F) = F \subset K$, ce qui confère à $K$, et à $D$, une structure de $\phi^{-n}(F)$-module. On note $\iota_n(D)$ le $\phi^{-n}(F)$-module ainsi obtenu, et on écrira $K\otimes_F^{\iota_n}D$ plutôt que $K \otimes_{\phi^{-n}(F)}\iota_n(D)$ afin d'alléger les notations. L'application $\xi_n~: K\otimes_FD \to K\otimes_F^{\iota_n}D$, envoyant $\mu \otimes x$ sur $\mu \otimes \iota_n(\phi^n(x))$ est un isomorphisme puisque $\iota_n=\phi^{-n}$, ce qui permet de munir $D_K^n:=K\otimes_F^{\iota_n}D$ de la filtration image par cette application $\xi_n$. On définit également une filtration sur $Y_n$ par $\Fil^iY_n = (\phi^{-n}([\tilde{\pi}])-\pi_n)^iY_n$, ce qui nous donne une filtration sur $Y_n \otimes_KD_K^n$, et on pose $M_n(D):=\Fil^0(Y_n\otimes_KD_K^n)$. Le foncteur $D \mapsto M_n(D)$ est alors un $\otimes$-foncteur exact. Comme de plus, 
$$\B_{\tau,\log,K}^\dagger\otimes_{\B_{\tau,\mathrm{rig},K}^\dagger}\D=\B_{\tau,\log,K}^\dagger\otimes_FD$$
et $\iota_n(\B_{\tau,\log,K}^{\dagger,r})\subset Y_n$ pour $n \geq n(r)$, on trouve que
$$Y_n \otimes_{\B_{\tau,\mathrm{rig},K}^{\dagger,r}}^{\iota_n}\D_r=Y_n\otimes_KD_K^n$$
pour $n \geq n(r)$, et donc $M_n(D)$ est un réseau de $Y_n \otimes_{\B_{\tau,\mathrm{rig},K}^{\dagger,r}}^{\iota_n}\D_r$.

\begin{prop}
La famille de réseaux $\{M_n\}$ de $Y_n \otimes_{\B_{\tau,\mathrm{rig},K}^{\dagger,r}}^{\iota_n}\D_r$ définie par 
$$M_n(D)=\Fil^0(Y_n\otimes_KD_K^n)$$
est $\phi$-compatible.
\end{prop} 
\begin{proof}
C'est l'analogue de \cite[Prop. II.2.1]{Ber08}.
Le $X_n$-module $M_n(D)=\Fil^0(Y_n\otimes_KD_K^n)$ est libre de rang $d$ et engendré par une base de la forme $(\phi^{-n}([\tilde{\pi}])-\pi_n)^{-h_i}\otimes\xi_n(e_i)$, où $e_1,\cdots,e_d$ est une base de $D_K$ adaptée à la filtration, et $h_i=t_H(e_i)$, de sorte que $M_{n+1}(D)=\phi_n(X_{n+1}\otimes_{X_n}M_n(D))$ puisque $\xi_{n+1} = \phi_n \otimes \xi_n$ sur $D_K$.
\end{proof}

Le théorème \ref{thm recollement} nous permet de définir un $(\phi,\tau)$-module, de façon analogue à \cite[Déf. II.2.2]{Ber08}.
\begin{defi}
\label{defi phiNphitau}
Si $D$ est un $(\phi,N)$-module filtré, on définit $\mathcal{M}(D)$ le $(\phi,\tau)$-module obtenu à partir du théorème \ref{thm recollement} à partir des réseaux 
$$M_n(D)=\Fil^0(Y_n\otimes_KD_K^n)$$
pour le $(\phi,\tau)$-module $\D=(\B_{\tau,\log,K}^\dagger\otimes_F D)^{N=0}$.
\end{defi}

\begin{prop}
Si $D$ est un $(\phi,N,\G_{M/K})$-module filtré et si $D'$ est le $(\phi,N)$-module filtré relatif à $M$ qu'on en déduit par oubli de l'action de $\G_{M/K}$, alors $\mathcal{M}(D')$ est un $(\phi,\tau_M)$-module sur $(\B_{\tau,\mathrm{rig},M}^{\dagger,r},\Bt_{\mathrm{rig},L_M}^{\dagger,r})$ muni d'une action de $\G_{M/K}$.
\end{prop}
\begin{proof}
C'est une conséquence directe de la définition \ref{phi tauM action GMK rig}.
\end{proof}

\begin{defi}
Si $D$ est un $(\phi,N,\G_{M/K})$-module filtré, on définit $\mathcal{M}(D)$ par $\mathcal{M}(D)=\mathcal{M}(D')^{H_{\tau,K}}$.
\end{defi}

\begin{prop}
\label{prop phiNrelatif}
Si $D$ est un $(\phi,N,\G_{M/K})$-module filtré, et si on pose $M_n(D_K):=\Fil^0(Y_n \otimes_K D_K^n)$, alors $X_n \otimes_{\B_{\tau,\mathrm{rig},K}^{\dagger,r}}^{\iota_n}\mathcal{M}(D)_r = M_n(D_K)$ pour tout $n \geq n(r)$.
\end{prop}
\begin{proof}
C'est l'analogue de \cite[Prop. II.2.5]{Ber08}.

Par construction, $X_n \otimes_{\B_{\tau,\mathrm{rig},K}^{\dagger,r}}^{\iota_n}\mathcal{M}(D)_r \subset \Fil^0(Y_{M,n} \otimes_M D_M^n)^{H_{\tau,K}}$. Par \cite[4.3.2]{fontaine1994representations}, on a 
$$\Fil^0(Y_{M,n}\otimes_M D_M^n)^{H_{\tau,K}} = \Fil^0(Y_n \otimes_K D_K^n),$$
et donc $X_n\otimes_{\B_{\tau,\mathrm{rig},K}^{\dagger,r}}^{\iota_n}\mathcal{M}(D)_r \subset M_n(D_K)$.

Si maintenant $x \in M_n(D_K)$, alors $x \in (X_{M,n}\otimes_{\B_{\tau,\mathrm{rig},M}^{\dagger,r}}^{\iota_n}\mathcal{M}(D)_r)^{H_{\tau,K}} = X_n\otimes_{\B_{\tau,\mathrm{rig},K}^{\dagger,r}}^{\iota_n}\mathcal{M}(D)_r$, et donc l'application
$$X_n\otimes_{\B_{\tau,\mathrm{rig},K}^{\dagger,r}}^{\iota_n}\mathcal{M}(D)_r \rightarrow M_n(D_K)$$
est un isomorphisme.
\end{proof}

On dispose d'un théorème analogue à \cite[Thm. II.2.6]{Ber08}~:
\begin{theo}
Le foncteur $D \mapsto \cal{M}(D)$ est un foncteur exact de la catégorie tannakienne des $(\phi,N,\G_K)$-modules filtrés dans la catégorie tannakienne des $(\phi,\tau)$-modules sur $\B_{\tau,\mathrm{rig},K}^\dagger$, et le rang de $\cal{M}(D)$ est égal à la dimension de $D$.
\end{theo}
\begin{proof}
La définition \ref{defi phiNphitau} montre que les $M_n(D)$ ont même rang que la dimension de $D$, et donc $\cal{M}(D)$ a pour rang la dimension de $D$.

Si on a une suite exacte 
$$0 \to D_1 \to D_2 \to D_3 \to 0,$$
on va montrer que $\cal{M}(D_2) \to \cal{M}(D_3)$ est surjectif. Comme $D \mapsto M_n(D)$ est $\otimes$-exact, $M_n(D_2) \to M_n(D_3)$ est surjectif pour tout $n$, de sorte que l'image $\M$ de $\cal{M}(D_2)$ dans $\cal{M}(D_3)$ vérifie 
$$X_n\otimes_{\B_{\tau,\mathrm{rig},K}^{\dagger,r}}^{\iota_n}\M_r=M_n(D_3)$$
et donc par l'unicité dans le théorème \ref{thm recollement}, l'image de $\cal{M}(D_2)$ dans $\cal{M}(D_3)$ est $\cal{M}(D_3)$, ce qui prouve la surjectivité.

Les mêmes arguments montrent que $\cal{M}(D_1\otimes D_2) = \cal{M}(D_1)\otimes\cal{M}(D_2)$.
\end{proof}

On va maintenant montrer que ce foncteur $D \mapsto \cal{M}(D)$ étend celui de Kisin~:
\begin{prop}
\label{étend foncteur Kisin}
Si $D$ est un $(\phi,N)$-module filtré effectif, et si $\mathfrak{M}(D)$ désigne le $(\phi,N_\nabla)$-module sur $\B_{\tau,\mathrm{rig},K}^+$ obtenu à partir de $D$ via le lemme \ref{phiNeffphinab}, alors $\cal{M}(D) \simeq \B_{\tau,\mathrm{rig},K}^\dagger\otimes_{\B_{\tau,\mathrm{rig},K}^+}\mathfrak{M}(D)$ en tant que $(\phi,N_{\nabla})$-modules.
\end{prop}
\begin{proof}
Soit $D$ un $(\phi,N)$-module filtré effectif. On a défini $\cal{M}(D)$ comme
$$\cal{M}(D)=\left\{x \in (\B_{\tau,\log,K}^\dagger[1/\lambda]\otimes_FD)^{N=0}~: \iota_n(x) \in \Fil^0(Y_n \otimes_KD_K^n) \right\}$$
et $\mathfrak{M}(D)$ était défini comme
$$\mathfrak{M}(D)=\left\{x \in (\B_{\tau,\log,K}^+[1/\lambda]\otimes_FD)^{N=0}~: j_n(x) \in \Fil^0(\tilde{Y}_n \otimes_KD_K)\right\}.$$
Comme $(\iota_n)_{|\B_{\tau,\log,K}^+}$ est la composée de $j_n$ et de l'application $\phi_\pi^{-n}~: \tilde{X}_n \to X_n$ et que $\phi_\pi^{-n}(Y_n) = \tilde{Y}_n$, on a $\mathfrak{M}(D) \subset \cal{M}(D)$. On pose $\mathfrak{M}^\dagger(D):=\B_{\tau,\mathrm{rig},K}^\dagger\otimes_{\B_{\tau,\mathrm{rig},K}^+}\mathfrak{M}(D)$, et l'inclusion précédente nous donne $\mathfrak{M}^{\dagger}(D) \subset \cal{M}(D)$. 

Soit $r \geq r(\cal{M}(D))$, et soit pour $n \geq n(r)$~:
$$M_n:=X_n\otimes_{\B_{\tau,\mathrm{rig},K}^{\dagger,r}}^{\iota_n}\mathfrak{M}^{\dagger}_r(D).$$
La suite $M_n$ est donc $\phi$-compatible pour $\mathfrak{M}(D)^{\dagger}$ et donc pour $\cal{M}(D)$ puisque $\cal{M}(D)$ et $\mathfrak{M}(D)$ ont même rang (égal à la dimension de $D$). De plus, par définition de $\mathfrak{M}(D)$ et comme on a un isomorphisme 
$$X_n \otimes_{\B_{\tau,\mathrm{rig},K}^+}(\B_{\tau,\log,K}^+\otimes_FD)^{N=0} \simeq X_n \otimes_KD_K^n$$
en appliquant $\phi_\pi^{-n}$ au $(1)$ du lemme \ref{lemme Kisin strucFil}, on en déduit que $M_n = \Fil^0(Y_n\otimes_KD_K^n)$. Or $\cal{M}(D)$ est défini par la suite $\phi$-compatible $M_n = \Fil^0(Y_n\otimes_KD_K^n)$, ce qui conclut par l'unicité dans le théorème \ref{thm recollement}.
\end{proof}
 
\subsection{Construction de $(\phi,N)$-modules filtrés}~
On va maintenant montrer comment on peut associer à certains $(\phi,\tau)$-modules sur $\B_{\tau,\mathrm{rig},K}^\dagger$ un $(\phi,N)$-module filtré, cette construction étant un inverse de celle de la partie précédente. Avant cela, on va avoir besoin de quelques définitions et rappels sur les modules à connexion qu'on va être amené à considérer.

Comme dans \cite[Déf. III.1.2]{Ber08}, on définit la notion de connexion localement triviale~:
\begin{defi}
\label{defi loc triv}
Soit $\D$ un $(\phi,\tau)$-module sur $\B_{\tau,\mathrm{rig},K}^\dagger$. On dit que l'opérateur $N_\nabla$ est localement trivial sur $\D$ s'il existe $r$ tel que~:
$$X_{n(r)}\otimes_{\B_{\tau,\mathrm{rig},K}^{\dagger,r}}^{\iota_{n(r)}}\D_r = X_{n(r)} \otimes_{K_{n(r)}}(X_{n(r)}\otimes_{\B_{\tau,\mathrm{rig},K}^{\dagger,r}}^{\iota_{n(r)}}\D_r)^{N_\nabla=0}.$$
\end{defi}

Soit $M$ un $Y_n$-espace vectoriel de dimension $d$, muni d'une connexion $N_\nabla~: M \to M$ qui étend celle sur $Y_n$. On appelle section horizontale de $M$ un élément de $M^{N_\nabla=0}$.

La connexion $N_\nabla$ est dite régulière si $M$ possède un $X_n$-réseau $M_0$ tel que $N_\nabla(M_0) \subset M_0$, et on dit que la connexion est triviale si $M$ possède un $X_n$-réseau $M_0$ tel que $N_\nabla(M_0) \subset (\phi^{-n}([\tilde{\pi}])-\pi_n)M_0$. 

\begin{lemm}
Soit $M$ un $Y_n$-espace vectoriel de dimension $d$, muni d'une connexion $N_\nabla~: M \to M$ qui étend celle sur $Y_n$. On a alors~:
\begin{enumerate}
\item $\dim_{K_n}M^{N_\nabla=0} \leq d$ ;
\item si $M_0$ est un $X_n$-réseau de $M$ tel que $N_\nabla(M_0) \subset (\phi^{-n}([\tilde{\pi}])-\pi_n)M_0$, alors $M_0^{N_\nabla=0}$ est un $K_n$-espace vectoriel de dimension $d$ et $M_0 = X_n \otimes_{K_n}M_0^{N_\nabla=0}$.
\end{enumerate}
\end{lemm}
\begin{proof}
Pour montrer le premier point, il suffit de montrer qu'une famille d'éléments de $M^{N_\nabla=0}$ libre sur $K_n$ l'est encore sur $Y_n$. On va procéder par récurrence sur le nombre d'éléments. S'il n'y en a qu'un seul, il n'y a rien à faire. Sinon, soit $n \geq 2$ et soient $x_1,\cdots,x_n$ des éléments de $M^{N_\nabla=0}$ formant une famille libre sur $K_n$. Si $\lambda_1,\cdots,\lambda_n$ sont des éléments non tous nuls de $Y_n$ tels que 
$$\sum \lambda_ix_i=0$$
alors par hypothèse de récurrence on peut supposer qu'ils sont tous non nuls, et quitte à diviser par $\lambda_n$, on peut réécrire cette égalité sous la forme 
$$x_n = \sum \lambda_i'x_i.$$
En appliquant $N_\nabla$, on trouve 
$$\sum N_\nabla(\lambda_i')x_i = 0$$
de sorte que pour tout $i$, $N_\nabla(\lambda_i') = 0$ et donc $\lambda_i' \in K_n$. L'identité 
$$x_n = \sum \lambda_i'x_i$$
contredit alors la liberté des $x_i$ sur $K_n$, ce qui conclut la preuve.

Pour le deuxième point, soit $M_0$ un $X_n$-réseau de $M$ tel que $N_\nabla(M_0) \subset (\phi^{-n}([\tilde{\pi}])-\pi_n)M_0$. En particulier, si on note $\partial = \frac{1}{\lambda}N_\nabla$, alors $\partial(M_0) \subset M_0$. Soit alors $D_k = \Mat(\partial^k)$ pour $k \in \N$ dans une certaine base de $M_0$. Un petit calcul montre alors que $H = \sum_{k \geq 0}(-1)^kD_k \frac{(\phi^{-n}([\tilde{\pi}])-\pi_n)^k}{k!}$ converge vers une solution de $\partial(H)+D_1H=0$, de sorte que $M_0^{N_\nabla=0}$ est un $K_n$-espace vectoriel de dimension $d$ tel que $M_0 = X_n \otimes_{K_n}M_0^{N_\nabla=0}$.
\end{proof}

En particulier, la connexion est triviale si et seulement si $\dim_{K_n}M^{N_\nabla=0}=d$, et dans ce cas $M_0 = X_n \otimes_{K_n}M^{N_\nabla=0}$ est l'unique $X_n$-réseau de $M$ tel que $N_\nabla(M_0) \subset (\phi^{-n}([\tilde{\pi}])-\pi_n)M_0$.

\begin{lemm}
\label{lemme loc triviale sous espace stable}
Si $N$ est un sous-$Y_n$-espace vectoriel de $M$ stable par $N_\nabla$ et si $N_\nabla$ est triviale sur $M$, alors elle est triviale sur $N$.
\end{lemm}
\begin{proof}
C'est un analogue de \cite[Lemm. III.1.3]{Ber08}.

Si $M_0$ est un $X_n$-réseau de $M$ tel que $N_\nabla(M_0) \subset (\phi^{-n}([\tilde{\pi}])-\pi_n)M_0$, alors $M_0 \cap N$ est un $X_n$-réseau de $N$ tel que $N_\nabla(M_0 \cap N) \subset (\phi^{-n}([\tilde{\pi}])-\pi_n)(M_0 \cap N)$.
\end{proof}

\begin{lemm}
\label{lemm loc triv n > n(r)}
Soient $\D$ et $r$ comme dans la définition \ref{defi loc triv}. Si $N_\nabla$ est localement triviale sur $\D$, alors $N_\nabla$ est triviale sur $Y_n \otimes_{\B_{\tau,\mathrm{rig},K}^{\dagger,r}}^{\iota_n}\D_r$ pour tout $n \geq n(r)$.

De plus, si $D_n$ est le réseau de $Y_n \otimes_{\B_{\tau,\mathrm{rig},K}^{\dagger,r}}^{\iota_n}\D_r$ tel que $N_\nabla(D_n) \subset (\phi^{-n}([\tilde{\pi}])-\pi_n)D_n$, alors $D_{n+1}=\phi_n(X_{n+1}\otimes_{X_n}D_n)$.
\end{lemm}
\begin{proof}
C'est un analogue de \cite[Lemm. III.1.4]{Ber08}.

Si $D_n = X_n \otimes_{K_n}(Y_n \otimes_{\B_{\tau,\mathrm{rig},K}^{\dagger,r}}^{\iota_n}\D_r)^{N_\nabla=0}$, alors $N_\nabla(D_{n(r)}) \subset (\phi^{-n}([\tilde{\pi}])-\pi_n)D_{n(r)}$ par hypothèse. Si $n \geq n(r)$ et $N_\nabla(D_n) \subset (\phi^{-n}([\tilde{\pi}])-\pi_n)D_n$, alors en notant $D_{n+1}$ le réseau de $Y_{n+1} \otimes_{\B_{\tau,\mathrm{rig},K}^{\dagger,r}}^{\iota_{n+1}}\D_r$ donné par $D_{n+1} = \phi_n(X_{n+1}\otimes_{X_n}D_n)$, on a $N_\nabla(D_{n+1}) \subset (\phi^{-(n+1)}([\tilde{\pi}])-\pi_{n+1})D_{n+1}$ puisque $N_\nabla$ commute à $\phi_n$.
\end{proof}

\begin{lemm}
\label{phi N filtré implique loc triv}
Si $D$ est un $(\phi,N,\G_{L/K})$-module filtré, alors $N_\nabla$ est localement triviale sur $\cal{M}(D)$.
\end{lemm}
\begin{proof}
C'est l'analogue de \cite[Prop. III.1.5]{Ber08}.

La proposition \ref{prop phiNrelatif} montre que 
$$Y_n \otimes_{\B_{\tau,\mathrm{rig},K}^{\dagger,r}}^{\iota_n}\mathcal{M}(D)_r = Y_n \otimes_KD_K^n,$$
et donc $N_\nabla=0$ sur $D_K^n$ puisque $\G_{M/K}$ est fini.
\end{proof}

En particulier, le lemme \ref{phi N filtré implique loc triv} montre que si $\cal{M}$ est dans l'image de $D \mapsto \cal{M}(D)$, alors $N_\nabla$ est localement triviale sur $\cal{M}$. On va maintenant s'intéresser à la réciproque, et on va déterminer l'image essentielle du foncteur $D \mapsto \mathcal{M}(D)$. 

Avant cela, on va avoir besoin de faire quelques rappels sur les équations différentielles $p$-adiques et notamment la conjecture de monodromie $p$-adique, dans le cas qui nous intéresse.

\begin{defi}
Si $\D$ est un $\phi$-module sur $\B_{\tau,\mathrm{rig},K}^\dagger$ muni d'un opérateur différentiel $\partial_{\D}$ qui étend l'opérateur $\partial~: f \mapsto X\frac{df}{dX}$ et tel que $\partial_\D \circ \phi = p \cdot \phi \circ \partial_\D$, alors on dit que $\D$ est une équation différentielle $p$-adique avec structure de Frobenius.
\end{defi}

\begin{rema}
\label{rmqeqdiff}
De façon plus générale, on appelle équation différentielle $p$-adique avec structure de Frobenius tout $\phi$-module $\D$ sur l'anneau de Robba $\cal{R}_K$ associé à un corps $p$-adique $K$ (et sur lequel on a choisi un Frobenius), muni d'un opérateur différentiel $\partial_{\D}$ qui étend n'importe quelle dérivation continue $\partial$  et tel que $\partial_\D \circ \phi = \frac{\partial(\phi(T))}{\phi(\partial T)} \cdot \phi \circ \partial_\D$.
\end{rema}

En particulier, si $\D$ est un $(\phi,\tau)$-module sur $\B_{\tau,\mathrm{rig},K}^\dagger$ et si $\D$ est stable par l'opérateur $\partial_\D = \frac{1}{\lambda}N_\nabla$, alors $\D$ est une équation différentielle avec structure de Frobenius.

Le théorème suivant, connu dans le cas général de la remarque \ref{rmqeqdiff} sous le nom de conjecture de monodromie $p$-adique, a été conjecturé par Crew et démontré par André \cite{andre2002}, Kedlaya \cite{kedlayamonodromie} et Mebkhout \cite{Meb02}~:

\begin{theo}
\label{thm eqdiff padic}
Si $\D$ est une équation différentielle $p$-adique avec structure de Frobenius, alors il existe une extension finie $M/K$ telle que l'application naturelle~:
$$\B_{\tau,\mathrm{rig},M}^\dagger[\log[\tilde{\pi}]] \otimes_{M_0'} (\B_{\tau,\mathrm{rig},M}^\dagger[\log[\tilde{\pi}]]\otimes_{\B_{\tau,\mathrm{rig},K}^\dagger}\D)^{\partial_\D = 0} \longrightarrow \B_{\tau,\mathrm{rig},M}^\dagger[\log[\tilde{\pi}]] \otimes_{\B_{\tau,\mathrm{rig},K}^\dagger}\D$$
soit un isomorphisme, $M_0'$ étant l'extension maximale non ramifiée de $K_0$ dans $M_\infty~:= M \cdot K_\infty$.
\end{theo}

\begin{rema}
Pour plus de détails sur les conjectures de monodromies $p$-adiques, on renvoie également au séminaire Bourbaki de Colmez \cite{Col01}.
\end{rema}

Si maintenant $\D$ est un $(\phi,\tau)$-module sur $\B_{\tau,\mathrm{rig},K}^\dagger$, stable sous l'action de $\partial_D = \frac{1}{\lambda}N_\nabla$, on pose $S_M(\D) = (\B_{\tau,\mathrm{rig},M}^\dagger[\log[\tilde{\pi}]]\otimes_{\B_{\tau,\mathrm{rig},K}^\dagger}\D)^{\partial_\D=0}$. C'est un $M_0'$-espace vectoriel, et comme $\nabla_\tau=0$ sur $S_M(\D)$, quitte à étendre les scalaires à une extension finie de $M$, on peut supposer que $\Gal((M\cdot L)/M_{\mathrm{cycl}})$ agit trivialement sur $S_M(\D)$. En particulier, on a alors $M_0 = M_0'$. Sans perte de généralité, on peut aussi supposer que $M_\infty/K_\infty$ est galoisienne. Sous ces hypothèses, $S_M(\D)$ est un $(\phi,N,\G_{M/K})$-module tel que 
$$\B_{\tau,\mathrm{rig},M}^\dagger[\log[\tilde{\pi}]] \otimes_{M_0} S_M(D) = \B_{\tau,\mathrm{rig},M}^\dagger[\log[\tilde{\pi}]]\otimes_{\B_{\tau,\mathrm{rig},K}^\dagger}\D.$$

On dispose d'un analogue du théorème III.2.3 de \cite{Ber08}~:
\begin{theo}
\label{thm phitau phiN}
Si $\M$ est un $(\phi,\tau)$-module sur $\B_{\tau,\mathrm{rig},K}^\dagger$ tel que $N_\nabla$ est localement triviale, alors il existe un unique $(\phi,\tau)$-module $\D \subset \M[1/\lambda]$ tel que $\D[1/\lambda]=\M[1/\lambda]$ et tel que $\partial_\M(\D) \subset \D$. De plus, la donnée de $\M$ détermine une filtration sur $M \otimes_{M_0}S_M(\D)$ et donc une structure de $(\phi,N,\G_{M/K})$-module filtré sur $S_M(\D)$ telle que $\M = \cal{M}(S_M(\D))$.
\end{theo}
Avant de démontrer ce résultat, on aura besoin du lemme suivant, qui est le lemme 7.6 de \cite{berger2012multivariable}~:
\begin{lemm}
\label{lemm W=Fil0}
Soit $D$ un $F$-espace vectoriel, et soit $W$ un $\Bdrplus$-réseau de $\Bdr \otimes_F D$ stable sous l'action de $\G_F$, où $\G_F$ agit trivialement sur $D$. Si on pose $\Fil^iD=D \cap t^i\cdot W$, alors $W=\Fil^0(\Bdr\otimes_F D)$.
\end{lemm}

\begin{proof}[Démonstration du théorème \ref{thm phitau phiN}]
Puisque $N_\nabla$ est localement triviale, il existe une famille de réseaux $D_n$ de $Y_n\otimes_{\B_{\tau,\mathrm{rig},K}^{\dagger,r}}^{\iota_n}\M_r$ tels que $N_{\nabla}(D_n) \subset (\phi^{-n}([\tilde{\pi}])-\pi_n)D_n$. Comme de plus, $D_{n+1}=\phi_n(X_{n+1}\otimes_{X_n}D_n)$ par le lemme \ref{lemm loc triv n > n(r)}, la famille $\{D_n\}$ est $\phi$-compatible. Le théorème \ref{thm recollement} nous donne alors un $(\phi,\tau)$-module $\D$ tel que $X_n \otimes_{\B_{\tau,\mathrm{rig},K}^{\dagger,r}}^{\iota_n} \D_r = D_n$. Or $N_\nabla(D_n) \subset (\phi^{-n}([\tilde{\pi}])-\pi_n)D_n$ pour tout $n$, et donc $N_\nabla(\D) \subset \lambda\D$, ce qui conclut la démonstration de l'existence d'un tel $\D$. L'unicité vient alors du fait qu'on a forcément $X_n\otimes_{\B_{\tau,\mathrm{rig},K}^{\dagger,r}}^{\iota_n}\D_r = D_n$ pour tout $n \geq n(r)$. 

Le théorème \ref{thm eqdiff padic} de monodromie $p$-adique et la discussion qui suit montrent qu'il existe une extension finie $M$ de $K$ telle que $M_\infty/K_\infty$ est galoisienne, $\Gal((M\cdot L)/M_{\mathrm{cycl}})$ agit trivialement sur $S_M(\D)$ et $S_M(\D)$ est un $(\phi,N,\G_{M/K})$-module tel que 
$$\B_{\tau,\mathrm{rig},M}^\dagger[\log[\tilde{\pi}]] \otimes_{M_0} S_M(\D) = \B_{\tau,\mathrm{rig},M}^\dagger[\log[\tilde{\pi}]]\otimes_{\B_{\tau,\mathrm{rig},K}^\dagger}\D.$$

Dans ce qui suit et pour simplifier les notations, on notera $D=S_M(\D)$ et $Y_{M,n} = Y_n \cdot M$, $X_{M,n}= \Fil^0(Y_{M,n})$. On va maintenant construire une filtration sur $D_M = M\otimes_{M_0}D$. On dispose d'isomorphismes
$$Y_{M,n}\otimes_MD_M^n = Y_{M,n}\otimes_{\B_{\tau,\mathrm{rig},K}^{\dagger,r}}^{\iota_n}\D_r \simeq Y_{M,nr} \otimes_{\B_{\tau,\mathrm{rig},K}^{\dagger,r}}^{\iota_n}\M_r$$
qu'on utilise pour définir $\Fil^iD_M^n = D_M^n \cap (\phi^{-n}([\tilde{\pi}])-\pi_n)^iY_{M,n}\otimes_{\B_{\tau,\mathrm{rig},K}^{\dagger,r}}^{\iota_n}\M_r$, et on définit alors $\Fil^iD_M$ comme le tiré en arrière de la filtration $\Fil^iD_M^n$ par l'isomorphisme $D_M \rightarrow D_M^n$.

Par définition, l'isomorphisme $D_M \simeq D_M^n$ est donné par $\mu \otimes x \mapsto \mu \otimes \iota_n(\phi^n(x))$, de sorte que la filtration induite sur $D_M$ ne dépend pas de $n$.

Pour finir, on a $X_n \otimes_{\B_{\tau,\mathrm{rig},K}^{\dagger,r}}^{\iota_n}\M_r = \Fil^0(Y_n \otimes_K D_K^n)$ de sorte que $\M = \cal{M}(D) = \cal{M}(S_M(\D))$ par le lemme \ref{lemm W=Fil0}.
\end{proof}

Comme conséquence directe, on obtient le théorème suivant~:
\begin{theo}
Le foncteur $D \mapsto \cal{M}(D)$, de la catégorie des $(\phi,N,\G_{M/K})$-modules filtrés, vers la catégorie des $(\phi,\tau)$-modules sur $\B_{\tau,\mathrm{rig},K}^\dagger$ dont la connexion associée est localement triviale, est une équivalence de catégories.
\end{theo}

\begin{coro}
\label{coro sousobjets phiN}
Si $D$ est un $(\phi,N,\G_{M/K})$-module filtré, et si $\M'$ est un sous-$(\phi,\tau)$-module saturé de $\M = \cal{M}(D)$, alors il existe un sous-objet $D' \subset D$ tel que $\M' = \cal{M}(D')$.
\end{coro}
\begin{proof}
Le lemme \ref{lemme loc triviale sous espace stable} montre que la connexion $N_\nabla$ est localement triviale sur $\M'$, et donc on peut écrire $\M' = \cal{M}(D')$ où $D'$ est un $(\phi,N,\G_{M'/K})$-module filtré avec $M'/K$ finie mais suffisamment grande. Il reste à montrer que comme $\M' \subset \M$ est saturé (le résultat est faux sans cette hypothèse), on obtient une inclusion $D' \subset D$ comme $(\phi,N,\G_{M'/K})$-modules filtrés, de sorte que $\G_{M'/M}$ agit trivialement sur $D'$ et donc que $D'$ est un sous-$(\phi,N,\G_{M/K})$-module filtré de $D$, ce qui permettra de conclure. Comme $\M' \subset \M$ est saturé, on en déduit que $\M/\M'$ est muni d'une structure de $(\phi,\tau)$-module, et on dispose d'une suite exacte 
$$0 \rightarrow \M' \rightarrow \M \rightarrow \M/\M' \rightarrow 0$$
et donc d'une suite exacte
$$0 \rightarrow D' \rightarrow D$$
dans la catégorie des $(\phi,N,\G_{M'/K})$-modules filtrés. 
\end{proof}

\subsection{Représentations potentiellement semi-stables}
Dans cette partie, on va s'intéresser au calcul des pentes de Frobenius du $(\phi,\tau)$-module $\cal{M}(D)$ associé au $(\phi,N)$-module filtré $D$. Ces résultats sont un analogue de ceux de la partie IV de Berger dans \cite{Ber08}.

\begin{theo}
Si $D$ est un $(\phi,N,\G_{M/K})$-module filtré, alors la pente de $\det \cal{M}(D)$ est égale à $t_N(D)-t_H(D)$.
\end{theo}
\begin{proof}
C'est l'analogue de \cite[Thm. IV.2.1]{Ber08}.

Le foncteur $D \mapsto \cal{M}(D)$ étant exact et compatible aux produits tensoriels, on a $\det\cal{M}(D) = \cal{M}(\det D)$. De plus, par définition, on a $t_N(D) = t_N(\det D)$ et $t_H(D)=t_H(\det D)$, de sorte qu'il suffit de montrer le résultat lorsque $D$ est de rang $1$.

Soit $D$ de rang $1$ et $e$ une base de $D$ telle que $\phi(e)=p^{\nu}\lambda_0\cdot e$, avec $\lambda_0 \in \O_{K_0}^\times$, et on note $\eta = t_H(e)$. Alors $\cal{M}(D) = \B_{\tau,\mathrm{rig},K}^{\dagger}\lambda^{-\eta} \otimes e$, et $\phi(\lambda^{-\eta}\otimes e)=(E([\tilde{\pi}])/E(0))^{\eta}p^{\nu}\lambda_0\lambda^{-\eta} \otimes e$ et donc la pente de $\cal{M}(D)$ est égale à $\nu-\eta$, c'est-à-dire $t_N(D)-t_H(D)$.
\end{proof}

\begin{prop}
\label{admissible = étale}
Si $D$ est un $(\phi,N,\G_{M/K})$-module filtré, alors $D$ est admissible si et seulement si $\cal{M}(D)$ est un $(\phi,\tau)$-module étale.
\end{prop}
\begin{proof}
C'est la même démonstration que \cite[Prop. IV.2.2.]{Ber08}.

On va commencer par montrer le sens direct. Supposons donc que $D$ est admissible. Le théorème 6.4.1 de \cite{slopes} montre que $\cal{M}(D)$ admet une filtration canonique par des sous $(\phi,\tau)$-modules isoclines de pentes croissantes. La somme de ces pentes est, avec multiplicité, la pente de $\det\cal{M}(D)$, c'est-à-dire $t_N(D)-t_H(D) = 0$. Il suffit donc de montrer que les pentes de $\cal{M}(D)$ sont toutes $\geq 0$. Le corollaire \ref{coro sousobjets phiN} nous dit que tout sous-objet de $\cal{M}(D)$ est de la forme $\cal{M}(D')$ où $D' \subset D$ et la pente de $\det\cal{M}(D')$ est $t_N(D')-t_H(D') \geq 0$ puisque $D$ est admissible. En particulier, $\cal{M}(D)$ ne peut pas contenir de sous-objet isocline de pente $< 0$, et donc $\cal{M}(D)$ est étale.

Pour l'autre sens, supposons à présent que $\cal{M}(D)$ est étale. La pente de $\det \cal{M}(D)$ est nulle et donc $t_N(D)-t_H(D) = 0$. Si $D'$ est un sous-objet de $D$, de dimension $d'$, alors $\det(D')$ est de dimension $1$ dans $\Lambda^{d'}D$ et $\cal{M}(\det D')$ est un sous-$\phi$-module de rang $1$ de $\cal{M}(\Lambda^{d'}D)$. La proposition \cite[Prop. IV.1.3]{Ber08} et le fait que $\cal{M}(\Lambda^{d'}D) = \Lambda^{d'}\cal{M}(D)$ est étale montrent que la pente de $\cal{M}(\det D')$ est positive, de sorte que $t_N(D')-t_H(D') \geq 0$, et donc $D$ est admissible.
\end{proof}

En particulier, ce résultat permet de caractériser les représentations potentiellement semi-stables en fonction de la connexion associée à leur $(\phi,\tau)$-module sur $\B_{\tau,\mathrm{rig},K}^\dagger$~: ce sont celles dont la connexion est localement triviale. On peut également montrer que le $(\phi,N,\G_{M/K})$-module filtré qu'on récupère à partir du $(\phi,\tau)$-module associé à une représentation $V$ est exactement $D_{\mathrm{st}}(V)$~:

\begin{prop}
\label{prop foncteur vrai Dst}
Si $V$ est une représentation $p$-adique de $\G_K$, dont la restriction à $\G_M$ est semi-stable, alors $\mathcal{M}(D_{\mathrm{st},M}(V)) = D_{\tau,\mathrm{rig}}^\dagger(V)$.
\end{prop}
\begin{proof}
Soit $D$ un $(\phi,N,\G_{M/K})$-module filtré admissible. Le théorème \ref{admissible = étale} montre qu'il existe une représentation $p$-adique $V$ de $\G_K$ telle que $\mathcal{M}(D) = D_{\tau,\mathrm{rig}}^\dagger(V)$. Or, le théorème \ref{retrouver Dst Dcris phitau} appliqué aux $(\phi,\tau_M)$-modules nous donne que 
$$D_{\mathrm{st},M}(V) = (\B_{\tau,\log,M}^\dagger[1/\lambda] \otimes_{\B_{\tau,\mathrm{rig},K}^\dagger}D_{\tau,\mathrm{rig}}^\dagger(V))^{\tau_M=1}.$$
Comme de plus, on a
$$\B_{\tau,\log,M}^\dagger[1/\lambda] \otimes_{\B_{\tau,\mathrm{rig},K}}\mathcal{M}(D) = \B_{\tau,\log,M}^\dagger[1/\lambda] \otimes_{M_0} D$$
et que $(\B_{\tau,\log,M}^\dagger[1/\lambda])^{\tau_M=1} = M_0$, on a alors $D = (\B_{\tau,\log,M}^\dagger[1/\lambda] \otimes_{\B_{\tau,\mathrm{rig},K}}D_{\tau,\mathrm{rig}}^\dagger(V))^{\tau_M=1}$ et donc que $D = D_{\mathrm{st},M}(V)$ en tant que $(\phi,N,\G_{M/K})$-modules.

Il reste à voir que les filtrations sur $D$ et $D_{\mathrm{st},M}(V)$ coïncident. La filtration sur $M \otimes_{M_0} D$ est par construction la même que celle sur $S_M(\mathcal{M}(D))$, et il faut donc voir qu'elle coïncide avec celle provenant de l'isomorphisme $(M\otimes_{M_0}D_{\mathrm{st},M}(V))^{\G_{M/K}} = D_{\mathrm{dR}}(V)$. Or, si $n \geq n(r)$, l'application $\iota_n$ envoie $\B_{\tau,\log,M}^{\dagger,r}\otimes_{\B_{\tau,\mathrm{rig},K}^{\dagger,r}}D_{\tau,\mathrm{rig}}^{\dagger,r}(V)$ dans $\Bdr \otimes_{\Qp}V$ et donc $Y_{M,n} \otimes_{M}D_M \subset \Bdr \otimes_{\Qp} V$. Mais pour $r$ tel que $D_{\tau,\mathrm{rig}}^{\dagger,r}$ engendre $D_{\tau,\mathrm{rig}}^\dagger$ et $n \geq n(r)$, si on pose 
$$W=\Bdrplus \otimes_{X_{M,n}}(X_{M,n} \otimes_{\B_{\tau,\mathrm{rig},K}^{\dagger,r}}^{\iota_n}D_{\tau,\mathrm{rig}}^{\dagger,r}(V)),$$
alors $W$ est un $\Bdrplus$-réseau de $\Bdr \otimes_{\Qp}V = \Bdr \otimes_{M}(M \otimes_{M_0}D_{\mathrm{st},M}(V))$, stable sous Galois, de sorte que par le lemme \ref{lemm W=Fil0}, on a $W = \Fil^0(\Bdr \otimes_{M}(M \otimes_{M_0}D_{\mathrm{st},M}(V)))$ et donc 
$$X_{M,n}\otimes_{\B_{\tau,\mathrm{rig},K}^{\dagger,r}}^{\iota_n}D_{\tau,\mathrm{rig}}^{\dagger,r}(V) = \Fil^0(Y_{M,n} \otimes_K D_{\mathrm{dR}}(V))$$
de sorte que la filtration sur $D$ construite dans le théorème \ref{thm phitau phiN} coïncide avec la filtration provenant de celle de $D_{\mathrm{dR}}(V)$, d'où l'égalité $D = D_{\mathrm{st},M}(V)$ en tant que $(\phi,N,\G_{M/K})$-modules filtrés.
\end{proof}

\section{Représentations de $E$-hauteur finie}
\subsection{Potentielle semi-stabilité des représentations de $E$-hauteur finie}

On a vu au chapitre 2 qu'on pouvait attacher à un $(\phi,\tau)$-module $D$ sur $(\B_{\tau,\mathrm{rig},K}^\dagger,\Bt_{\mathrm{rig},L}^\dagger)$ un opérateur différentiel $N_\nabla$ qui stabilise le $\phi$-module sur $\B_{\tau,\mathrm{rig},K}^\dagger$ et donc de considérer le $(\phi,N_\nabla)$-module sur $\B_{\tau,\mathrm{rig},K}^\dagger$ associé. On va en fait montrer que les propositions \ref{stabilité connexion} et \ref{V V' meme Nabla} permettent de retrouver le théorème principal de \cite{GaoEhauteur} (qui était un des principaux résultats de l'article de Caruso \cite[Thm. 3]{Car12} mais il semblerait qu'il y ait un problème dans la preuve, on renvoie d'ailleurs à l'appendice de \cite{GaoEhauteur} pour une discussion sur la question), c'est-à-dire que les représentations de $E$-hauteur finie de $\G_K$ sont potentiellement semi-stables. On signale au passage que la stratégie employée ici repose principalement sur les mêmes arguments que ceux utilisés par Gao. 

On note $\Mod_{/\B_{\tau,\mathrm{rig},K}^+}^{\phi,N_\nabla,0}$ la catégorie des $(\phi,N_\nabla)$-modules étales sur $\B_{\tau,\mathrm{rig},K}^+$ de $E$-hauteur finie, et on note $\Mod_{/\B_{\tau,\mathrm{rig},K}^+}^{\phi,\tau,0}$ la catégorie des $(\phi,\tau)$-modules étales sur $(\B_{\tau,\mathrm{rig},K}^+,\Bt_{\mathrm{rig},L}^+)$ de $E$-hauteur finie. 

On va en fait reconstruire un analogue du foncteur $S_{\phi,N_\nabla}$ de \cite[Thm. 3.5]{Car12}. Comme dans \cite[3.1.3]{Car12}, on construit un foncteur $\cal{R}_{\phi,\tau} : \Mod_{/\B_{\tau,\mathrm{rig},K}^+}^{\phi,N_\nabla,0} \to \Mod_{/\B_{\tau,\mathrm{rig},K}^+}^{\phi,\tau,0}$ (c'est le même foncteur $\cal{R}_{\phi,\tau}$ que celui de Caruso au théorème de pleine fidélité de Kedlaya \cite[Thm. 6.3.3]{slopes} près, c'est-à-dire que notre $\cal{R}_{\phi,\tau}(\cal{M})$ est le tensorisé de celui de Caruso par $(\B_{\tau,\mathrm{rig},K}^+,\Bt_{\mathrm{rig},L}^+)$ au-dessus de $(\B_{\tau,K}^+,\Bt_L^+)$). Soit $\mathfrak{M}$ un $(\phi,N_\nabla)$-module sur $\B_{\tau,\mathrm{rig},K}^+$ qui est de $E$-hauteur finie. La théorie de Kisin \cite[Thm. 1.2.15]{KisinFiso} associe à un tel objet une représentation semi-stable $V(\mathfrak{M})$ à poids de Hodge-Tate négatifs (en fait Kisin associe à un tel objet un $(\phi,N)$-module filtré effectif, mais on utilise ici directement l'équivalence de catégories entre $(\phi,N)$-modules filtrés effectifs et représentations semi-stables $V$ à poids de Hodge-Tate négatifs) dont le $(\phi,\tau)$-module associé est engendré par un $(\phi,\tau)$-module défini sur $(\B_{\tau,K}^+,\Bt_L^+)$, et on note $\cal{R}_{\phi,\tau}(\cal{M})$ ce $(\phi,\tau)$-module différentiel sur $(\B_{\tau,K}^+,\Bt_L^+)$. La démonstration du théorème D de l'introduction repose en fait sur la construction d'un presque quasi-inverse de $\cal{R}_{\phi,\tau}$, qui découle du théorème suivant~:

\begin{theo}
\label{theo pcpal Gao}
Soit $D$ un $(\phi,\tau)$-module étale sur $(\B_{\tau,\mathrm{rig},K}^+,\Bt_{\mathrm{rig},L}^+)$ de $E$-hauteur finie. Alors la connexion associée $N_\nabla : \B_{\tau,\mathrm{rig},K}^\dagger \otimes_{\B_{\tau,\mathrm{rig},K}^+}D \to \B_{\tau,\mathrm{rig},K}^\dagger \otimes_{\B_{\tau,\mathrm{rig},K}^+}D$ vérifie $N_\nabla(D) \subset D$.
\end{theo}
\begin{proof}
Voir \cite[Prop. 6.1.1]{GaoEhauteur}.
\end{proof}

\begin{rema}
Ce qui est clair est que l'opérateur $\nabla_\tau$ laisse $(\Bt_{\mathrm{rig},L}^+)^{\pa}\otimes_{\B_{\tau,\mathrm{rig},K}^+}D$ invariant, mais lorsqu'on pose $N_\nabla = -\frac{\lambda}{t}\nabla_\tau$, on fait apparaître des singularités qui font qu'\textit{a priori} $N_\nabla$ n'est pas définie sur $D$. Le théorème précédent montre que, dans le cas où $D$ est de $E$-hauteur finie, ces singularités sont effaçables.
\end{rema}

On montre maintenant comment en déduire le théorème $D$ de l'introduction :
\begin{theo}
\label{theo Gao}
Soit $V$ une représentation $p$-adique de $\G_K$. Alors $V$ est de $E$-hauteur finie si, et seulement si, la restriction de $V$ à $\G_{K_s}$ est semi-stable à poids de Hodge-Tate négatifs, $s$ étant le plus grand entier tel que $K^{\mathrm{unr}}$ contienne $\zeta_{p^s}$.
\end{theo}
\begin{proof}
Commençons par remarquer que la catégorie des représentations de $E$-hauteur finie est équivalente, par passage aux $(\phi,\tau)$-modules, à la catégorie $\Mod_{/\B_{\tau,\mathrm{rig},K}^+}^{\phi,\tau,0}$.

On va à présent construire un foncteur $S_{\phi,N_\nabla} : \Mod_{/\B_{\tau,\mathrm{rig},K}^+}^{\phi,\tau,0} \to \Mod_{/\B_{\tau,\mathrm{rig},K}^+}^{\phi,N_\nabla,0}$ vérifiant les propriétés suivantes (on pourra comparer avec \cite[Thm 3.5]{Car12})~:
\begin{enumerate}
\item pour tout objet $D \in \Mod_{/\B_{\tau,\mathrm{rig},K}^+}^{\phi,N_\nabla,0}$, il y a un isomorphisme canonique 
$$f_D : S_{\phi,N_\nabla} \circ \cal{R}_{\phi,\tau}(D) \simeq D$$
dans la catégorie $\Mod_{/\B_{\tau,\mathrm{rig},K}^+}^{\phi,N_\nabla,0}$ ;
\item pour tout objet $\mathfrak{M} \in \Mod_{/\B_{\tau,\mathrm{rig},K}^+}^{\phi,\tau,0}$, il y a un isomorphisme canonique 
$$g_{\mathfrak{M}} : \cal{R}_{\phi,\tau} \circ S_{\phi,N_\nabla}(\mathfrak{M}) \simeq \mathfrak{M}$$
en tant que $(\phi,\tau^{p^s})$-modules.
\end{enumerate}
Soit $D \in \Mod_{/\B_{\tau,\mathrm{rig},K}^+}^{\phi,\tau,0}$. Par le théorème \ref{theo pcpal Gao}, la connexion $N_\nabla$ provenant de l'action infinitésimale de $\tau$ laisse stable le $\phi$-module $D$ sur $\B_{\tau,\mathrm{rig},K}^+$, ce qui confère à $D$ une structure de $(\phi,N_\nabla)$-module étale sur $\B_{\tau,\mathrm{rig},K}^+$, et on note $S_{\phi,N_\nabla}(D)$ l'objet obtenu, qui est toujours de $E$-hauteur finie puisque $D$ l'est. Il reste à montrer que $S_{\phi,N_\nabla}$ vérifie les propriétés énoncées. \'Etant donné $\mathfrak{M} \in \Mod_{/\B_{\tau,\mathrm{rig},K}^+}^{\phi,N_\nabla,0}$, les propositions \ref{prop foncteur vrai Dst} et \ref{étend foncteur Kisin} montrent que, si $V$ est la représentation associée à $\mathfrak{M}$ par la théorie de Kisin, alors $\mathfrak{M} \simeq D_{\tau,\mathrm{rig}}^+(V)$ en tant que $(\phi,N_\nabla)$-modules de $E$-hauteur finie sur $\B_{\tau,\mathrm{rig},K}^+$, c'est-à-dire que $S_{\phi,N_\nabla} \circ R_{\phi,\tau}(\mathfrak{M}) \simeq \mathfrak{M}$ dans la catégorie $\Mod_{/\B_{\tau,\mathrm{rig},K}^+}^{\phi,N_\nabla,0}$.

Réciproquement, étant donné un $(\phi,\tau)$-module $D$ de $E$-hauteur finie sur $\B_{\tau,\mathrm{rig},K}^+$, ce qu'on vient de faire montre qu'on a un isomorphisme 
$$S_{\phi,N_\nabla}(D) \simeq S_{\phi,N_\nabla}(\cal{R}_{\phi,\tau} \circ S_{\phi,N_\nabla})(D),$$
de sorte qu'il existe $n \geq 0$ tel que $D \simeq (\cal{R}_{\phi,\tau} \circ S_{\phi,N_\nabla})(D)$ en tant que $(\phi,\tau^{p^n})$-modules par la proposition \ref{V V' meme Nabla}. Mais, si on pose $D'=(\cal{R}_{\phi,\tau} \circ S_{\phi,N_\nabla})(D)$, alors $D'$ est le $(\phi,\tau)$-module associé à une représentation semi-stable $V(D')$ à poids négatifs par le théorème \ref{thm pcpal de Kisin}, d'où le fait que la représentation $V(D)$ associée à $V$ devient semi-stable à poids négatifs en restriction à $\G_{K_n}$, et donc est en fait semi-stable à poids négatifs comme représentation de $\G_{K_s}$ (voir par exemple \cite[Lemm. 6.2.1]{GaoEhauteur}), et on a alors $V(D) \simeq V(D')$ en tant que $\G_{K_s}$-représentations (cf. \cite[Prop. 3.4]{Car12}), ce qui conclut.

Si maintenant $V$ est une représentation $p$-adique dont la restriction à $\G_{K_s}$ est semi-stable à poids négatifs, le théorème \ref{thm pcpal de Kisin} de Kisin appliqué aux représentations de $\G_{K_s}$ montre que $V$ est de $E$-hauteur finie.
\end{proof} 

\begin{rema}
On pouvait presque démontrer ce résultat à la suite des propriétés des $(\phi,N_\nabla)$-modules du chapitre 2, mais on avait besoin du fait que le $(\phi,N_\nabla)$-module de $E$-hauteur finie construit dans la théorie de Kisin et associé à une représentation semi-stable $V$ à poids négatifs correspond exactement à $D_{\tau,\mathrm{rig}}^\dagger(V)$, ce qui est une conséquence des propositions \ref{prop foncteur vrai Dst} et \ref{étend foncteur Kisin} du chapitre 3.
\end{rema}

En particulier, le théorème \ref{theo Gao} permet de réinterpréter le fait d'être de $E$-hauteur finie en fonction de la connexion $N_\nabla$, en démontrant un analogue des résultats de \cite[§5.2]{Ber02}. 

\begin{prop}
\label{conditions equiv unipotent}
Soit $D$ un $(\phi,N_\nabla)$-module de rang $d$ sur $\B_{\tau,\mathrm{rig},K}^\dagger$ (resp. $\B_{\tau,\mathrm{rig},K}^+$). Les propriétés suivantes sont équivalentes~:
\begin{enumerate}
\item $N_\nabla$ est triviale sur $D \otimes_{\B_{\tau,\mathrm{rig},K}^\dagger}\B_{\tau,\log,K}^\dagger$ (resp. sur $D \otimes_{\B_{\tau,\mathrm{rig},K}^+}\B_{\tau,\log,K}^+$), c'est-à-dire qu'il existe 
$$e_0,\cdots,e_{d-1} \in D \otimes_{\B_{\tau,\mathrm{rig},K}^\dagger}\B_{\tau,\log,K}^\dagger (\textrm{resp. } D \otimes_{\B_{\tau,\mathrm{rig},K}^+}\B_{\tau,\log,K}^+)$$
tels que pour tout $i$, $N_\nabla e_i=0$ et $\bigoplus e_i\B_{\tau,\log,K}^\dagger[1/\lambda] = D \otimes_{\B_{\tau,\mathrm{rig},K}^\dagger}\B_{\tau,\log,K}^\dagger[1/\lambda]$ (resp. $\bigoplus e_i\B_{\tau,\log,K}^+[1/\lambda] = D \otimes_{\B_{\tau,\mathrm{rig},K}^+}\B_{\tau,\log,K}^+[1/\lambda]$) ;
\item il existe $d$ éléments $f_0, \cdots,f_{d-1}$ de $D$, formant une base de $D \otimes_{\B_{\tau,\mathrm{rig},K}^\dagger}\B_{\tau,\mathrm{rig},K}^\dagger[1/\lambda]$ sur $\B_{\tau,\mathrm{rig},K}^\dagger[1/\lambda]$ (resp. formant une base de $D \otimes_{\B_{\tau,\mathrm{rig},K}^+}\B_{\tau,\mathrm{rig},K}^+[1/\lambda]$ sur $\B_{\tau,\mathrm{rig},K}^+[1/\lambda]$), tels que $N_\nabla(f_i) \in \lambda\cdot\pscal{f_{i-1},\cdots,f_0}$
où $\pscal{\cdot}$ est le $\B_{\tau,\mathrm{rig},K}^\dagger$-module engendré (resp. le $\B_{\tau,\mathrm{rig},K}^+$-module engendré).
\end{enumerate}
\end{prop}
\begin{proof}
La preuve est similaire à \cite[Prop. 5.5]{Ber02}. On ne va montrer que le cas où $D$ est un $(\phi,N_\nabla)$-module de rang $d$ sur $\B_{\tau,\mathrm{rig},K}^\dagger$, la preuve étant la même sur $\B_{\tau,\mathrm{rig},K}^+$ (à un détail près qu'on explicitera). On va commencer par montrer que le premier point implique le second. Comme les $e_i$ engendrent le noyau de $N_\nabla$ sur $D \otimes_{\B_{\tau,\mathrm{rig},K}^\dagger}\B_{\tau,\log,K}^\dagger$, et comme $N_\nabla$ commute à $N$, on en déduit que le $F$-espace vectoriel engendré par les $e_i$ est laissé stable par $N$. L'égalité $N\phi = p\phi N$ montre que $N$ est nilpotent sur $\pscal{e_0,\cdots,e_{d-1}}$, et on peut donc supposer que $N(e_i) \in \pscal{e_{i-1},\cdots,e_0}$. On peut écrire de manière unique $e_i = \sum_{j=0}^{d-1} \log^j(\pi) d_{ji}$ et on va montrer que $f_i=d_{0,i}$ est une famille qui satisfait la condition du deuxième point. Le fait que $N_\nabla(e_i)=0$ implique que 
$$N_\nabla(d_{0,i})=\lambda\cdot d_{1,i}$$
et le fait que $N(e_i) \in \pscal{e_{i-1},\cdots,e_0}$ montre que 
$$d_{1,i} \in \pscal{d_{0,i-1},\cdots,d_{0,0}}.$$
On en déduit que $N_\nabla(d_{0,i}) \in \lambda  \pscal{d_{0,i-1},\cdots,d_{0,0}}$.

Montrons à présent que les $f_i$ engendrent $D \otimes_{\B_{\tau,\mathrm{rig},K}^\dagger}\B_{\tau,\mathrm{rig},K}^\dagger[1/\lambda]$ sur $\B_{\tau,\mathrm{rig},K}^\dagger[1/\lambda]$. Si $x \in D \otimes_{\B_{\tau,\mathrm{rig},K}^\dagger}\B_{\tau,\mathrm{rig},K}^\dagger[1/\lambda]$, alors par hypothèse on peut écrire $x =\sum x_i e_i$ et donc $x = \sum \lambda_i f_i +y \cdot \log([\tilde{\pi}])$ où $y_i$ est le terme constant de $x_i$ et comme $x \in D \otimes_{\B_{\tau,\mathrm{rig},K}^\dagger}\B_{\tau,\mathrm{rig},K}^\dagger[1/\lambda]$ on a $y=0$.

Il reste à montrer l'implication réciproque. On va montrer par récurrence que l'on peut prendre $e_i = \sum_{j=0}^i f_j a_{ji}$ avec $a_{ji} \in \B_{\tau,\log,K}^\dagger$ et $a_{ii}=1$. Pour $D$ de rang $1$, il n'y a rien à faire. Si $d \geq 2$ tel que le résultat est vrai pour $D$ de rang $d-1$, la connexion $N_\nabla$ induit une connexion sur $(\bigoplus f_i\B_{\tau,\mathrm{rig},K}^\dagger)/f_0 \B_{\tau,\mathrm{rig},K}^\dagger$ 
qui satisfait les mêmes conditions et il existe donc $e'_1,\cdots,e'_{d-1}$ et $\alpha_1,\cdots,\alpha_{d-1}$ tels que $N_\nabla(e'_i)=\alpha_i f_0$. Le fait que $N_\nabla(f_i) \in \lambda\cdot\pscal{f_{i-1},\cdots,f_0}$ montre que les $\alpha_i \in \lambda \B_{\tau,\log,K}^\dagger$. S'il existe des $\beta_i \in \B_{\tau,\log,K}^\dagger$ tels que $N_\nabla(\beta_i)=\alpha_i$, on pose alors $e_0=f_0$ et $e_i=e'_i-\beta_i f_0$, ce qui termine la récurrence. La matrice de passage des $f_i$ aux $e_i$ est triangulaire avec des $1$ sur la diagonale de sorte que les $e_i$ engendrent bien $D \otimes_{\B_{\tau,\log,K}^\dagger}\B_{\tau,\log,K}^\dagger[1/\lambda]$ sur $\B_{\tau,\log,K}^\dagger[1/\lambda]$.

Il reste à montrer que $\frac{1}{\lambda}N_\nabla : \B_{\tau,\log,K}^\dagger \to \B_{\tau,\log,K}^\dagger$ est surjective (et que $\frac{1}{\lambda}N_\nabla : \B_{\tau,\log,K}^+ \to \B_{\tau,\log,K}^+$ est surjective pour adapter la preuve au cas où $D$ est un $(\phi,N_\nabla)$-module de rang $d$ sur $\B_{\tau,\mathrm{rig},K}^+$). Notons $\partial = \frac{1}{\lambda}N_\nabla$ et soit $f \in \B_{\tau,\mathrm{rig},K}^\dagger$. On écrit $f = \sum_{n \in \Z}a_n[\tilde{\pi}]^n$ avec les $a_n$ dans $F$, et on vérifie alors que $g = a_0\cdot \log [\tilde{\pi}] + \sum_{n \neq 0}\frac{a_n}{n}\cdot [\tilde{\pi}]^n$ vérifie $\partial g = f$, de sorte que $\partial$ réalise une surjection de $\B_{\tau,\mathrm{rig},K}^\dagger + F\cdot \log [\tilde{\pi}]$ dans $\B_{\tau,\mathrm{rig},K}^\dagger$ et de $\B_{\tau,\mathrm{rig},K}^+ + F\cdot \log [\tilde{\pi}]$ dans $\B_{\tau,\mathrm{rig},K}^+$. Pour finir, comme $\partial (h(T)\cdot \log^j [\tilde{\pi}]) = \log^j [\tilde{\pi}] \partial h([\tilde{\pi}]) + h([\tilde{\pi}])\log^{j-1} [\tilde{\pi}]$, on en déduit par récurrence que $\partial : \B_{\tau,\log,K}^\dagger \to \B_{\tau,\log,K}^\dagger$ est surjective, et qu'il en est de même pour $\partial : \B_{\tau,\log,K}^+ \to \B_{\tau,\log,K}^+$. Cela conclut la preuve. 
\end{proof}

On dit qu'un $(\phi,N_\nabla)$-module vérifiant les conditions précédentes est unipotent, et qu'il est trivial si on peut choisir les $f_i$ du second point de la proposition tels que $N_\nabla(f_i) = 0$.

\begin{prop}
\label{carac cristal unipotent}
Soit $V$ une représentation $p$-adique de dimension $d$. Alors il existe $n$ tel que la restriction de $V$ à $\G_{K_n}$ est semi-stable (resp. cristalline) si et seulement si $D_{\tau,\mathrm{rig}}^\dagger(V)[1/\lambda]$ contient un sous $(\phi,N_\nabla)$-module unipotent (resp. trivial) de rang $d$.
\end{prop}
\begin{proof}
C'est un analogue direct de \cite[Prop.5.6]{Ber02}.

Soit $V$ une représentation $p$-adique de $\G_K$. On a vu à la proposition \ref{retrouver Dst Dcris phitau} que $V$ est une représentation semi-stable de $G_{K_n}$ si et seulement si $D_{\tau,\log}^\dagger(V)[1/\lambda]$ contient $d$ éléments linéairement indépendant fixes sous $\tau^{p^n}$ qui forment alors une base de $D_{\mathrm{st}}(V_{|\G_{K_n}})$. Si $V_{|\G_{K_n}}$ est semi-stable, alors $N_\nabla$ est triviale sur le $\B_{\tau,\mathrm{rig},K}^\dagger$-module qu'ils engendrent, et donc satisfont la première condition de la proposition précédente en raison du théorème de comparaison 
$$D_{\mathrm{st}}(V) \otimes_F \B_{\tau,\log,K}^\dagger[1/\lambda] = D_{\tau,\mathrm{rig}}^\dagger(V) \otimes_{\B_{\tau,\mathrm{rig},K}^\dagger}\B_{\tau,\log,K}^\dagger[1/\lambda].$$

Le deuxième point de la proposition précédente fournit alors un sous $(\phi,N_\nabla)$-module unipotent de $D_{\tau,\mathrm{rig},K}^\dagger(V)[1/\lambda]$.

Réciproquement, si $D_{\tau,\mathrm{rig}}^\dagger(V)[1/\lambda]$ contient un sous $(\phi,N_\nabla)$-module unipotent, alors les $e_0,\cdots,e_{d-1}$ engendrent un $F$-espace vectoriel sur lequel $\log(\tau)$ agit trivialement, de sorte que les $e_i$ sont stables par $\tau^{p^n}$ pour $n$ assez grand et forment alors une base de $D_{\mathrm{st}}(V_{|\G_{K_n}})$.

De plus $V$ est cristalline si et seulement si on peut choisir les $f_i$ tels que $N_\nabla(f_i)=0$ et la connexion est alors triviale.
\end{proof}

\begin{prop}
Soit $V$ une représentation $p$-adique de dimension $d$. Alors $V$ est de $E$-hauteur finie si et seulement si $D_{\tau,\mathrm{rig}}^+(V)$ est unipotent de rang $d$.
\end{prop}
\begin{proof}
Le théorème \ref{theo Gao} montre que $V$ est de $E$-hauteur finie si, et seulement si, il existe $n \geq 0$ tel que $V_{|\G_{K_n}}$ est semi-stable à poids négatifs. En particulier, la preuve sera presque identique à celle de la proposition \ref{carac cristal unipotent}. 

On a vu à la proposition \ref{retrouver Dst Dcris phitau} que $V$ est une représentation semi-stable de $G_{K_n}$ à poids négatifs si et seulement si $D_{\tau,\log}^+(V)$ contient $d$ éléments linéairement indépendant fixes sous $\tau^{p^n}$ qui forment alors une base de $D_{\mathrm{st}}(V_{|\G_{K_n}})$. En particulier, $N_\nabla$ est triviale sur le $\B_{\tau,\mathrm{rig},K}^\dagger$-module qu'ils engendrent. Si on note $(e_0,\cdots,e_{d_1})$ ces éléments, il reste à montrer qu'ils satisfont 
$$\bigoplus e_i\B_{\tau,\log,K}^+[1/\lambda] = D \otimes_{\B_{\tau,\mathrm{rig},K}^+}\B_{\tau,\log,K}^+[1/\lambda].$$

La remarque \ref{rema isocomparaisonst} montre que 
$$\B_{\tau,\log,K,\infty}^+ \otimes_{\B_{\tau,K}^+}D_{\tau}^+(V) =  \B_{\tau,\log,K,\infty}^+ \otimes_F D_{\mathrm{st}}(V),$$
et donc en particulier, il existe $n \geq 0$ tel que 
$$D_{\tau,\log}^+(V) \subset \B_{\tau,\log,K,n}^+ \otimes_F D_{\mathrm{st}}(V),$$
et donc
$$\phi^n(D_{\tau,\log}^+(V)) \subset \B_{\tau,\log,K}^+ \otimes_F D_{\mathrm{st}}(V).$$

Mais comme $V$ est de $E$-hauteur finie et que $\phi(\lambda) = \frac{E(0)}{E(u)}\lambda$, on en déduit que 
$$\B_{\tau,\log,K}^+ \otimes_{\phi^n(\B_{\tau,\log,K}^+)}  \phi^n(D_{\tau,\log}^+(V)[\frac{1}{\lambda}]) =  D_{\tau,\log}^+(V)[\frac{1}{\lambda}].$$
On a donc un isomorphisme de comparaison 
$$\B_{\tau,\log,K}^+[1/\lambda] \otimes_{\B_{\tau,K}^{\dagger}}D_{\tau}^+(V) =  \B_{\tau,\log,K}^+[1/\lambda] \otimes_F D_{\mathrm{st}}(V).$$

Comme les $e_i$ forment une base de $D_{\mathrm{st}}(V)$, cela montre qu'ils satisfont la condition 
$$\bigoplus e_i\B_{\tau,\log,K}^+[1/\lambda] = D \otimes_{\B_{\tau,\mathrm{rig},K}^+}\B_{\tau,\log,K}^+[1/\lambda].$$
Le deuxième point de la proposition \ref{conditions equiv unipotent} fournit alors un sous $(\phi,N_\nabla)$-module unipotent de $D_{\tau,\mathrm{rig},K}^+(V)$.

Réciproquement, si $D_{\tau,\mathrm{rig}}^+(V)$ est unipotent, alors les $e_0,\cdots,e_{d-1}$ engendrent un $F$-espace vectoriel sur lequel $\log(\tau)$ agit trivialement, de sorte que les $e_i$ sont stables par $\tau^{p^n}$ pour $n$ assez grand et forment alors une base de $D_{\mathrm{st}}(V_{|\G_{K_n}})$, et ces éléments sont dans $\B_{\tau,\log,K}^+ \otimes_{\B_{\tau,K}^+}D_{\tau}^+(V) \subset \Bt_{\log}^+ \otimes V$, de sorte que $V$ est à poids négatifs. On en déduit donc que $V$ est de $E$-hauteur finie par le théorème \ref{theo Gao}.
\end{proof}

\subsection{Passage du $(\phi,\tau)$-module d'une représentation à son $(\phi,\tau')$-module}
Soit $K_\infty/K$ une extension de Kummer déterminée par le choix d'une uniformisante $\pi \in \O_K$ et d'un élément $\tilde{\pi} \in \Etplus$, et soit $K_\infty'$ une autre extension de Kummer, associée à une uniformisante $\pi'$ (on n'exclut pas le choix $\pi=\pi'$) et un élément $\tilde{\pi}' \in \Etplus$. On note aussi $L' = K_\infty \cdot K_\infty'$. On va dans cette section construire un anneau $\B_{\tau,\tau',K}$ permettant de faire le lien entre le $(\phi,\tau)$-module et le $(\phi,\tau')$-module attachés à une représentation $V$. 

Soit $\E_{\tau,\tau',K,0}=k(\!(\tilde{\pi},\tilde{\pi}')\!) \subset \Et$, et soit $\E_{\tau,\tau'}$ la clôture séparable de $\E_{\tau,\tau',K,0}$ dans $\Et$. On note $\A_{\tau,\tau',K,0}$ le complété $p$-adique du localisé de $\O_F[\![ [\tilde{\pi}],[\tilde{\pi}'] ]\!]$, qui est donc un anneau de Cohen pour $\E_{\tau,\tau',K,0}$, et on pose $\B_{\tau,\tau',K,0}=\A_{\tau,\tau',K,0}[1/p]$. On définit $\A_{\tau,\tau'}$ comme le complété $p$-adique de l'unique extension étale infinie de $\A_{\tau,\tau',K,0}$ incluse dans $\At$ et dont le corps résiduel s'identifie à $\E_{\tau,\tau'}$. On pose alors $\A_{\tau,\tau',K} = \A_{\tau,\tau'}^{\Gal(\overline{K}/L'}$ et $\B_{\tau,\tau',K}=\A_{\tau,\tau',K}[1/p]$.

\begin{prop}
Soit $V$ une représentation $p$-adique de $\G_K$. Alors on a 
$$\B_{\tau,\tau',K} \otimes_{\B_{\tau,K}}D(V) = \B_{\tau,\tau',K} \otimes_{\B_{\tau',K}}D'(V).$$
De plus, on peut récupérer $D(V)$ en fonction de $D'(V)$ \textit{via} la formule
$$D(V) = (\B_{\tau,\tau',K} \otimes_{\B_{\tau',K}}D'(V))^{H_{\tau,K}}.$$
\end{prop}
\begin{proof}
Par la théorie des $(\phi,\tau)$-modules (voir \cite[Lemm. 1.20]{Car12}), on a 
$$\B_\tau \otimes_{\Qp}V \simeq \B_\tau \otimes_{\B_{\tau,K}}D_\tau(V),$$
de sorte qu'en tensorisant par $\B_{\tau,\tau'}$, on obtient
$$\B_{\tau,\tau',K} \otimes_{\B_{\tau,K}}D(V) = \B_{\tau,\tau',K} \otimes_{\Qp}V = \B_{\tau,\tau',K} \otimes_{\B_{\tau',K}}D'.$$
La deuxième partie de la proposition s'obtient en prenant les invariants sous $H_{\tau,K}$.
\end{proof}

Cette méthode pour récupérer le $(\phi,\tau)$-module $D(V)$ à partir de $D'(V)$ est loin d'être explicite, notamment parce que l'anneau $\B_{\tau,\tau',K}$ ne possède pas de description satisfaisante. Si on considère les modules définis sur les anneaux de Robba $D_{\mathrm{rig}}^\dagger$ et $D_{\mathrm{rig}}'^{\dagger}$, une stratégie analogue à celle de \cite{GP18} devrait pouvoir montrer que les deux modules deviennent les mêmes en tant que $(\Bt_{\mathrm{rig},L'}^\dagger)^{\pa}$, et que les anneaux $(\Bt_{\mathrm{rig},L'}^I)^{\la}$ sont les perfectisés d'anneaux de séries en deux variables qui correspondent à des extensions algébriques de $\B_{\tau,\tau',K,0}^\dagger$. Il n'existe cependant pas, à la connaissance de l'auteur, de résultats analogues à ceux de Kedlaya dans \cite{slopes} pour des anneaux de Robba en plusieurs variables qui permettraient de descendre de modules sur ces extensions algébriques à des modules sur $\B_{\tau,\tau',\mathrm{rig},K,0}^\dagger$, ni de modules sur $\B_{\tau,\tau',\mathrm{rig},K,0}^\dagger$ à des modules sur $\B_{\tau,\tau',K,0}^\dagger$. Avec de tels résultats, on pourrait peut-être montrer que les deux modules deviennent les mêmes sur $\B_{\tau,\tau',K,0}^\dagger$.

Il semble néanmoins compliqué dans le cas général de donner une méthode explicite permettant de passer de $D(V)$ à $D'(V)$, puisque comme le montre la proposition suivante, dûe à Gee et Liu, une telle méthode ne peut pas respecter la notion de $E$-hauteur~:

\begin{prop}
Soit $V$ une représentation qui est de $E$-hauteur finie pour tout $E$. Alors $V$ est semi-stable.
\end{prop}
\begin{proof}
Voir \cite[Thm. F.11]{emertongeestacks}.
\end{proof}

On signale que, le cas où $V$ est semi-stable à poids négatifs, Liu a construit (voir \cite[Thm. 2.2.1]{liudifferentuniformizers}) un certain sous-anneau de $\Atplus$ permettant de faire le lien entre $D^+(V)$ et $D'^+(V)$.

\bibliographystyle{amsalpha}
\bibliography{bibli}

\newcommand{\etalchar}[1]{$^{#1}$}
\providecommand{\bysame}{\leavevmode\hbox to3em{\hrulefill}\thinspace}
\providecommand{\MR}{\relax\ifhmode\unskip\space\fi MR }
\providecommand{\MRhref}[2]{%
  \href{http://www.ams.org/mathscinet-getitem?mr=#1}{#2}
}
\providecommand{\href}[2]{#2}
\begin{thebibliography}{Fon94b}

\bibitem[And02]{andre2002}
Yves Andr{\'e}, \emph{Filtrations de type {H}asse-{A}rf et monodromie
  $p$-adique}, Inventiones mathematicae \textbf{148} (2002), no.~2, 285--317.

\bibitem[BC16]{Ber14SenLa}
Laurent Berger and Pierre Colmez, \emph{Th\'eorie de {S}en et vecteurs
  localement analytiques}, Ann. Sci. \'Ec. Norm. Sup\'er. (4) \textbf{49}
  (2016), no.~4, 947--970.

\bibitem[Ber02]{Ber02}
Laurent Berger, \emph{Représentations p-adiques et {\'e}quations
  différentielles}, Inventiones mathematicae \textbf{148} (2002), no.~2,
  219--284.

\bibitem[Ber08]{Ber08}
\bysame, \emph{{\'E}quations differentielles $p$-adiques et ($\varphi$,
  {$N$})-modules filtr{\'e}s}, Ast{\'e}risque \textbf{319} (2008), 13--38.

\bibitem[Ber13]{berger2012multivariable}
\bysame, \emph{Multivariable {L}ubin-{T}ate {$(\varphi,\Gamma)$}-modules and
  filtered {$\varphi$}-modules}, Math. Res. Lett. \textbf{20} (2013), no.~3,
  409--428.

\bibitem[Ber16]{Ber14MultiLa}
\bysame, \emph{Multivariable {$(\varphi,\Gamma)$}-modules and locally analytic
  vectors}, Duke Math. J. \textbf{165} (2016), no.~18, 3567--3595.

\bibitem[Bre98]{breuil1998schemas}
Christophe Breuil, \emph{Sch{\'e}mas en groupes et corps de normes}, 1998.

\bibitem[Car13]{Car12}
Xavier Caruso, \emph{Repr{\'e}sentations galoisiennes $p$-adiques et
  $(\varphi,\tau)$-modules}, Duke Mathematical Journal \textbf{162} (2013),
  no.~13, 2525--2607.

\bibitem[CC98]{cherbonnier1998representations}
Fr{\'e}d{\'e}ric Cherbonnier and Pierre Colmez, \emph{Repr{\'e}sentations
  p-adiques surconvergentes}, Inventiones mathematicae \textbf{133} (1998),
  no.~3, 581--611.

\bibitem[CC99]{cherbcoliwa}
\bysame, \emph{Th{\'e}orie d’{I}wasawa des repr{\'e}sentations $p$-adiques
  d’un corps local}, Journal of the American Mathematical Society \textbf{12}
  (1999), no.~1, 241--268.

\bibitem[CF00]{CF00}
Pierre Colmez and Jean-Marc Fontaine, \emph{Construction des repr\'esentations
  {$p$}-adiques semi-stables}, Invent. Math. \textbf{140} (2000), no.~1, 1--43.
  \MR{1779803}

\bibitem[Che96]{Che96}
Frédéric Cherbonnier, \emph{Representations $p$-adiques surconvergentes},
  Ph.D. thesis, 1996.

\bibitem[Col]{Col01}
Pierre Colmez, \emph{Les conjectures de monodromie $p$-adiques}, S{\'e}minaire
  Bourbaki \textbf{44}, no.~2001-2002, 2001--2002.

\bibitem[Col99]{colmez1999representations}
\bysame, \emph{Repr{\'e}sentations cristallines et repr{\'e}sentations de
  hauteur finie}, Journal fur die Reine und Angewandte Mathematik (1999),
  119--143.

\bibitem[Col02]{Col02}
\bysame, \emph{Espaces de {B}anach de dimension finie}, Journal of the
  Institute of Mathematics of Jussieu \textbf{1} (2002), no.~3, 331--439.

\bibitem[EG19]{emertongeestacks}
Matthew Emerton and Toby Gee, \emph{Moduli stacks of {\'e}tale
  {$(\varphi,\Gamma)$}-modules and the existence of crystalline lifts}, arXiv
  preprint arXiv:1908.07185 (2019).

\bibitem[Eme17]{emerton2004locally}
Matthew Emerton, \emph{Locally analytic vectors in representations of locally
  {$p$}-adic analytic groups}, Mem. Amer. Math. Soc. \textbf{248} (2017),
  no.~1175, iv+158.

\bibitem[Fon90]{Fon90}
Jean-Marc Fontaine, \emph{Repr{\'e}sentations p-adiques des corps locaux
  (1{\`e}re partie)}, The Grothendieck Festschrift, Springer, 1990,
  pp.~249--309.

\bibitem[Fon94a]{fontaine1994corps}
\bysame, \emph{Le corps des p{\'e}riodes $p$-adiques}, Ast{\'e}risque (1994),
  no.~223, 59--102.

\bibitem[Fon94b]{fontaine1994representations}
\bysame, \emph{Repr{\'e}sentations $p$-adiques semi-stables}, Ast{\'e}risque
  \textbf{223} (1994), 113--184.

\bibitem[Gao19]{GaoEhauteur}
Hui Gao, \emph{Breuil--kisin modules and integral p-adic hodge theory, arxiv
  e-prints (2019)}, arXiv preprint arXiv:1905.08555 (2019).

\bibitem[GL16]{gao2016loose}
Hui Gao and Tong Liu, \emph{Loose crystalline lifts and overconvergence of
  étale $(\varphi,\tau)$-modules}, arXiv preprint arXiv:1606.07216 (2016).

\bibitem[GP18]{GP18}
Hui Gao and L{\'e}o Poyeton, \emph{Locally analytic vectors and overconvergent
  $(\varphi,\tau)$-modules}, à paraître,J. Inst. Math. Jussieu (2018).

\bibitem[Ked04]{kedlayamonodromie}
Kiran~S Kedlaya, \emph{A $p$-adic local monodromy theorem}, Annals of
  mathematics (2004), 93--184.

\bibitem[Ked05]{slopes}
\bysame, \emph{Slope filtrations revisited}, Doc. Math \textbf{10} (2005),
  no.~447525.15.

\bibitem[Kis06]{KisinFiso}
Mark Kisin, \emph{Crystalline representations and {F}-crystals}, Algebraic
  geometry and number theory, Springer, 2006, pp.~459--496.

\bibitem[KR09]{KR09}
Mark Kisin and Wei Ren, \emph{Galois representations and {L}ubin-{T}ate
  groups}, Doc. Math \textbf{14} (2009), 441--461.

\bibitem[Laz62]{lazardfonctions1962}
Michel Lazard, \emph{Les z{\'e}ros des fonctions analytiques d’une variable
  sur un corps valu{\'e} complet}, Publications Math{\'e}matiques de l'Institut
  des Hautes {\'E}tudes Scientifiques \textbf{14} (1962), no.~1, 47--75.

\bibitem[Liu08]{Liu08}
Tong Liu, \emph{On lattices in semi-stable representations: a proof of a
  conjecture of {B}reuil}, Compositio Mathematica \textbf{144} (2008), no.~1,
  61--88.

\bibitem[Liu10]{Liu10}
\bysame, \emph{A note on lattices in semi-stable representations},
  Mathematische Annalen \textbf{346} (2010), no.~1, 117.

\bibitem[Liu18]{liudifferentuniformizers}
\bysame, \emph{Compatibility of {K}isin modules for different uniformizers},
  Journal f{\"u}r die reine und angewandte Mathematik (Crelles Journal)
  \textbf{2018} (2018), no.~740, 1--24.

\bibitem[M{\etalchar{+}}95]{matsuda1995local}
Shigeki Matsuda et~al., \emph{Local indices of $p$-adic differential operators
  corresponding to {A}rtin-{S}chreier-{W}itt coverings}, Duke Mathematical
  Journal \textbf{77} (1995), no.~3, 607--625.

\bibitem[Meb02]{Meb02}
Zoghman Mebkhout, \emph{Analogue p-adique du th{\'e}oreme de {T}urrittin et le
  th{\'e}oreme de la monodromie $p$-adique}, Inventiones mathematicae
  \textbf{148} (2002), no.~2, 319--351.

\bibitem[Sch11]{schneider2011p}
Peter Schneider, \emph{$p$-{A}dic {L}ie groups}, vol. 344, Springer Science \&
  Business Media, 2011.

\bibitem[ST02a]{schneider2002bis}
Peter Schneider and Jeremy Teitelbaum, \emph{Banach space representations and
  {I}wasawa theory}, Israel journal of mathematics \textbf{127} (2002), no.~1,
  359--380.

\bibitem[ST02b]{schneider2002}
\bysame, \emph{Locally analytic distributions and $p$-adic representation
  theory, with applications to {$GL_2$}}, Journal of the American Mathematical
  Society \textbf{15} (2002), no.~2, 443--468.

\bibitem[Tat67]{Tat67}
John~T Tate, \emph{$p$-divisible groups}, Proceedings of a Conference on Local
  Fields, Springer, 1967, pp.~158--183.

\bibitem[Win83]{Win83}
Jean-Pierre Wintenberger, \emph{Le corps des normes de certaines extensions
  infinies de corps locaux; applications}, Annales scientifiques de l'Ecole
  Normale Sup{\'e}rieure, vol.~16, Soci{\'e}t{\'e} math{\'e}matique de France,
  1983, pp.~59--89.

\end{thebibliography}
\end{document}